\documentclass[AMA,STIX1COL]{WileyNJD-v2}
\usepackage{amsfonts}
\usepackage{graphicx}
\usepackage{mathrsfs}
\usepackage{color}
\usepackage{makecell,multirow,diagbox}
\usepackage{amsfonts}
\usepackage{CJK}
\usepackage{ulem}
\usepackage{cancel}
\usepackage{graphicx}
\usepackage{amsmath,amsxtra,latexsym,amsthm, amscd, pb-diagram}
\usepackage{subfigure}

\newtheorem{thm}{Theorem}[section]
\newtheorem{prop}{Proposition}[section]
\newtheorem{lem}{Lemma}[section]
\newtheorem{coro}{Corollary}[section]
\newtheorem{rem}{Remark}[section]

\newcommand{\ml}{\mathcal}
\newcommand{\mb}{\mathbb}
\DeclareMathOperator{\diag}{diag}
\DeclareMathOperator{\divv}{div}
\DeclareMathOperator{\intt}{int}
\DeclareMathOperator{\extt}{ext}
\DeclareMathOperator{\midd}{mid}
\articletype{RESEARCH ARTICLE}%

\received{xx xx xxxx}
\revised{xx xx xxxx}
\accepted{xx xx xxxx}

\begin{document}

\title{Weakly coupled systems of semi-linear elastic waves with different damping mechanisms in 3D}

\author[1]{Wenhui Chen}

\author[1]{Michael Reissig*}

\authormark{Chen \textsc{et al}}

\address[1]{\orgdiv{Institute of Applied Analysis, Faculty for Mathematics and Computer Science}, \orgname{Technical University Bergakademie Freiberg}, \orgaddress{\state{Pr\"uferstra\ss e 9, 09596, Freiberg}, \country{Germany}}}

\corres{*Correspondence to: Michael Reissig, Institute of Applied Analysis, Faculty for Mathematics and Computer Science, Technical University Bergakademie Freiberg, Pr\"uferstra\ss e. 9, 09596, Freiberg, Germany \email{reissig@math.tu-freiberg.de}}

\presentaddress{Faculty for Mathematics and Computer Science, Technical University Bergakademie Freiberg, Pr\"uferstr. 9 - 09596, Freiberg, Germany.}

\abstract[Summary]{We consider the following Cauchy problem for weakly coupled systems of semi-linear damped elastic waves with a power source non-linearity in three-dimensions:
\begin{equation*}
U_{tt}-a^2\Delta U-\big(b^2-a^2\big)\nabla\divv U+(-\Delta)^{\theta}U_t=F(U),\,\,\,\, (t,x)\in(0,\infty)\times\mb{R}^3,
\end{equation*}
where $U=U(t,x)=\big(U^{(1)}(t,x),U^{(2)}(t,x),U^{(3)}(t,x)\big)^{\mathrm{T}}$ with $b^2>a^2>0$ and $\theta\in[0,1]$. Our interests are some qualitative properties of solutions to the corresponding linear model with vanishing right-hand side and the influence of the value of $\theta$ on the exponents $p_1,p_2,p_3$ in $F(U)=\big(|U^{(3)}|^{p_1},|U^{(1)}|^{p_2},|U^{(2)}|^{p_3}\big)^{\mathrm{T}}$ to get results for the global (in time) existence of small data solutions.
}


\maketitle
\section{Introduction}\label{Introduction}
This paper is devoted to the Cauchy problem for the following weakly coupled systems of semi-linear elastic waves with damping mechanisms $(-\Delta)^{\theta}U_t$ for $(0,\infty)\times \mb{R}^3$:
\begin{equation}\label{fsdew001}
\left\{
\begin{aligned}
&U^{(1)}_{tt}-a^2\Delta U^{(1)}-\big(b^2-a^2\big)\partial_{x_1}\big(\partial_{x_1}U^{(1)}+\partial_{x_2}U^{(2)}+\partial_{x_3}U^{(3)}\big)+(-\Delta)^{\theta}U^{(1)}_t=|U^{(3)}|^{p_1},\\
&U^{(2)}_{tt}-a^2\Delta U^{(2)}-\big(b^2-a^2\big)\partial_{x_2}\big(\partial_{x_1}U^{(1)}+\partial_{x_2}U^{(2)}+\partial_{x_3}U^{(3)}\big)+(-\Delta)^{\theta}U^{(2)}_t=|U^{(1)}|^{p_2},\\
&U^{(3)}_{tt}-a^2\Delta U^{(3)}-\big(b^2-a^2\big)\partial_{x_3}\big(\partial_{x_1}U^{(1)}+\partial_{x_2}U^{(2)}+\partial_{x_3}U^{(3)}\big)+(-\Delta)^{\theta}U^{(3)}_t=|U^{(2)}|^{p_3},\\
&\big(U^{(1)},U^{(2)},U^{(3)}\big)(0,x)=\big(U^{(1)}_0,U^{(2)}_0,U^{(3)}_0\big)(x),\\
&\big(U_t^{(1)},U_t^{(2)},U_t^{(3)}\big)(0,x)=\big(U^{(1)}_1,U^{(2)}_1,U^{(3)}_1\big)(x),
\end{aligned}
\right.
\end{equation}
where two material-dependent quantities $a^2$ and $b^2$ called Lam\'e constants arise in the strain-stress relationship and satisfy the condition $b^2>a^2>0$.
We assume in the above model $\theta\in[0,1]$, where $\theta=0$ appears in the model with \emph{friction or external damping}, $\theta\in(0,1]$ appears in the model with \emph{structural damping}, in particular, $\theta=1$ appears in the model with \emph{viscoelastic damping}.

The present paper is a continuation of the paper \cite{Reissig2016}, in which linear structurally damped elastic waves in 2D with $\theta\in(0,1]$ were studied. The author obtained Gevrey smoothing if $\theta\in(0,1)$, propagation of singularities if $\theta=1$ and estimates of higher-order energies. Recently, by choosing data $\big(U^{(k)}_0,U^{(k)}_1\big)\in(H^{s+1}\cap L^1)\times(H^s\cap L^1)$ for $k=1,\dots,n$, the paper \cite{Ikehata2014} proved almost sharp energy estimates for the corresponding linear model to \eqref{fsdew001} in $\mb{R}^n$, $n\geq2$, with $\theta\in[0,1]$ and vanishing right-hand side.

Let us recall some results for critical exponents in Cauchy problems for semi-linear damped wave models.
For the single semi-linear classical damped wave equation
	\begin{equation}\label{semidampedwave}
	\left\{\begin{aligned}
	&U_{tt}-\Delta U+U_t=|U|^p,&(t,x)\in(0,\infty)\times\mb{R}^n,\\
	&(U,U_t)(0,x)=(U_0,U_1)(x),&x\in\mb{R}^n,
	\end{aligned}\right.
	\end{equation}
	the critical exponent is the Fujita exponent, i.e.,
	\begin{equation*}
	p_{\text{crit}}=p_{\text{crit}}(n)=1+\frac{2}{n}.
	\end{equation*}
	On one hand, the pioneering paper \cite{TodorovaYordanov2001} proved the global (in time) existence of small data solutions for $p>p_{\text{crit}}$ assuming compactly supported data. The assumption for compactly supported data can be relaxed to smallness in some weighted energy spaces \cite{IkehataTanizawa2005} for some dimensions $n\geq1$. On the other hand, \cite{Zhang2001} proved that the critical exponent $p=p_{\text{crit}}$ belongs to the blow-up case.
	Recently, the semi-linear structurally damped wave models
	\begin{equation}\label{semistructuraldampedwave}
	\left\{\begin{aligned}
	&U_{tt}-\Delta U+(-\Delta)^{\theta}U_t=|U|^p,&(t,x)\in(0,\infty)\times\mb{R}^n,\\
	&(U,U_t)(0,x)=(U_0,U_1)(x),&x\in\mb{R}^n,
	\end{aligned}\right.
	\end{equation}
	with $p>1$ and $\theta\in(0,1]$ were studied. The global (in time) existence of small data solutions was investigated in \cite{D'abbiccoReissig2014}. The authors proved
the existence of unique global (in time) Sobolev solutions for some low dimensions $n\geq2$ if we assume that the exponent $p$ satisfies some conditions related to the application of the Gagliardo-Nirenberg inequality and
	\begin{equation*}
	p>\left\{\begin{aligned}
	&1+\frac{2}{n-2\theta} &\text{if}\,\,\,\,\theta\in\left(0,1/2\right],\\
	&1+\frac{1+2\theta}{n-1} &\text{if}\,\,\,\,\theta\in\left(1/2,1\right].
	\end{aligned}\right.
	\end{equation*}
	Additionally, by using some suitable $L^m$-$L^q$ estimates, \cite{DabbiccoEbert201401} studied the global existence of solutions to \eqref{semistructuraldampedwave} with suitably small data when
	\begin{equation*}
	p>1+\frac{2}{n-2\theta}\,\,\,\,\text{if}\,\,\,\, \theta\in\left(0,1/2\right],
	\end{equation*}
	in some high space dimensions $n\geq3$. The papers \cite{D'abbiccoReissig2014} and \cite{DabbiccoEbert201401} proved the critical exponent
\begin{equation*}
	p_{\text{crit}}=p_{\text{crit}}(n)=1+\frac{2}{n-1}
	\end{equation*}
	for $n\geq2$ to the model \eqref{semistructuraldampedwave} with $\theta=1/2$ by the test function method.

Now let us turn to weakly coupled systems. Firstly, we consider the weakly coupled system of semi-linear classical damped waves
	\begin{equation}\label{weaklycoupleddamped}
	\left\{\begin{aligned}
	&U_{tt}-\Delta U+U_t=|V|^{p},&(t,x)\in(0,\infty)\times\mb{R}^n,\\
	&V_{tt}-\Delta V+V_t=|U|^{q},&(t,x)\in(0,\infty)\times\mb{R}^n,\\
	&(U,U_t,V,V_t)(0,x)=(U_0,U_1,V_0,V_1)(x),&x\in\mb{R}^n,
	\end{aligned}\right.
	\end{equation}
	where $p,q>1$ and $U=U(t,x)$, $V=V(t,x)$ are real-valued unknown functions. Critical exponents to the system \eqref{weaklycoupleddamped} are described by the condition
	\begin{equation*}
	\alpha_{\max}=\max\left\{\frac{p+1}{pq-1};\frac{q+1}{pq-1}\right\}=\frac{n}{2}.
	\end{equation*}
In \cite{SunWang2007} the authors investigated for $n=1,3$, that if $\alpha_{\max}<n/2$, then there exists a unique global (in time) Sobolev solution for small data. If $\alpha_{\max}\geq n/2$, then local (in time) solutions, in general, blow up in finite time. The paper \cite{Narazaki2009} generalized their existence results to $n=1,2,3$ and improved the time decay estimates when $n=3$. The recent paper \cite{NishiharaWakasugi} determined the critical exponents for any space dimension $n$. The proof of the global (in time) existence of energy solutions is based on the weighted energy method.\\
	Later \cite{Takeda2009} considered the following generalization of the model \eqref{weaklycoupleddamped}:
	\begin{equation}\label{weaklycoupleddampedgeneral}
	\left\{\begin{aligned}
	&U^{(1)}_{tt}-\Delta U^{(1)}+U^{(1)}_t=|U^{(k)}|^{p_1},\\
	&U^{(2)}_{tt}-\Delta U^{(2)}+U^{(2)}_t=|U^{(1)}|^{p_2},\\
	&\quad\vdots\\
	&U^{(k)}_{tt}-\Delta U^{(k)}+U^{(k)}_t=|U^{(k-1)}|^{p_k},\\
	\end{aligned}\right.
	\end{equation}
	where $k\geq2$ and $p_j>1$ for $j=1,\dots,k$. We define the matrix $P$ as
	\[
	\begin{split}
	P=\left(
	{\begin{array}{*{20}c}
		0 & 0 & \cdots & 0 & p_1\\
		p_2 & 0 & \cdots & 0 & 0\\
		0 & p_3 & \cdots & 0 & 0\\
		\vdots & \vdots & \ddots & \vdots & \vdots\\
		0 & 0 & \cdots & p_k & 0 \\
		\end{array}}
	\right)
	\end{split}
	\]
	and consider $P-I$, where $I$ is the identity matrix. Therefore, it is clear that
	\begin{equation*}
	|P-I|=(-1)^{k-1}\Big(\prod\limits_{j=1}^kp_j-1\Big),
	\end{equation*}
	and the inverse matrix of $P-I$ exists because $|P-I|\neq0$. Then, we can define
	\begin{equation*}
	\alpha=(\alpha_1,\dots,\alpha_k)=(P-I)^{-1}\cdot(1,\dots,1)^{\mathrm{T}}.
	\end{equation*}
	The author of \cite{Takeda2009} proved that when $n\leq3$, then the critical exponents of the system \eqref{weaklycoupleddampedgeneral} are described by the condition
	\begin{equation*}
	\alpha_{\max}=\max\left\{\alpha_1;\dots;\alpha_k\right\}=\frac{n}{2}.
	\end{equation*}
	In addition, the author obtained blow-up results for any space dimensions. Then, the paper \cite{NishiharaWakasugi1} determined the critical exponents for any space dimension $n$, where the weighted energy method has been used in the proof of global (in time) existence of solutions.\\
	Lastly, we consider the weakly coupled system of semi-linear structurally damped wave equations
	\begin{equation}\label{weaklycoupledstrucuraldamped}
	\left\{\begin{aligned}
	&U_{tt}-\Delta U+2(-\Delta)^{1/2}U_t=|V|^{p},&(t,x)\in(0,\infty)\times\mb{R}^n,\\
	&V_{tt}-\Delta V+2(-\Delta)^{1/2}V_t=|U|^{q},&(t,x)\in(0,\infty)\times\mb{R}^n,\\
	&(U,U_t,V,V_t)(0,x)=(U_0,U_1,V_0,V_1)(x),&x\in\mb{R}^n.
	\end{aligned}\right.
	\end{equation}
	The recent paper \cite{D'abbicco2015} showed that the global (in time) existence of small data Sobolev solutions to the system \eqref{weaklycoupledstrucuraldamped} holds if
	\begin{equation*}
	\alpha_{\max}=\max\left\{\frac{p+1}{pq-1};\frac{q+1}{pq-1}\right\}<\frac{n-1}{2},
	\end{equation*}
	and $n \geq 2$. On the contrary, the nonexistence result for global (in time) solutions holds if $\alpha_{\max}>(n-1)/2$ for $n\geq1$.

In this paper, we generalize the weakly coupled system of the semi-linear damped wave models \eqref{weaklycoupleddampedgeneral} and \eqref{weaklycoupledstrucuraldamped} to weakly coupled systems of semi-linear damped elastic waves \eqref{fsdew001}.\medskip

The paper is organized as follows: In Section \ref{EstimateforthelinearCauchyproblem}, we prepare the asymptotic behavior and some qualitative properties, including well-posedness and smoothing effect, of solutions to the corresponding linear Cauchy problem with vanishing right-hand side.  In Section \ref{energyestimates}, we prove suitable energy estimates by phase space analysis and energy methods in the Fourier space. In Section \ref{diffusionphenomenon}, diffusion phenomena for linear elastic waves with friction or structural damping are studied. In Section \ref{ge}, the global (in time) existence of small data solutions to \eqref{fsdew001} are treated. In Section \ref{Concludingremark}, some concluding remarks complete the paper. \medskip

We provide some notations used in this paper. Let $\chi_{\intt},\chi_{\midd},\chi_{\extt}\in \mathcal{C}^{\infty}$ having their supports in $Z_{\intt}(\varepsilon):=\{|\xi|<\varepsilon\}$, $Z_{\midd}(\varepsilon):=\{\varepsilon\leq|\xi|\leq1/\varepsilon\}$ and $Z_{\extt}(\varepsilon):=\{|\xi|>1/\varepsilon\}$, respectively, so that $\chi_{\midd}=1-\chi_{\intt}-\chi_{\extt}$. Here $\varepsilon>0$ is a sufficiently small constant.

In addition, the symbol $\oplus$ between Jordan matrices $J_{l_j}(\lambda_j)$ is used as follows:
\[
\begin{split}
J_{l_1}(\lambda_1)\oplus J_{l_2}(\lambda_2)\oplus\cdots\oplus J_{l_n}(\lambda_n):=\left(
{\begin{array}{*{20}c}
	J_{l_1}(\lambda_1) &  & &\\
	& J_{l_2}(\lambda_2) &  &\\
	& & & \ddots&\\
	& & & &J_{l_{n}}(\lambda_n)
	\end{array}}
\right),\,\,\,\,
\text{where}\,\,\,\,
J_{l_j}(\lambda_j):=\left(
{\begin{array}{*{20}c}
	\lambda_j & 1 &  & &\\
	& \lambda_j & 1 & &\\
	& & \ddots&\ddots &\\
	& &  &\lambda_j & 1\\
	& & & & \lambda_j
	\end{array}}
\right).
\end{split}
\]

For the sake of clarity, we introduce for any $s\geq0$ and $m\in[1,2]$ the spaces
\begin{equation*}
\begin{split}
\ml{D}^s_{m,1}:&=(H^{s+1}\cap L^m)\times(H^s\cap L^m),\\
\ml{D}^s_{m,2}:&=(|D|^{-1}H^{s}\cap \dot{H}^{1}_{m})\times(H^s\cap L^m),
\end{split}
\end{equation*}
carrying the corresponding norms:
\begin{equation*}
\begin{split}
\big\|\big(U^{(k)}_0,U^{(k)}_1\big)\big\|_{\ml{D}^s_{m,1}}:=&\|U^{(k)}_0\|_{H^{s+1}}+\|U^{(k)}_0\|_{L^m}+\|U^{(k)}_1\|_{H^s}+\|U^{(k)}_1\|_{L^m},\\
\big\|\big(U^{(k)}_0,U^{(k)}_1\big)\big\|_{\ml{D}^s_{m,2}}:=&\|U^{(k)}_0\|_{|D|^{-1}H^{s}}+\|U^{(k)}_0\|_{\dot{H}^{1}_{m}}+\|U^{(k)}_1\|_{H^s}+\|U^{(k)}_1\|_{L^m}.
\end{split}
\end{equation*}
Moreover, we remark that $\ml{D}_{2,1}^0=H^1\times L^2$ and $\ml{D}_{2,2}^0=\dot{H}^1\times L^2$ are classical energy spaces. Here $H^{s}_{m}$ and $\dot{H}^{s}_{m}$ denote Bessel and Riesz potential spaces based on $L^m$, respectively, and $|D|^{s}$ stands for the pseudo-differential operator with the symbol $|\xi|^{s}$.\\
\noindent  By $|D|^{-1}H^s\big(\mb{R}^n\big)$, $s\in\mb{R}$, we denote the class of all distributions $f$ from $\ml{Z}'\big(\mb{R}^n\big)$ such that
\begin{equation*}
|D|^{-1}H^s\big(\mb{R}^n\big):=\Big\{f\in\ml{Z}'\big(\mb{R}^n\big):\|f\|_{|D|^{-1}H^s}:=\Big(\int\nolimits_{\mb{R}^n}|\xi|^2\langle \xi\rangle^{2s}|\hat{f}(\xi)|^2d\xi\Big)^{1/2}<\infty\Big\},
\end{equation*}
where $\ml{Z}'\big(\mb{R}^n\big)$ denotes the topological dual space to the subspace of the Schwartz space $\ml{S}\big(\mb{R}^n\big)$ consisting of functions with $d_{\xi}^j\hat{f}(0)=0$ for all $j\in\mb{N}$. In other words, we can identify $\ml{Z}'$ with the factor space $\ml{S}'/\ml{P}$. Here $\ml{P}$ is the space of all polynomials \cite{RunstSickel1996}. We can discuss some properties of the distribution $f\in|D|^{-1}H^s$ in two ways. On the one hand we may use some properties of the Bessel potential space $H^s$ because $|D|f\in H^s$, on the other hand we may use some properties of the Riesz potential space $\dot{H}^1$ because $\langle D\rangle^s f\in \dot{H}^1$.

Hereinafter, we denote the Fourier and inverse Fourier transforms by $\ml{F}$ and $\ml{F}^{-1}$. We write $f\lesssim g$, when there exists a constant $C>0$ such that $f\leq Cg$. Also, we denote $\lceil x\rceil:=\min\left\{x\in\mb{Z}:x\leq C\right\}$ the ceiling function for $x$. The identity matrix is denoted by $I$.
\section{Estimates for the solutions to the linear Cauchy problem}\label{EstimateforthelinearCauchyproblem}
To study global (in time) solutions to the semi-linear model \eqref{fsdew001} our starting point is to study the corresponding Cauchy problem for linear elastic waves with friction or structural damping
\begin{equation}\label{linearproblem}
\left\{
\begin{aligned}
&u_{tt}-a^2\Delta u-\big(b^2-a^2\big)\nabla\divv u+(-\Delta)^{\theta}u_t=0,\quad &(t,x)\in(0,\infty)\times\mb{R}^3,\\
&(u,u_t)(0,x)=(u_0,u_1)(x),\quad &x\in\mb{R}^3.
\end{aligned}
\right.
\end{equation}
Any solution $u=u(t,x)$ to \eqref{linearproblem} corresponds to a solution $U=U(t,x)$ to \eqref{fsdew001} with vanishing right-hand side. To obtain the asymptotic behavior of solutions diagonalization schemes and the energy method in the Fourier space are available.
\subsection{Asymptotic behavior of solutions}\label{Asymptoticbehaviorofsolutions}
We introduce the corresponding system to \eqref{linearproblem} by the aid of the partial Fourier transform $\hat{u}(t,\xi)=\mathcal{F}_{x\rightarrow\xi}(u)(t,\xi)$, that is,
\begin{equation*}
\left\{
\begin{aligned}
&\hat{u}_{tt}+|\xi|^{2\theta}\hat{u}_t+|\xi|^2A(\eta)\hat{u}=0,\quad &(t,\xi)\in(0,\infty)\times\mb{R}^3,\\
&(\hat{u},\hat{u}_t)(0,\xi)=(\hat{u}_0,\hat{u}_1)(\xi),\quad&\xi\in\mb{R}^3,
\end{aligned}\right.
\end{equation*}
where $\eta=\xi/|\xi|\in\mb{S}^1$ and
\[
\begin{split}
A(\eta)=\left(
{\begin{array}{*{20}c}
	a^2+\big(b^2-a^2\big)\eta_1^2 & \big(b^2-a^2\big)\eta_1\eta_2 & \big(b^2-a^2\big)\eta_1\eta_3\\
	\big(b^2-a^2\big)\eta_1\eta_2 & a^2+(b^2-a^2)\eta_2^2 & \big(b^2-a^2\big)\eta_2\eta_3\\
	\big(b^2-a^2\big)\eta_1\eta_3 & \big(b^2-a^2\big)\eta_2\eta_3 & a^2+\big(b^2-a^2\big)\eta_3^2\\
	\end{array}}
\right).
\end{split}
\]
Our assumption $b^2>a^2$ implies that $A(\eta)$ is positive definite. The eigenvalues of $A(\eta)$ are $a^2$, $a^2$ and $b^2$. For further approach, we introduce
\[
\begin{split}
M(\eta):=\left(
{\begin{array}{*{20}c}
	-\eta_2/\eta_1 & -\eta_3/\eta_1 & \eta_1/\eta_3\\
	1 & 0 & \eta_2/\eta_3\\
	0 & 1 & 1\\
	\end{array}}
\right) \,\,\,\,\text{and}\,\,\,\, A_{\text{diag}}(|\xi|)=|\xi|^2\text{diag}\big(a^2,a^2,b^2\big).
\end{split}
\]
After the change of variables $v(t,\xi)=M^{-1}(\eta)\hat{u}(t,\xi)$ we obtain
\begin{equation*}
D_t^2v-i|\xi|^{2\theta}D_tv-A_{\text{diag}}(|\xi|)v=0,\quad (v,D_tv)(0,\xi)=(v_0,-iv_1)(\xi),
\end{equation*}
where $v_j(\xi)=M^{-1}(\eta)\hat{u}_j(\xi)$ for $j=0,1$. After introducing the micro-energy
\begin{equation}\label{trans01}
W^{(0)}(t,\xi)=\Big(D_tv(t,\xi)+A_{\text{diag}}^{1/2}(|\xi|)v(t,\xi),D_tv(t,\xi)-A_{\text{diag}}^{1/2}(|\xi|)v(t,\xi)\Big)^\mathrm{T},
\end{equation}
we conclude the first-order system
\begin{equation}\label{weshould}
D_tW^{(0)}-\frac{i}{2}|\xi|^{2\theta}B_0W^{(0)}-|\xi|B_1W^{(0)}=0,\,\,\,\, W^{(0)}(0,\xi)=W_0^{(0)}(\xi),
\end{equation}
where \[ W_0^{(0)}(\xi)=\Big(-iv_1(\xi)+A_{\text{diag}}^{1/2}(|\xi|)v_0(\xi),-iv_1(\xi)-A_{\text{diag}}^{1/2}(|\xi|)v_0(\xi)\Big)^\mathrm{T}\] and the matrices
$B_0$ and $B_1$ are defined as follows: \[
\begin{split}
B_0=\left(
{\begin{array}{*{20}c}
	1 & 0 & 0 & 1 & 0 & 0\\
	0 & 1 & 0 & 0 & 1 & 0\\
	0 & 0 & 1 & 0 & 0 & 1\\
	1 & 0 & 0 & 1 & 0 & 0\\
	0 & 1 & 0 & 0 & 1 & 0\\
	0 & 0 & 1 & 0 & 0 & 1\\
	\end{array}}
\right)\,\,\,\,\text{and} \,\,\,\,B_1=\text{diag}\Big(\sqrt{a^2},\sqrt{a^2},\sqrt{b^2},-\sqrt{a^2},-\sqrt{a^2},-\sqrt{b^2}\Big).
\end{split}
\]
In order to understand the influence of the parameter $|\xi|$ on the asymptotic behavior of solutions we distinguish between the following cases:

\emph{Case 2.0}: $\quad$ $\theta\neq 1/2$ with $\xi\in Z_{\midd}(\varepsilon)$;

\emph{Case 2.1}: $\quad$ $\theta\in\left[0,1/2\right)$ with $\xi\in Z_{\intt}(\varepsilon)$ or $\theta\in\left(1/2,1\right]$ with $\xi\in Z_{\extt}(\varepsilon)$;

\emph{Case 2.2}: $\quad$ $\theta\in\left[0,1/2\right)$ with $\xi\in Z_{\extt}(\varepsilon)$ or $\theta\in\left(1/2,1\right]$ with $\xi\in Z_{\intt}(\varepsilon)$;

\emph{Case 2.3}: $\quad$ $\theta=1/2$ for all frequencies.
\subsubsection{Energy methods for bounded frequencies away from zero for Case 2.0}
Our goal is to prove an exponential decay result for a suitable energy in the middle zone $Z_{\midd}(\varepsilon)$ by the energy method in the phase space. We are interested in the decoupled system
\begin{equation}\label{20001}
\left\{\begin{aligned}
&v_{tt}+|\xi|^{2\theta}v_t+A_{\text{diag}}(|\xi|)v=0, &\quad(t,\xi)\in(0,\infty)\times\mb{R}^3,\\
&(v,v_t)(0,\xi)=(v_0,v_1)(\xi),&\quad\xi\in\mb{R}^3,
\end{aligned}\right.
\end{equation}
with bounded frequencies $|\xi|\in\big[\varepsilon,1/\varepsilon\big]$.
\begin{thm}\label{middlezone}
	Consider frequencies in the middle zone $Z_{\midd}(\varepsilon)$. The solution $v=v(t,\xi)$ to the Cauchy problem \eqref{20001} satisfies for all $t>0$ the estimate
	\begin{equation*}
	|v^{(k)}(t,\xi)|^2+|v_t^{(k)}(t,\xi)|^2+|\xi|^2|v^{(k)}(t,\xi)|^2\lesssim e^{-ct}\sum\limits_{k=1}^3\big(|v^{(k)}_0(\xi)|^2+|v^{(k)}_1(\xi)|^2\big),
	\end{equation*}
	where $k=1,2,3$ and $c$ is a positive constant.
\end{thm}
\begin{proof} Defining $\varpi_k=a^2\text{ for }k=1,2$, and $\varpi_k=b^2 \text{ for }k=3$, we introduce the energy $E_{\midd}(v)=E_{\midd}(v)(t,\xi)$ in the middle zone $Z_{\midd}(\varepsilon)$ according to the structure of the system \eqref{20001} as follows:
	\begin{equation*}
	\begin{split}
	E_{\midd}(v)(t,\xi):=\frac{1}{2}\sum\limits_{k =1}^3\big(|v^{(k)}_t(t,\xi)|^2+\varpi_k|\xi|^2|v^{(k)}(t,\xi)|^2\big).
	\end{split}
	\end{equation*}
	Multiplying \eqref{20001} by $\bar{v}_t=\bar{v}_t(t,\xi)$ we obtain
	\begin{equation}\label{20003}
	\frac{\partial}{\partial t} E_{\midd}(v)(t,\xi)=-|\xi|^{2\theta}\sum\limits_{k=1}^3|v^{(k)}_t(t,\xi)|^2.
	\end{equation}
	Homoplastically, we multiply \eqref{20001} by $\bar{v}=\bar{v}(t,\xi)$ and take the real part of the resulting identity to arrive at
	\begin{equation}\label{20004}
	\begin{split}
	\frac{1}{2}\sum\limits_{k=1}^3\varpi_k|\xi|^2|v^{(k)}(t,\xi)|^2
	+\frac{\partial}{\partial t}\Big(\sum\limits_{k=1}^3\Re\big(\bar{v}^{(k)}(t,\xi)v_t^{(k)}(t,\xi)\big)\Big)\leq c_0\sum\limits_{k=1}^3|v_t^{(k)}(t,\xi)|^2,
	\end{split}
	\end{equation}
	where $c_0=1+\max\limits_{\varepsilon\leq|\xi|\leq1/\varepsilon}|\xi|^{4\theta-2}/(2a^2)$.
	
	Next, we define the desired Lyapunov function $F_{\midd}(v)=F_{\midd}(v)(t,\xi)$ such that
	\begin{equation*}
	F_{\midd}(v)(t,\xi):=\frac{1}{c_1}E_{\midd}(v)(t,\xi)+\sum\limits_{k=1}^3\Re\big(\bar{v}^{(k)}(t,\xi)v^{(k)}_t(t,\xi)\big)
	\end{equation*}
	with a sufficiently small positive constant $c_1$ to be chosen later.\\
	Taking into consideration (\ref{20003}) and (\ref{20004}) we get for the first-order partial derivative of $F_{\midd}(v)$ with respect to $t$ the relation
	\begin{equation*}
	\frac{\partial}{\partial t}F_{\midd}(v)(t,\xi)\leq-\frac{1}{2}\Big(\frac{2\varepsilon^{2\theta}}{c_1}-2c_0\Big)
	\sum\limits_{k=1}^3|v^{(k)}_t(t,\xi)|^2
	-\frac{1}{2}|\xi|^2\sum\limits_{k=1}^3\varpi_k|v^{(k)}(t,\xi)|^2.
	\end{equation*}
	There exist positive constants $c_2=1/(a^2\varepsilon^2)$ and $c_3=c_2+1/c_1$ satisfying \[ c_2 E_{\midd}(v)(t,\xi)\leq F_{\midd}(v)(t,\xi)\leq c_3
	E_{\midd}(v)(t,\xi)\] for a sufficiently small constant $c_1$. We observe that
	\begin{equation*}
	\begin{split}
	\Big|\sum\limits_{k=1}^3\Re\big(\bar{v}^{(k)}(t,\xi)v^{(k)}_t(t,\xi)\big)\Big|\leq c_2E_{\midd}(v)(t,\xi).
	\end{split}
	\end{equation*}
	Choosing $c_1=\min\left\{2\varepsilon^{2\theta}/(2c_0+1);1/(2c_2)\right\}$ we get
	\begin{equation}\label{modi1}
	\begin{split}
	\frac{\partial}{\partial t}F_{\midd}(v)(t,\xi)\leq-\frac{1}{c_3}F_{\midd}(v)(t,\xi).
	\end{split}
	\end{equation}
	The use of Gronwall's inequality in \eqref{modi1} and the relationship between $F_{\midd}(v)$ and $E_{\midd}(v)$ imply an exponential decay estimate for the energy $E_{\midd}(v)$.
\end{proof}
\subsubsection{Diagonalization procedures in Cases 2.1, 2.2, 2.3}
The main tool in studying the asymptotic behavior of solutions to (\ref{weshould}) if $\xi\in Z_{\intt}(\varepsilon)\cup Z_{\extt}(\varepsilon)$ is the use of the diagonalization schemes developed in the papers \cite{Jachmann2008}$^,$\cite{Reissig2016}$^,$ \cite{ReissigWang2005}. In each zone, we diagonalize the principal part of the system \eqref{weshould}. More in detail, we are especially interested in the asymptotic behavior of the eigenvalues of the coefficient matrix $-\frac{i}{2}|\xi|^{2\theta}B_0-|\xi|B_1$ for small and large frequencies.\\

\noindent \emph{Case 2.1}: $\quad$ To start the diagonalization procedure the matrix $\frac{i}{2}|\xi|^{2\theta}B_0$ has a dominant influence  in comparison with the diagonal matrix $|\xi|B_1$. For this reason, we should diagonalize $B_0$ firstly.
\begin{thm}\label{case1}
	When $\theta\in\left[0,1/2\right)$ with $\xi\in Z_{\intt}(\varepsilon)$ or $\theta\in\left(1/2,1\right]$ with $\xi\in Z_{\extt}(\varepsilon)$, after $\ell$ steps of the diagonalization procedure the starting system \eqref{weshould} is transformed to
	\begin{equation*}
	D_tW^{(\ell)}-A_{\ell}W^{(\ell)}=0,\,\,\,\, W^{(\ell)}(0,\xi)=W_0^{(\ell)}(\xi),
	\end{equation*}
	where
	\begin{equation*}
	A_{\ell}=i|\xi|^{2\theta}\Lambda_1+\sum\limits_{j=2}^{\ell}\Lambda_{j}+R_{\ell}
	\end{equation*}
	and with diagonal matrices $\Lambda_1,\dots,\Lambda_{\ell}$ and the remainder $R_{\ell}$. The asymptotic behavior of these matrices can be described as follows:
	\begin{equation*}
	\begin{split}
	\Lambda_1=O(1),\,\,\,\, \Lambda_{j}=O\big(|\xi|^{2(j-1)+2\theta(3-2j)}\big),\,\,\,\, R_{\ell}=O\big(|\xi|^{2\ell-1+4\theta(1-\ell)}\big).
	\end{split}
	\end{equation*}
	Moreover, the characteristic roots $\mu_{\ell,j}=\mu_{\ell,j}(|\xi|)$ with $j=1,\dots,6,$ of the matrix $A_{\ell}$ have the following asymptotic behavior:
	\begin{alignat*}{3}
	&\mu_{\ell,1}&=ia^2|\xi|^{2-2\theta}+z_1(a)+z_2,\quad&\mu_{\ell,4}&=i|\xi|^{2\theta}-ib^2|\xi|^{2-2\theta}-z_1(a)-z_2,\\
	&\mu_{\ell,2}&=ia^2|\xi|^{2-2\theta}+z_1(a)+z_3,\quad&\mu_{\ell,5}&=i|\xi|^{2\theta}-ia^2|\xi|^{2-2\theta}-z_1(b)-z_5,\\
	&\mu_{\ell,3}&=ib^2|\xi|^{2-2\theta}+z_1(b)+z_4,\quad&\mu_{\ell,6}&=i|\xi|^{2\theta}-ia^2|\xi|^{2-2\theta}-z_1(a)-z_3,
	\end{alignat*}
	modulo $O\big(|\xi|^{7-12\theta}\big)$, where $z_1(a),z_1(b)=O\big(|\xi|^{4-6\theta}\big)$, $z_2,z_3,z_4,z_5=O\big(|\xi|^{6-10\theta}\big)$ and their formulas are shown in Appendix \ref{appendix1}.
\end{thm}
\begin{proof}
	After four steps of the diagonalization procedure, we obtain six pairwise distinct eigenvalues and we may carry out further steps of diagonalization as well. Here, we apply the diagonalization procedure proposed in \cite{ReissigSmith2005}$^,$\cite{Yagdjian1997}.
\end{proof}

\noindent \emph{Case 2.2}: $\quad$ To start the diagonalization procedure the matrix $|\xi|B_1$ has a dominant influence in comparison with the matrix $\frac{i}{2}|\xi|^{2\theta}B_0$. For this reason, we should diagonalize $B_1$ firstly.
\begin{thm}\label{case2}
	When $\theta\in\left[0,1/2\right)$ with $\xi\in Z_{\extt}(\varepsilon)$ or $\theta\in\left(1/2,1\right]$ with $\xi\in Z_{\intt}(\varepsilon)$, after $\ell$ steps of the diagonalization procedure the starting system \eqref{weshould} is transformed to
	\begin{equation*}
	D_tW^{(\ell)}-A_{\ell}W^{(\ell)}=0,\,\,\,\, W^{(\ell)}(0,\xi)=W_0^{(\ell)}(\xi),
	\end{equation*}
	where
	\begin{equation*}
	A_{\ell}=\Lambda_1+\Lambda_2+\sum\limits_{j=3}^{\ell}\Lambda_{j}+R_{\ell}
	\end{equation*}
	and with diagonal matrices $\Lambda_1,\dots,\Lambda_{\ell}$ and the remainder $R_{\ell}$. The asymptotic behavior of these matrices can be described as follows:
	\begin{equation*}
	\begin{split}
	\Lambda_1=O(|\xi|),\,\,\,\,\Lambda_2=O\big(|\xi|^{2\theta}\big),\,\,\,\, \Lambda_{j}=O\big(|\xi|^{(2\theta-1)(2j-5)+2\theta}\big),\,\,\,\, R_{\ell}=O\big(|\xi|^{2(2\theta-1)(\ell-2)+2\theta}\big).
	\end{split}
	\end{equation*}
	Moreover, the characteristic roots $\mu_{\ell,j}=\mu_{\ell,j}(|\xi|)$ with $j=1,\dots,6,$ of the matrix $A_{\ell}$ have the following asymptotic behavior:
	\begin{equation*}
	\mu_{\ell,j}=\pm |\xi|\sqrt{y^2}-\frac{i}{2}|\xi|^{2\theta}-\frac{1}{8\sqrt{y^2}}|\xi|^{4\theta-1}
	\end{equation*}
	modulo $O\big(|\xi|^{6\theta-2}\big)$, where $y=a$ as $j=1,2,4,5$ and $y=b$ as $j=3,6$. If $j=4,5,6$, then we take in the first term the negative sign in above characteristic roots, when $j=1,2,3$, we take in the first term the positive sign in above characteristic roots.
\end{thm}
\begin{proof} After two steps of the diagonalization procedure, we only can get four pairwise distinct eigenvalues. Then, we may decompose the remainder and apply further steps of diagonalizaion. Here, we also apply the diagonalization procedure proposed in \cite{ReissigSmith2005}$^,$\cite{Yagdjian1997}.
\end{proof}

\noindent \emph{Case 2.3}: $\quad$ The matrices $\frac{i}{2}|\xi|^{2\theta}B_0$ and $|\xi|B_1$ have the same influence on the principal part. For this reason, we apply directly the diagonalization procedure to the system
\begin{equation}\label{50001}
D_tW^{(0)}-|\xi|(\frac{i}{2}B_0+B_1)W^{(0)}=0
\end{equation}
as a whole. To study the eigenvalues we recall the relation
\begin{equation*}
\det((\frac{i}{2}B_0+B_1)-\lambda I)=(\lambda^2-i\lambda-a^2)^2(\lambda^2-i\lambda-b^2)=0.
\end{equation*}
Hence, we study the solutions of the two equations
\begin{equation*}
\lambda^2-i\lambda-a^2=0\,\,\,\,\text{or}\,\,\,\, \lambda^2-i\lambda-b^2=0.
\end{equation*}
Let us consider the first equation. Setting $\lambda=\Re\lambda+i\Im\lambda$ with real numbers $\Re\lambda$ and $\Im\lambda$ gives the system of equations
\begin{equation*}
(\Re\lambda)^2-(\Im\lambda)^2+\Im\lambda-a^2=0\,\,\,\,\text{and}\,\,\,\, 2(\Re\lambda)(\Im\lambda)-\Re\lambda=0.
\end{equation*}

\emph{Case 2.3.1}: $\quad$ If $\Re\lambda=0$, then $(\Im\lambda)_{1,2}=\frac{1}{2}\pm\frac{1}{2}\sqrt{1-4a^2}$.

\emph{Case 2.3.2}: $\quad$ If $\Im\lambda=\frac{1}{2}$, then $(\Re\lambda)_{1,2}=\pm\frac{1}{2}\sqrt{4a^2-1}$.

If $a^2=1/4$, then $\Re\lambda=0$ and $(\Im\lambda)_{1,2}=\frac{1}{2}$; if $a^2\in\left(0,1/4\right)$, then $\Re\lambda=0$ and $(\Im\lambda)_{1,2}=\frac{1}{2}\pm\frac{1}{2}\sqrt{1-4a^2}$; if $a^2\in\left(1/4,\infty\right)$, then $(\Re\lambda)_{1,2}=\pm\frac{1}{2}\sqrt{4a^2-1}$ and $\Im\lambda=\frac{1}{2}$. One can follow the above discussion to study the second equation with respect to $b$.

However, it is impossible to verify that all eigenvalues are pairwise distinct for all $\xi\in\mathbb{R}^3\backslash\{0\}$. Hence, the matrix $|\xi|(\frac{i}{2}B_0+B_1)$ cannot be completely diagonalized. Consequently, Jordan normal forms come into play. After a change of variables, we obtain the `almost diagonalized' system.
\begin{thm}
	Assume $\theta=1/2$. If $b^2>a^2>0$, $b^2\neq1/4$, $a^2\neq1/4$, then there exist eigenvalues $\lambda_{1,2}=\big(\frac{1}{2}\pm\frac{1}{2}\sqrt{1-4a^2}\big)i$ when $a^2\in\left(0,1/4\right)$, $\lambda_{1,2}=\frac{1}{2}\sqrt{4a^2-1}+\frac{i}{2}$ when $a^2\in\left(1/4,\infty\right)$, $\lambda_{3,4}=\big(\frac{1}{2}\pm\frac{1}{2}\sqrt{1-4b^2}\big)i$ when $b^2\in\left(0,1/4\right)$ and $\lambda_{3,4}=\pm\frac{1}{2}\sqrt{4b^2-1}+\frac{i}{2}$ when $b^2\in\left(1/4,\infty\right)$. The system \eqref{50001} can be transformed to
	\[
	\begin{split}
	D_t W^{(1)}-|\xi|\big(J_2(\lambda_1)\oplus J_2(\lambda_2)\oplus J_1(\lambda_3)\oplus J_1(\lambda_4)\big)W^{(1)}=0.
	\end{split}
	\]
	A representation of solutions to this system is
		\begin{equation*}
		\begin{aligned}
		 W^{(1)}_{1}(t,\xi)&=e^{i|\xi|\lambda_1t}\big(W_{0,1}^{(1)}(\xi)+i|\xi|tW_{0,2}^{(1)}(\xi)\big),&\quad&W^{(1)}_{2}(t,\xi)=e^{i|\xi|\lambda_1t}W_{0,2}^{(1)}(\xi),\\
		 W^{(1)}_{3}(t,\xi)&=e^{i|\xi|\lambda_2t}\big(W_{0,3}^{(1)}(\xi)+i|\xi|tW_{0,4}^{(1)}(\xi)\big),&\quad&W^{(1)}_{4}(t,\xi)=e^{i|\xi|\lambda_2t}W_{0,4}^{(1)}(\xi),\\
		W^{(1)}_{5}(t,\xi)&=e^{i|\xi|\lambda_3t}W_{0,5}^{(1)}(\xi),&\quad &W^{(1)}_{6}(t,\xi)=e^{i|\xi|\lambda_4t}W_{0,6}^{(1)}(\xi).
		\end{aligned}
		\end{equation*}
	Furthermore, if $1/4=b^2>a^2>0$, then there exist eigenvalues $\lambda_{1,2}=i\big(\frac{1}{2}\pm\frac{1}{2}\sqrt{1-4a^2}\big)$ and $\lambda_3=\frac{i}{2}$, such that the system \eqref{50001} can be transformed to
	\[
	\begin{split}
	D_t W^{(1)}-|\xi|\big(J_2(\lambda_1)\oplus J_2(\lambda_2)\oplus J_2(\lambda_3)\big)W^{(1)}=0.
	\end{split}
	\]
	A representation of solutions to this system is
		\begin{equation*}
		\begin{aligned}
		 W^{(1)}_{1}(t,\xi)&=e^{i|\xi|\lambda_1t}\big(W_{0,1}^{(1)}(\xi)+i|\xi|tW_{0,2}^{(1)}(\xi)\big),&\quad&W^{(1)}_{2}(t,\xi)=e^{i|\xi|\lambda_1t}W_{0,2}^{(1)}(\xi),\\
		 W^{(1)}_{3}(t,\xi)&=e^{i|\xi|\lambda_2t}\big(W_{0,3}^{(1)}(\xi)+i|\xi|tW_{0,4}^{(1)}(\xi)\big),&\quad&W^{(1)}_{4}(t,\xi)=e^{i|\xi|\lambda_2t}W_{0,4}^{(1)}(\xi),\\
		 W^{(1)}_{5}(t,\xi)&=e^{i|\xi|\lambda_3t}\big(W_{0,5}^{(1)}(\xi)+i|\xi|tW_{0,6}^{(1)}(\xi)\big),&\quad&W^{(1)}_{6}(t,\xi)=e^{i|\xi|\lambda_3t}W_{0,6}^{(1)}(\xi).
		\end{aligned}
		\end{equation*}
	In the case $b^2>a^2=1/4$, there exist eigenvalues  $\lambda_{1}=\frac{i}{2}$ and $\lambda_{2,3}=\pm\frac{1}{2}\sqrt{4b^2-1}+\frac{i}{2}$ and the system \eqref{50001} can be transformed to
	\[
	\begin{split}
	D_t W^{(1)}-|\xi|\big(J_4(\lambda_1)\oplus J_1(\lambda_2)\oplus J_1(\lambda_3)\big)W^{(1)}=0.
	\end{split}
	\]
	A representation of solutions to this system is
		\begin{equation*}
		\begin{split}
		 W^{(1)}_{1}(t,\xi)&=e^{i|\xi|\lambda_1t}\big(W_{0,1}^{(1)}(\xi)+i|\xi|tW_{0,2}^{(1)}(\xi)-\frac{1}{2}|\xi|^2t^2W_{0,3}^{(1)}(\xi)-\frac{i}{6}|\xi|^3t^3W_{0,4}^{(1)}(\xi)\big),\\
		W^{(1)}_{2}(t,\xi)&=e^{i|\xi|\lambda_1t}\big(W_{0,2}^{(1)}(\xi)+i|\xi|tW_{0,3}^{(1)}(\xi)-\frac{1}{2}|\xi|^2t^2W_{0,4}^{(1)}(\xi)\big),\\
		W^{(1)}_{3}(t,\xi)&=e^{i|\xi|\lambda_1t}\big(W_{0,3}^{(1)}(\xi)+i|\xi|tW_{0,4}^{(1)}(\xi)\big),\quad\quad W^{(1)}_{4}(t,\xi)=e^{i|\xi|\lambda_1t}W_{0,4}^{(1)}(\xi),\\
		W^{(1)}_{5}(t,\xi)&=e^{i|\xi|\lambda_2t}W_{0,5}^{(1)}(\xi),\quad\quad\quad\quad\quad\quad\quad\quad\quad W^{(1)}_{6}(t,\xi)=e^{i|\xi|\lambda_3t}W_{0,6}^{(1)}(\xi).
		\end{split}
		\end{equation*}
	\end{thm}
\begin{proof} The proof of the statements is based on the application of the diagonalization procedure and a careful treatment of the Jordan matrices appearing in the systems
	(see the proof of Theorem 7.6 in \cite{Reissig2016}).
\end{proof}
\subsubsection{Representation of solutions}\label{representationsolution}
First of all, we introduce the structure of matrices $T_{\theta,\intt}(|\xi|)$ and $T_{\theta,\extt}(|\xi|)$ as follows:
\begin{equation*}
T_{\theta,\intt}(|\xi|)=\left\{\begin{aligned}
&T_1\big(I+\mathcal{N}_1(|\xi|,\theta)\big)\big(I+\mathcal{N}_2(|\xi|,\theta)\big)\big(I+\mathcal{N}_3(|\xi|,\theta)\big)\,\,\,\,&\text{as}\,\,\,\,\theta\in\left[0,1/2\right),\\
&I+\mathcal{N}_4(|\xi|,\theta)\,\,\,\, &\text{as}\,\,\,\,\theta\in\left(1/2,1\right],
\end{aligned}\right.
\end{equation*}
and
\begin{equation*}
T_{\theta,\extt}(|\xi|)=\left\{\begin{aligned}
&I+\mathcal{N}_4(|\xi|,\theta)\quad &\text{as }\,\theta\in\left[0,1/2\right),\\
&T_1\big(I+\mathcal{N}_1(|\xi|,\theta)\big)\big(I+\mathcal{N}_2(|\xi|,\theta)\big)\big(I+\mathcal{N}_3(|\xi|,\theta)\big)\quad&\text{as }\,\theta\in\left(1/2,1\right],
\end{aligned}\right.
\end{equation*}
where
	\[
	\begin{split}
	&T_1=\frac{1}{\sqrt{2}}\left(
	{\begin{array}{*{20}c}
		0 & 1 & 0 & 0 & 0 & 1\\
		1 & 0 & 0 & 1 & 0 & 0\\
		0 & 0 & 1 & 0 & 1 & 0\\
		0 & -1 & 0 & 0 & 0 & 1\\
		-1 & 0 & 0 & 1 & 0 & 0\\
		0 & 0 & -1 & 0 & 1 & 0\\
		\end{array}}
	\right),\,\,\,\,
	\mathcal{N}_1(|\xi|,\theta)=i|\xi|^{1-2\theta}\left(
	{\begin{array}{*{20}c}
		0 & 0 & 0 & -\sqrt{a^2} & 0 & 0\\
		0 & 0 & 0 & 0 & 0 & -\sqrt{a^2}\\
		0 & 0 & 0 & 0 & -\sqrt{b^2} & 0\\
		\sqrt{a^2} & 0 & 0 & 0 & 0 & 0\\
		0 & 0 & \sqrt{b^2} & 0 & 0 & 0\\
		0 & \sqrt{a^2} & 0 & 0 & 0 & 0\\
		\end{array}}
	\right),
	\end{split}
	\]
	\[
	\begin{split}
	\mathcal{N}_2(|\xi|,\theta)=\frac{-i}{|\xi|^{2\theta}}\left(
	{\begin{array}{*{20}c}
		0 & 0 & 0 & z_6(a) & 0 & 0\\
		0 & 0 & 0 & 0 & 0 & z_6(a)\\
		0 & 0 & 0 & 0 & z_6(b) & 0\\
		-z_6(a) & 0 & 0 & 0 & 0 & 0\\
		0 & 0 & -z_6(b) & 0 & 0 & 0\\
		0 & -z_6(a) & 0 & 0 & 0 & 0\\
		\end{array}}
	\right),\,\,\,\, \mathcal{N}_3(|\xi|,\theta)=\frac{i|\xi|^{2\theta-1}}{4}\left(
	{\begin{array}{*{20}c}
		0 & 0 & 0 & -\frac{1}{\sqrt{a^2}} & 0 & 0\\
		0 & 0 & 0 & 0 & -\frac{1}{\sqrt{a^2}} & 0\\
		0 & 0 & 0 & 0 & 0 & -\frac{1}{\sqrt{b^2}}\\
		\frac{1}{\sqrt{a^2}} & 0 & 0 & 0 & 0 & 0\\
		0 & \frac{1}{\sqrt{a^2}} & 0 & 0 & 0 & 0\\
		0 & 0 & \frac{1}{\sqrt{b^2}} & 0 & 0 & 0\\
		\end{array}}
	\right).
	\end{split}
	\]
	We have the following asymptotic behavior: \[ T_1=O(1),\,\,\,\, \mathcal{N}_1(|\xi|,\theta)=O\big(|\xi|^{1-2\theta}\big),\,\,\,\, \mathcal{N}_2(|\xi|,\theta)=O\big(|\xi|^{3-6\theta}\big),\,\,\,\, \mathcal{N}_3(|\xi|,\theta)=O\big(|\xi|^{5-10\theta}\big)\,\,\,\,\mbox{and}\,\,\,\, \mathcal{N}_4(|\xi|,\theta)=O\big(|\xi|^{2\theta-1}\big).\]
From Theorem \ref{case1}, we know when $\theta\in\left[0,1/2\right)$ with small frequencies, the uniform invertibility of $T_{\theta,\text{int}}(|\xi|)$ is clear. Nevertheless, we cannot get six pairwise distinct eigenvalues in \emph{Case 2.2}. Here, according to \cite{Jachmann2008} and applying the Duhamel's principle, we also obtain the representation of solutions. Taking account of the notation $\diag\big(e^{-\mu_l(|\xi|)t}\big)_{l=1}^6:=\diag\big(e^{-\mu_1(|\xi|)t},\dots,e^{-\mu_6(|\xi|)t}\big)$ we can formulate the following results.
\begin{thm}\label{smallmatrix}
	There exists a matrix $T_{\theta,\text{int}}=T_{\theta,\text{int}}(|\xi|)$ for $\theta\in\left[0,1/2\right)\cup\left(1/2,1\right]$, which is uniformly invertible for small frequencies such that the following representation formula holds:
	\begin{equation*}
	W^{(0)}(t,\xi)=T_{\theta,\text{int}}^{-1}(|\xi|)\diag\big(e^{-\mu_l(|\xi|)t}\big)_{l=1}^6T_{\theta,\text{int}}(|\xi|)W_0^{(0)}(\xi),
	\end{equation*}
	where the characteristic roots $\mu_l=\mu_l(|\xi|)$ have the following asymptotic behavior:
	\begin{itemize}
		\item $\theta\in\left[0,1/2\right)$:
		\begin{alignat*}{3}
		&\mu_1(|\xi|)&=a^2|\xi|^{2-2\theta}-iz_1(a)-iz_2,\quad&\mu_4(|\xi|)&=|\xi|^{2\theta}-b^2|\xi|^{2-2\theta}+iz_1(a)+iz_2,\\
		&\mu_2(|\xi|)&=a^2|\xi|^{2-2\theta}-iz_1(a)-iz_3,\quad&\mu_5(|\xi|)&=|\xi|^{2\theta}-a^2|\xi|^{2-2\theta}+iz_1(b)+iz_5,\\
		&\mu_3(|\xi|)&=b^2|\xi|^{2-2\theta}-iz_1(b)-iz_4,\quad&\mu_6(|\xi|)&=|\xi|^{2\theta}-a^2|\xi|^{2-2\theta}+iz_1(a)+iz_3,
		\end{alignat*}
		modulo $O\big(|\xi|^{7-12\theta}\big)$;
		\item $\theta\in\left(1/2,1\right]$: $y=a$ as $l=1,2,4,5$ and $y=b$ as $l=3,6$; when $l=1,2,3$, we take in the first term the negative sign and when $l=4,5,6$, we take in the first term the positive sign in
		\begin{equation*}
		\mu_l(|\xi|)=\pm i|\xi|\sqrt{y^2}+\frac{1}{2}|\xi|^{2\theta}+\frac{i}{8\sqrt{y^2}}|\xi|^{4\theta-1}
		\end{equation*}
		modulo $O\big(|\xi|^{6\theta-2}\big).$
	\end{itemize}
\end{thm}
\begin{proof} The statements of Theorems \ref{case1} and \ref{case2} and the structure of the matrices $T_{\theta,\intt}(|\xi|)$ and $T_{\theta,\intt}^{-1}(|\xi|)$ allow to get the above representations of solutions.
\end{proof}
For large frequencies, there exists the matrix $T_{\theta,\text{ext}}(|\xi|)$ and its inverse matrix $T_{\theta,\text{ext}}^{-1}(|\xi|)$ by using the same explanations for the existence of matrices $T_{\theta,\intt}(|\xi|)$ as well as $T_{\theta,\intt}^{-1}(|\xi|)$.
\begin{thm}\label{largematrix}
	There exists a matrix $T_{\theta,\extt}=T_{\theta,\extt}(|\xi|)$ for $\theta\in\left[0,1/2\right)\cup\left(1/2,1\right]$, which is uniformly invertible for large frequencies such that the following representation formula holds:
	\begin{equation*}
	W^{(0)}(t,\xi)=T_{\theta,\extt}^{-1}(|\xi|)\diag\big(e^{-\mu_l(|\xi|)t}\big)_{l=1}^6T_{\theta,\extt}(|\xi|)W_0^{(0)}(\xi)
	\end{equation*}
	and the characteristic roots $\mu_l=\mu_l(|\xi|)$ have the following asymptotic behavior:
	\begin{itemize}
		\item $\theta\in\left[0,1/2\right)$: $y=a$ as $l=1,2,4,5$ and $y=b$ as $l=3,6$; when $l=1,2,3$, we take in the first term the negative sign and when $l=4,5,6$, we take in the first term the positive sign in
		\begin{equation*}
		\mu_l(|\xi|)=\pm i|\xi|\sqrt{y^2}+\frac{1}{2}|\xi|^{2\theta}+\frac{i}{8\sqrt{y^2}}|\xi|^{4\theta-1}
		\end{equation*}
		modulo $O\big(|\xi|^{6\theta-2}\big);$
		\item $\theta\in\left(1/2,1\right]$:
		\begin{alignat*}{3}
		&\mu_1(|\xi|)&=a^2|\xi|^{2-2\theta}-iz_1(a)-iz_2,\quad&\mu_4(|\xi|)&=|\xi|^{2\theta}-b^2|\xi|^{2-2\theta}+iz_1(a)+iz_2,\\
		&\mu_2(|\xi|)&=a^2|\xi|^{2-2\theta}-iz_1(a)-iz_3,\quad&\mu_5(|\xi|)&=|\xi|^{2\theta}-a^2|\xi|^{2-2\theta}+iz_1(b)+iz_5,\\
		&\mu_3(|\xi|)&=b^2|\xi|^{2-2\theta}-iz_1(b)-iz_4,\quad&\mu_6(|\xi|)&=|\xi|^{2\theta}-a^2|\xi|^{2-2\theta}+iz_1(a)+iz_3,
		\end{alignat*}
		modulo $O\big(|\xi|^{7-12\theta}\big)$.
	\end{itemize}
\end{thm}
\begin{proof} The statements of Theorems \ref{case1} and \ref{case2} and the structure of the matrices $T_{\theta,\intt}(|\xi|)$ and $T_{\theta,\intt}^{-1}(|\xi|)$ allow to get the above representations of solutions.
\end{proof}
\subsubsection{Some qualitative properties of solutions}
In this section, we study some properties of solutions to the linear model \eqref{linearproblem}. Smoothing effect comes in if we choose structural damping  $(-\Delta)^{\theta}u_t$ with $\theta\in(0,1)$ in \eqref{linearproblem}. Nevertheless, the friction ($\theta=0$) and viscoelastic damping ($\theta=1$) lead to the property of propagation of suitable singularities of solutions.
\begin{thm}\label{smoothingeffect}
	Let us consider the Cauchy problem \eqref{linearproblem} with $\theta\in(0,1)$. Data are supposed to belong to the energy space, that is $\big(u^{(k)}_{0},u^{(k)}_{1}\big)\in\ml{D}_{2,2}^{0}$ for $k=1,2,3$. Then, the property of Gevrey smoothing \cite{Rodino1993} appears. This means, that the solutions have the following property:
	\begin{equation*}
	|D|^{s+1} u^{(k)}(t,\cdot), |D|^s u^{(k)}_t(t,\cdot)\in \Gamma^{\kappa}\,\,\,\,\text{for all}\,\,\,\,s\geq 0\,\,\,\,\text{and}\,\,\,\, t>0,
	\end{equation*}
	where the constant ${\kappa}=1/\left(2\min\left\{1-\theta;\theta\right\}\right)$.
\end{thm}
\begin{proof} Following the paper \cite{Reissig2016} one can prove Theorem \ref{smoothingeffect}. We should point out that $\theta=1/2$ yields, in particular, analytic smoothing. Hence, we expect some better behaviors of solutions in the case of the model \eqref{linearproblem} with structural damping mechanism $(-\Delta)^{1/2}u_t$. It is clear that Gevrey smoothing excludes the property of propagation of $\mathcal{C}^\infty$-singularities.
\end{proof}
If we consider elastic waves with friction or viscoelastic damping, then the property of propagation of $H^s$-singularities of solutions could be of interest. The proof of the following two theorems strictly follows the proof of Theorem 7.5 from the paper \cite{Reissig2016}.
\begin{thm}\label{classicalsingularity}
	Let us consider the Cauchy problem \eqref{linearproblem} with friction or exterior damping. Assume $\nabla u^{(k)}_{0},u^{(k)}_{1}\in H^s\big(\mb{R}^3\big)$ but $\nabla u^{(k)}_{0},u^{(k)}_{1}\notin H^{s+1}_{\text{loc}}(x_1,x_2,x_3)$ for $k=1,2,3$ and a given point $(x_1,x_2,x_3)\in\mathbb{R}^3$. Then,
	\begin{equation*}
	\nabla_x u^{(k)}(t,\cdot),u^{(k)}_t(t,\cdot)\notin H^{s+1}_{\text{loc}}\big(\big\{(x_1,x_2,x_3)\pm \sqrt{a^2}t\mathbf{e}\big\}\cup\big\{(x_1,x_2,x_3)\pm \sqrt{b^2}t\mathbf{e}\big\}\big)
	\end{equation*}
	for all $t>0$, where $\mathbf{e}$ is an arbitrary unit vector in $\mathbb{R}^3$.
\end{thm}
\begin{thm}\label{viscosingularity}
	Let us consider the Cauchy problem \eqref{linearproblem} with viscoelastic damping. Assume $\nabla u^{(k)}_{0},u^{(k)}_{1}\in H^s\big(\mathbb{R}^3\big)$ but $\nabla u^{(k)}_{0},u^{(k)}_{1}\notin H^{s+1}_{\text{loc}}(x_1,x_2,x_3)$ for $k=1,2,3$ and a given point $(x_1,x_2,x_3)\in\mathbb{R}^3$. Then,
	\begin{equation*}
	\nabla_x u^{(k)}(t,\cdot),u^{(k)}_t(t,\cdot)\notin H^{s+1}_{\text{loc}}\big(\big\{(x_1,x_2,x_3)-\sqrt{a^2}t\mathbf{e}\big\}\cup\big\{(x_1,x_2,x_3)- \sqrt{b^2}t\mathbf{e}\big\}\big)
	\end{equation*}
	for all $t>0$, where $\mathbf{e}$ is an arbitrary unit vector in $\mathbb{R}^3$.
\end{thm}
\begin{rem}
	Comparing the statements of Theorem \ref{classicalsingularity} with Theorem \ref{viscosingularity}, the propagation pictures are different taking account that the dominant parts of eigenvalues $\mu_l(|\xi|)$ in the cases $l=4,5,6$ are $1/2$ (case $\theta=0$) and $|\xi|^2$ (case $\theta=1$), respectively. It implies $W_{4,5,6}^{(0)} \in \mathcal{C}\big([0,\infty),L^{2,s+1}\big(\mb{R}^3\big)\big)$, that is, $\langle\xi\rangle^{s+1}W_{4,5,6}^{(0)}\in \ml{C}\big([0,\infty),L^2\big(\mb{R}^3\big)\big)$, without any singularity in the model \eqref{linearproblem} with viscoelastic damping.
\end{rem}
The problem of $L^2$ well-posedness can be immediately solved by the above results containing Gevrey smoothing if $\theta\in(0,1)$ and propagation of singularities if $\theta=0,1$.
\begin{thm}\label{wellposedness1}
	Let us consider the Cauchy problem \eqref{linearproblem} with $\theta\in[0,1]$ and $\big(u^{(k)}_{0},u^{(k)}_{1}\big)\in\ml{D}_{2,2}^0$ for $k=1,2,3$. Then, there exists a uniquely determined energy solution
	\begin{equation*}
	u \in \big(\mathcal{C}\big([0,\infty),\dot{H}^1\big(\mb{R}^3\big)\big)\cap \mathcal{C}^1\big([0,\infty),L^2\big(\mb{R}^3\big)\big)\big)^3.
	\end{equation*}
\end{thm}
\begin{proof} From Theorems \ref{smoothingeffect} to \ref{viscosingularity} we may conclude $|D|u^{(k)},u_t^{(k)}\in L^{\infty}\big([0,\infty),L^2\big(\mb{R}^3\big)\big)$.
	By the same approach proving continuity (in time) of energy solutions to the Cauchy problem for the classical wave equation, more careful considerations yield $|D|u^{(k)},u_t^{(k)}\in \mathcal{C}\big([0,\infty),L^2\big(\mb{R}^3\big)\big).$
\end{proof}
\section{Energy estimates}\label{energyestimates}
To show the global (in time) existence of small data solutions, we find solutions in evolution spaces, this means, the solutions are continuous in time and $H^{s+1}$-valued with respect to the spatial variables. Therefore, we need some Matsumura type (almost sharp) $L^2$ estimates for solutions to linear elastic waves with different damping terms and data belonging to $(H^{s+1}\cap L^m)\times(H^s\cap L^m)$ for $s\geq0$ and $m\in[1,2]$. Moreover, to understand some sharp energy estimates for the solutions to \eqref{linearproblem}, we choose data from the space $(|D|^{-1}H^s\cap \dot{H}^{1}_{m})\times(H^s\cap L^m)$. \smallskip

Concerning the sharpness of the derived estimates, we have to point out that the estimates of higher-order energies from Theorems \ref{energydecay} to \ref{additionaldecay} are sharp, where data are from $(|D|^{-1}H^s\cap \dot{H}^{1}_{m})\times(H^s\cap L^m)$. Moreover, when $\theta\in\left[0,1/2\right)$, the total energy estimates from Theorem \ref{enee} with data being from $(H^{s+1}\cap L^1)\times(H^s\cap L^1)$ are almost sharp modulo a parameter $\epsilon>0$. Furthermore, the total energy estimates of solutions for the system \eqref{linearproblem} with structural damping $(-\Delta)^{1/2}u_t$ are sharp.
\subsection{Energy estimates by using the diagonalization procedure}\label{Energyestimatesbyphasespaceanalysis}
We focus on estimates of the classical energy and higher-order energies of solutions with data being from $|D|^{-1}H^s\times H^s$ with or without an additional regularity
$\dot{H}^1_m\times L^m$, $m\in[1,2)$. The main tools are the asymptotic formulas of the solutions from Theorems \ref{smallmatrix} to \ref{largematrix}. Since the proofs are quite standard, we only sketch them.
\begin{thm}\label{energydecay}
	Let us consider the Cauchy problem \eqref{linearproblem} with $\theta\in[0,1]$ and $\big(u^{(k)}_{0},u^{(k)}_{1}\big)\in\ml{D}_{2,2}^s$ for $k=1,2,3,$ and $s\geq0$. Then, we have the following estimates of energies of higher order:
	\begin{equation*}
	\begin{split}
	\||D|^{s+1}u^{(k)}(t,\cdot)\|_{L^2} + \||D|^su^{(k)}_t(t,\cdot)\|_{L^2}&\lesssim(1+t)^{-\frac{s}{2\max\left\{1-\theta;\theta\right\}}}\sum\limits_{k=1}^3\big\|\big(u^{(k)}_{0},u^{(k)}_1\big)\big\|_{\ml{D}_{2,2}^s}.
	\end{split}
	\end{equation*}
\end{thm}
\begin{proof} By virtue of the Parseval-Plancherel theorem and the embedding $H^s\big(\mb{R}^3\big)\hookrightarrow L^2\big(\mb{R}^3\big)$ all for $s\geq0$ we can derive the estimates for the energies of solutions of higher order.
\end{proof}
\begin{thm}\label{energyestimatessss} Let us consider the Cauchy problem \eqref{linearproblem} with $\theta\in[0,1]$ and $\big(u^{(k)}_{0},u^{(k)}_{1}\big)\in\dot{H}^{s+1}\times \dot{H}^s$ for $k=1,2,3$, and $s\geq0$. Then, we have the following
	estimates for the energies of solutions of higher order:
	\begin{equation*}
	\begin{split}
	\||D|^{s+1}u^{(k)}(t,\cdot)\|_{L^2} + \||D|^su^{(k)}_t(t,\cdot)\|_{L^2}&\lesssim\sum\limits_{k=1}^3\big\|\big(u^{(k)}_{0},u^{(k)}_1\big)\big\|_{\dot{H}^{s+1}\times\dot{H}^s}.
	\end{split}
	\end{equation*}
\end{thm}
\begin{proof}
	The estimates for the higher-order energies are determined by estimates localized to the zone $Z_{\intt}(\varepsilon)$. For this reason we may apply
	\begin{equation*}
	\Big\|\ml{F}_{\xi\rightarrow x}^{-1}\Big(\chi_{\intt}(\xi)|\xi|^sT_{\theta,\intt}^{-1}(|\xi|)\diag\big(e^{-\mu_l(|\xi|)t}\big)_{l=1}^6T_{\theta,\intt}(|\xi|)W_0^{(0)}(\xi)\Big)\Big\|_{L^2}
	\lesssim\big\|\ml{F}^{-1}\big(W_0^{(0)}(\xi)\big)\big\|_{\dot{H}^s}.
	\end{equation*}
	This yields the desired estimates. \end{proof}
Now, we suppose an additional regularity $\dot{H}^1_m\times L^m$ for data with $m\in[1,2)$. This implies an additional decay in the corresponding estimates.
\begin{thm}\label{additionaldecay}
	Let us consider the Cauchy problem \eqref{linearproblem} with $\theta\in[0,1]$ and $\big(u^{(k)}_{0},u^{(k)}_{1}\big)\in\ml{D}_{m,2}^s$ for $k=1,2,3$, where $s\geq0$ and  $m\in[1,2)$. Then, we have the following estimates for the solutions and their energies of higher order:
	\begin{equation*}
	\begin{aligned}
	 \|u^{(k)}(t,\cdot)\|_{L^2}&\lesssim(1+t)^{-\frac{6-5m}{4m\max\left\{1-\theta;\theta\right\}}}\sum\limits_{k=1}^3\big\|\big(u^{(k)}_{0},u^{(k)}_1\big)\big\|_{\ml{D}_{m,2}^0}&\text{if}\,\,\,\,& m\in\left[1,6/5\right),\\
	 \||D|^{s+1}u^{(k)}(t,\cdot)\|_{L^2}+\||D|^su^{(k)}_t(t,\cdot)\|_{L^2}&\lesssim(1+t)^{-\frac{3(2-m)+2ms}{4m\max\left\{1-\theta;\theta\right\}}}\sum\limits_{k=1}^3\big\|\big(u^{(k)}_{0},u^{(k)}_1\big)\big\|_{\ml{D}_{m,2}^s} &\text{if}\,\,\,\,& m\in\left[1,2\right).
	\end{aligned}
	\end{equation*}
\end{thm}
\begin{proof} For small frequencies, we apply H\"older's inequality and the Hausdorff-Young inequality to obtain the decay estimates
	\begin{equation*}
	 \|\chi_{\intt}(D)|D|^{s+1}u^{(k)}(t,\cdot)\|_{L^2}+\|\chi_{\intt}(D)|D|^{s}u_t^{(k)}(t,\cdot)\|_{L^2}\lesssim(1+t)^{-\frac{3(2-m)+2ms}{4m\max\left\{1-\theta;\theta\right\}}}\sum\limits_{k=1}^3\big\|\big(u^{(k)}_{0},u^{(k)}_1\big)\big\|_{\ml{D}_{m,2}^0}.
	\end{equation*}
	Moreover, an exponential decay estimate for the solutions with data belonging to $\ml{D}_{2,2}^s$ appears in the zone $Z_{\extt}(\varepsilon)$ of large frequencies.
	However, when we discuss decay estimates for the solution itself by using H\"older's inequality, we immediately obtain
	\begin{equation*}
	\Big\|\ml{F}_{\xi\rightarrow x}^{-1}\big(\chi_{\intt}(\xi)|\xi|^{-1}W^{(0)}(t,\xi)\big)\Big\|_{L^2}\lesssim\big\|\chi_{\intt}(\xi)|\xi|^{-1}e^{-c|\xi|^{2\max\left\{1-\theta;\theta\right\}}t}\big\|_{L^{\frac{2m}{2-m}}}\sum\limits_{k=1}^3\big\|\big(u_0^{(k)},u_1^{(k)}\big)\big\|_{\dot{H}_m^1\times L^m}.
	\end{equation*}
	So, we need $m\in\left[1,6/5\right)$ to avoid a strong influence (non integrability) of the singularity as $|\xi|\rightarrow +0$ and conclude
	\begin{equation*}
	\big\|\chi_{\intt}(\xi)|\xi|^{-1}e^{-c|\xi|^{2\max\left\{1-\theta;\theta\right\}}t}\big\|_{L^{\frac{2m}{2-m}}}<\infty.
	\end{equation*}
	So, the proof is complete.
\end{proof}
\begin{rem} If we use estimates for $\|u_t^{(k)}(t,\cdot)\|_{L^2}$ to derive estimates for $\|u^{(k)}(t,\cdot)\|_{L^2}$ of the solution with data belonging to $\ml{D}^0_{m,2}$ with $m\in\left[6/5,2\right]$, we need an additional assumption for data. To be more precise, we apply the following integral formula:
		\begin{equation}\label{rell}
		\int\nolimits_0^{t}u_{\tau}^{(k)}(\tau,x)d\tau=u^{(k)}(t,x)-u_0^{(k)}(x).
		\end{equation}
		Using estimates for $\|u_t^{(k)}(t,\cdot)\|_{L^2}$ we obtain for $\theta \in [0,1] \setminus \{1/2\}$ the estimate
		\begin{equation*}
		\begin{split}
		 \|u^{(k)}(t,\cdot)\|_{L^2}\lesssim(1+t)^{1-\frac{3(2-m)}{4m\max\left\{1-\theta;\theta\right\}}}\sum\limits_{k=1}^3\big\|\big(u_0^{(k)},u_1^{(k)}\big)\big\|_{\ml{D}_{m,2}^0}+\|u_0^{(k)}\|_{L^2}.
		\end{split}
		\end{equation*}
		Therefore, we suppose for data $\big(u^{(k)}_0,u^{(k)}_1\big)\in(L^2\cap\dot{H}^1\cap \dot{H}^1_m)\times (L^2\cap L^m)$ for $m\in\left[6/5,2\right]$.
	\end{rem}
\begin{rem}
	We cannot expect energy estimates in Section \ref{Energyestimatesbyphasespaceanalysis} depending on a single data only due to the fact that from \eqref{trans01} and the diagonalization procedure, we obtain the representations of solutions by the coupling matrices $T_{\theta,\intt}(|\xi|)$ and $T_{\theta,\extt}(|\xi|)$. These two matrices mix the influences of both data for estimating the solutions.
\end{rem}
\subsection{Energy estimates by energy methods in the Fourier space}\label{Energyestimatesbyenergymethod}
In this part, some energy estimates with data being from Bessel potential spaces with or without an additional regularity $L^m$, $m\in[1,2)$ are of interest. We define the energy of solutions in the phase space for all frequencies as follows:
\begin{equation*}
E_{\text{pha}}(\hat{u})=E_{\text{pha}}(\hat{u})(t,\xi):=|\hat{u}_t(t,\xi)|^2+a^2|\xi|^2|\hat{u}(t,\xi)|^2+\big(b^2-a^2\big)|\xi\cdot\hat{u}(t,\xi)|^2,
\end{equation*}
where dot $\cdot$ denotes the usual inner product in $\mb{R}^3$.
\begin{lem}\label{Hl+1Hllem} The energy $E_{\text{pha}}(\hat{u})$ of the Fourier image $\hat{u}=\hat{u}(t,\xi)$ of the solution $u=u(t,x)$ to the Cauchy problem \eqref{linearproblem} satisfies the following estimate:
	\begin{equation*}
	E_{\text{pha}}(\hat{u})(t,\xi)\lesssim\left\{\begin{aligned} &e^{-c|\xi|^{2\max\left\{1-\theta;\theta\right\}}t}E_{\text{pha}}(\hat{u})(0,\xi) &\text{if}\,\,\,\,&\xi\in Z_{\intt}(\varepsilon),\\
	&e^{-ct} E_{\text{pha}}(\hat{u})(0,\xi)&\text{if}\,\,\,\,&\xi\in Z_{\midd}(\varepsilon)\cup Z_{\extt}(\varepsilon).
	\end{aligned}\right.
	\end{equation*}
\end{lem}
\begin{proof}  Applying the partial Fourier transformation to \eqref{linearproblem} we arrive at the new system
	\begin{equation}\label{energyRn2}
	\left\{
	\begin{aligned}
	&\hat{u}_{tt}+a^2|\xi|^2\hat{u}+\big(b^2-a^2\big)(\xi\cdot\hat{u})\xi+|\xi|^{2\theta}\hat{u}_t=0,\quad &(t,\xi)\in(0,\infty)\times\mb{R}^3,\\
	&(\hat{u},\hat{u}_t)(0,\xi)=(\hat{u}_0,\hat{u}_1)(\xi),\quad &\xi\in\mb{R}^3.
	\end{aligned}
	\right.
	\end{equation}
	After multiplying by $\bar{\hat{u}}_t=\bar{\hat{u}}_t(t,\xi)$ both sides of \eqref{energyRn2} and taking the real part we get
	\begin{equation}\label{energyRn3}
	\frac{\partial}{\partial t}E_{\text{pha}}(\hat{u})(t,\xi)+2|\xi|^{2\theta}|\hat{u}_t(t,\xi)|^2=0.
	\end{equation}
	To obtain decay rates for $E_{\text{pha}}(\hat{u})$ we divide the energy into two parts $\chi_{\intt}(\xi)E_{\text{pha}}(\hat{u})$ and $(1-\chi_{\intt}(\xi))E_{\text{pha}}(\hat{u})$ related to the zones $Z_{\intt}(\varepsilon)$ and $Z_{\midd}(\varepsilon)\cup Z_{\extt}(\varepsilon)$, respectively.
	
	For small frequencies, we multiply $|\xi|^{\gamma}\bar{\hat{u}}(t,\xi)$ on the both sides of \eqref{energyRn2} and take the real part to obtain
	\begin{equation}\label{energyRn4}
	\begin{split}
	\frac{\partial}{\partial t}\big(|\xi|^{\gamma}\Re\big(\hat{u}_t(t,\xi)\bar{\hat{u}}(t,\xi)\big)\big)&-|\xi|^{\gamma}|\hat{u}_t(t,\xi)|^2+a^2|\xi|^{2+\gamma}|\hat{u}(t,\xi)|^2\\
	&+\big(b^2-a^2\big)|\xi|^{\gamma}|\xi\cdot\hat{u}(t,\xi)|^2+|\xi|^{2\theta+\gamma}\Re\big(\hat{u}_t(t,\xi)\bar{\hat{u}}(t,\xi)\big)=0,
	\end{split}
	\end{equation}
	where the positive constant $\gamma$ will be determined later. Adding \eqref{energyRn3} and \eqref{energyRn4} yields
	\begin{equation}\label{energyRn5}
	\begin{split}
	&\frac{\partial}{\partial t}\big(E_{\text{pha}}(\hat{u})(t,\xi)+|\xi|^{\gamma}\Re\big(\hat{u}_t(t,\xi)\bar{\hat{u}}(t,\xi)\big)\big)\\
	&\quad\quad\quad=-\big(2|\xi|^{2\theta}-|\xi|^{\gamma}\big)|\hat{u}_t(t,\xi)|^2-a^2|\xi|^{2+\gamma}|\hat{u}(t,\xi)|^2-\big(b^2-a^2\big)
	|\xi|^{\gamma}|\xi\cdot\hat{u}(t,\xi)|^2-|\xi|^{2\theta+\gamma}\Re\big(\hat{u}_t(t,\xi)\bar{\hat{u}}(t,\xi)\big).
	\end{split}
	\end{equation}
	With \eqref{energyRn5} and Cauchy's inequality it follows
	\begin{equation*}
	\begin{split}
	\chi_{\intt}(\xi)\frac{\partial}{\partial t}&\big(E_{\text{pha}}(\hat{u})(t,\xi)+|\xi|^{\gamma}\Re\big(\hat{u}_t(t,\xi)\bar{\hat{u}}(t,\xi)\big)\big)\\
	 &\quad\quad\leq-\min\left\{2|\xi|^{2\theta}-|\xi|^{\gamma}-|\xi|^{\tilde{\gamma}}/(4a^2);|\xi|^{\gamma}-|\xi|^{4\theta+2\gamma-{\tilde{\gamma}}-2}\right\}
	\chi_{\intt}(\xi)E_{\text{pha}}(\hat{u})(t,\xi),
	\end{split}
	\end{equation*}
	where the positive constant $\tilde{\gamma}$ will be determined later. We state the restrictions for $\gamma$ and $\tilde{\gamma}$, such that, $2\theta\leq\gamma$, $2\theta\leq {\tilde{\gamma}}$ and ${\tilde{\gamma}}+2-4\theta\leq\gamma$. These restrictions show that $\gamma\geq\max\left\{{\tilde{\gamma}}+2-4\theta;2\theta\right\}\geq2\max\left\{1-\theta;\theta\right\}$. Hence,
	\begin{equation*}
	\chi_{\intt}(\xi)\frac{\partial}{\partial t}\big(E_{\text{pha}}(\hat{u})(t,\xi)+|\xi|^{\gamma}\Re\big(\hat{u}_t(t,\xi)\bar{\hat{u}}(t,\xi)\big)\big)\leq-|\xi|^{\gamma}\chi_{\intt}(\xi)E_{\text{pha}}(\hat{u})(t,\xi).
	\end{equation*}
	Again, by virtue of Cauchy's inequality the energy term $\chi_{\intt}(\xi)E_{\text{pha}}(\hat{u})$ can be controlled as follows:
	\begin{equation*}
	\chi_{\intt}(\xi)E_{\text{pha}}(\hat{u})(t,\xi)\leq 3e^{-\frac{2}{3}|\xi|^{\gamma}t}\chi_{\intt}(\xi)E_{\text{pha}}(\hat{u})(0,\xi).
	\end{equation*}
	Hence, $\gamma=2\max\left\{1-\theta;\theta\right\}$ and $\tilde{\gamma}=2\theta$ are the optimal choices, respectively.
	
	For middle and large frequencies, we only multiply \eqref{energyRn3} by $\bar{\hat{u}}=\bar{\hat{u}}(t,\xi)$ and take the real part of it, i.e., $\gamma=0$ in \eqref{energyRn5}, to get
	\begin{equation*}
	\begin{split}
	\big(1-\chi_{\intt}(\xi)\big)\frac{\partial}{\partial t}\big(E_{\text{pha}}(\hat{u})(t,\xi)+\Re\big(\hat{u}_t(t,\xi)\bar{\hat{u}}(t,\xi)\big)\big)\leq-\frac{1}{2}\big(1-\chi_{\intt}(\xi)\big)E_{\text{pha}}(\hat{u})(t,\xi)
	\end{split}
	\end{equation*}
	due to $4\theta-2\leq 2\theta$ for all $\theta\in[0,1]$. Thus,
	\begin{equation*}
	\big(1-\chi_{\intt}(\xi)\big)\frac{\partial}{\partial t}\big(E_{\text{pha}}(\hat{u})(t,\xi)+\Re\big(\hat{u}_t(t,\xi)\bar{\hat{u}}(t,\xi)\big)\big)\leq -\frac{2}{3}\big(1-\chi_{\intt}(\xi)\big)\big(E_{\text{pha}}(\hat{u})(t,\xi)+\Re\big(\hat{u}_t(t,\xi)\bar{\hat{u}}(t,\xi)\big)\big).
	\end{equation*}
	Finally, we may conclude
	\begin{equation*}
	\big(1-\chi_{\intt}(\xi)\big)E_{\text{pha}}(\hat{u})(t,\xi)\leq3e^{-\frac{2}{3}t}\big(1-\chi_{\intt}(\xi)\big)E_{\text{pha}}(\hat{u})(0,\xi).
	\end{equation*}
	This completes the proof.
\end{proof}
\begin{thm}\label{Hl+1Hl}
	Let us consider the Cauchy problem \eqref{linearproblem} with $\theta\in[0,1]$ and $\big(u^{(k)}_{0},u^{(k)}_{1}\big)\in \ml{D}_{2,1}^s$ for $k=1,2,3,$ and $s\geq0$. Then, we have the following estimates:
	\begin{equation*}
	\begin{split}
	\|u^{(k)}(t,\cdot)\|_{L^2}&\lesssim(1+t)\sum\limits_{k=1}^3\big\|\big(u_0^{(k)},u_1^{(k)}\big)\big\|_{\ml{D}_{2,1}^0},\\
	 \||D|^{s+1}u^{(k)}(t,\cdot)\|_{L^2}+\||D|^su^{(k)}_t(t,\cdot)\|_{L^2}&\lesssim(1+t)^{-\frac{s}{2\max\left\{1-\theta;\theta\right\}}}\sum\limits_{k=1}^3\big\|\big(u_0^{(k)},u_1^{(k)}\big)\big\|_{\ml{D}_{2,1}^s}.
	\end{split}
	\end{equation*}
\end{thm}
\begin{proof} By using the Parseval-Plancherel theorem the proof follows immediately from Lemma \ref{Hl+1Hllem}.\end{proof}
Now we turn to energy estimates with an additional regularity $L^m$, $m \in [1,2)$, for data. In the estimates for the solution there appears the time-dependent function
\begin{equation*}
d_{0,m}=d_{0,m}(t):=\left\{\begin{aligned}
&(1+t)^{-\rho_0(m,\theta)}&\text{if}\,\,\,\,\theta\in[0,1/2),\,\,\,\,& m\in\left[1,6/5\right),\\
&(1+t)^{1-\rho_1(m,\theta)}&\text{if}\,\,\,\,\theta\in[0,1/2),\,\,\,\,& m\in\left[6/5,2\right),\\
&(1+t)^{-\frac{6-5m}{4m\theta}}&\text{if}\,\,\,\,\theta\in[1/2,1],\,\,\,\,& m\in\left[1,6/5\right),\\
&(1+t)^{1-\frac{6-3m}{4m\theta}}&\text{if}\,\,\,\,\theta\in[1/2,1],\,\,\,\,& m\in\left[6/5,2\right),
\end{aligned}
\right.
\end{equation*}
where
\begin{eqnarray*}
	&& \rho_0(m,\theta)<\min\left\{\frac{6-5m+2m(1-2\theta)}{4m(1-\theta)};\frac{6-5m}{4m\theta}\right\},\\
	&& \rho_{1}(m,\theta)<\min\left\{\frac{6-3m+2m(1-2\theta)}{4m(1-\theta)};\frac{6-3m}{4m\theta}\right\}.
\end{eqnarray*}
In the estimates for the energies of higher order of the solution there appears the time-dependent function
\begin{equation*}
d_{s+1,m}=d_{s+1,m}(t):=\left\{\begin{aligned}
&(1+t)^{-\rho_{s+1}(m,\theta)}&\text{if}\,\,\,\,\theta\in\left[0,1/2\right),\\
&(1+t)^{-\frac{6-3m+2sm}{4m\theta}}&\text{if}\,\,\,\,\theta\in\left[1/2,1\right],
\end{aligned}
\right.
\end{equation*}
where $s\geq0$, $m\in[1,2)$ and
\begin{equation*}
\rho_{s+1}(m,\theta)<\min\left\{\frac{6-3m+2sm+2m(1-2\theta)}{4m(1-\theta)};\frac{6-3m+2sm}{4m\theta}\right\}.
\end{equation*}
\begin{thm}\label{enee} Let us consider the Cauchy problem \eqref{linearproblem} with $\theta\in[0,1]$, $\big(u^{(k)}_0,u^{(k)}_1\big)\in\ml{D}_{m,1}^s$, for $k=1,2,3,$ $s\geq0$, and $m\in[1,2)$. Then, we have the following estimates:
	\begin{equation*}
	\begin{split}
	\|u^{(k)}(t,\cdot)\|_{L^2}&\lesssim d_{0,m}(t) \sum\limits_{k=1}^3\big\|\big(u_0^{(k)},u_1^{(k)}\big)\big\|_{\ml{D}_{m,1}^0},\\
	\||D|^{s+1}u^{(k)}(t,\cdot)\|_{L^2}+\||D|^su^{(k)}_t(t,\cdot)\|_{L^2}&\lesssim d_{s+1,m}(t)\sum\limits_{k=1}^3\big\|\big(u_0^{(k)},u_1^{(k)}\big)\big\|_{\ml{D}_{m,1}^s}.
	\end{split}
	\end{equation*}
\end{thm}
\begin{proof} The estimates for the case $m=1$ have been studied in detail in \cite{Ikehata2014}. Although the authors described the long-time behavior of the energy for $t\geq T\big(\|u_0\|_{H^1},\|u_0\|_{L^1},\|u_1\|_{L^2},\|u_1\|_{L^1}\big)$, we know that a suitable energy of the solutions with data belonging to the space $\ml{D}_{m,1}^s$ is decaying for all $t\geq0$ by the proof of Lemma \ref{Hl+1Hllem}.
	
	First, we prove the estimates for the energies of higher order in the case $m\in(1,2)$. Using the method from \cite{Ikehata2014} if $\theta\in\left[0,1/2\right)$, then we introduce $d_{s+1,m}(t)=(1+t)^{-\rho_{s+1}(m,\theta)}$. Moreover, since we can follow the approach of the paper \cite{Ikehata2014} we only need to prove
	\begin{equation*}
	\begin{split}
	\int\nolimits_{|\xi|\leq\varepsilon}|\xi|^{2s-4\theta\rho_{s+1}(m,\theta)}|\hat{u}^{(k)}_t(t,\xi)|^2d\xi&
	\lesssim \sum\limits_{k=1}^3\big(\|u^{(k)}_0\|_{L^m}^2+\|u^{(k)}_1\|_{L^m}^2\big),\\
	\int\nolimits_{|\xi|\leq\varepsilon}|\xi|^{2s-4(1-\theta)\rho_{s+1}(m,\theta)+2(1-2\theta)}|\hat{u}^{(k)}(t,\xi)|^2d\xi&
	\lesssim \sum\limits_{k=1}^3\big(\|u^{(k)}_0\|_{L^m}^2+\|u^{(k)}_1\|_{L^m}^2\big).
	\end{split}
	\end{equation*}
	After applying H\"older's inequality and the Hausdorff-Young inequality, the above inequalities can be proved if we require
	\begin{equation*}
	\begin{split}
	2\big(s-2\theta\rho_{s+1}(m,\theta)\big)\frac{m}{2-m}+2&>-1,\\
	2\big(s-2(1-\theta)\rho_{s+1}(m,\theta)+1-2\theta\big)\frac{m}{2-m}+2&>-1.
	\end{split}
	\end{equation*}
	In conclusion, we assume
	\begin{equation*}
	\rho_{s+1}(m,\theta)<\min\left\{\frac{6-3m+2sm+2m(1-2\theta)}{4m(1-\theta)};\frac{6-3m+2sm}{4m\theta}\right\}\,\,\,\,\text{if}\,\,\,\,\theta\in\left[0,1/2\right).
	\end{equation*}
	Following the same approach and setting $d_{0,m}(t)=(1+t)^{-\rho_{0}(m,\theta)}$, we also arrive at the estimate of the solution itself in the case $m\in\left(1,6/5\right)$. In the case $\theta\in\left[1/2,1\right]$, for the estimate of the solution itself with $m\in\left[1,6/5\right)$ and for the estimate of the higher order energies of the solutions with $m\in[1,2)$, after applying Lemma \ref{Hl+1Hllem}, H\"older's inequality and the Hausdorff-Young inequality we may conclude
	\begin{equation*}
	\begin{split}
	 \|u^{(k)}(t,\cdot)\|_{L^2}&\lesssim(1+t)^{-\frac{6-5m}{4m\theta}}\Big((1+t)^{-\frac{1}{2\theta}}\sum\limits_{k=1}^3\|u^{(k)}_0\|_{L^m}+\sum\limits_{k=1}^3\|u^{(k)}_1\|_{L^m}\Big)
+e^{-ct} \sum\limits_{k=1}^3\big\|\big(u^{(k)}_0,u_1^{(k)}\big)\big\|_{H^1\times L^2},\\
	\end{split}
	\end{equation*}
	and
	\begin{equation*}
	\begin{split}
	&\||D|^{s+1}u^{(k)}_t(t,\cdot)\|_{L^2}+\||D|^su^{(k)}_t(t,\cdot)\|_{L^2}\\
	&\qquad\qquad\lesssim(1+t)^{-\frac{3(2-m)+2sm}{4m\theta}}
	 \Big((1+t)^{-\frac{1}{2\theta}}\sum\limits_{k=1}^3\|u^{(k)}_0\|_{L^m}+\sum\limits_{k=1}^3\|u^{(k)}_1\|_{L^m}\Big)+e^{-ct}\sum\limits_{k=1}^3\big\|\big(u^{(k)}_0,u_1^{(k)}\big)\big\|_{H^{s+1}\times H^s}.
	\end{split}
	\end{equation*}
	To get estimates for the solution itself with $m\in\left[6/5,2\right)$ for all $\theta\in[0,1]$, we can use the estimate for the $L^2$ norm of $u_t^{(k)}(t,\cdot)$ by the relation \eqref{rell}. Finally, let us mention that we could have better estimates by using the right-hand sides $\tilde{d}_1(t)\sum\limits_{k=1}^3\|u^{(k)}_0\|+\tilde{d}_2(t)\sum\limits_{k=1}^3\|u^{(k)}_1\|$ with suitable norms. Nevertheless, our goal is to derive estimates with right-hand sides $\tilde{d}(t)\sum\limits_{k=1}^3\big\|\big(u^{(k)}_0,u^{(k)}_1\big)\big\|$ with suitable norms.
\end{proof}
\begin{rem} In the case  $\theta\in\left[1/2,1\right]$ with $m=1$, the decay rates in Theorem \ref{enee} are better than those of the paper \cite{Ikehata2014} due to the fact that there is no any ambiguity of $\epsilon>0$ in the statements of Theorem \ref{enee}.
\end{rem}
\begin{thm}\label{01a} Let us consider the Cauchy problem \eqref{linearproblem} with $\theta\in[0,1]$, $\big(u^{(k)}_{0},u^{(k)}_{1}\big)\in(\dot{H}^{s+1}\cap L^m)\times (\dot{H}^s\cap L^m)$ for $k=1,2,3$, $s\geq0$, and $m\in[1,2)$. Then, we have the following estimates:
	\begin{equation*}
	\begin{split}
	\|u^{(k)}(t,\cdot)\|_{L^2}&\lesssim d_{0,m}(t)\sum\limits_{k=1}^3\big\|\big(u_0^{(k)},u_1^{(k)}\big)\big\|_{(\dot{H}^1\cap L^m)\times (L^2\cap L^m)},\\
	\||D|^{s+1}u^{(k)}(t,\cdot)\|_{L^2}+\||D|^su^{(k)}_t(t,\cdot)\|_{L^2}&\lesssim d_{s+1,m}(t)\sum\limits_{k=1}^3\big\|\big(u_0^{(k)},u_1^{(k)}\big)\big\|_{(\dot{H}^{s+1}\cap L^m)\times (\dot{H}^s\cap L^m)}.
	\end{split}
	\end{equation*}
\end{thm}
\begin{proof}
	The above assumptions for data allow modifying the considerations for large frequencies $\xi$.
\end{proof}
\section{Diffusion phenomena}\label{diffusionphenomenon}
The diffusion phenomenon allows us to bridge a decay behavior of solutions to dissipative elastic waves with a decay behavior of solutions to corresponding evolution systems with suitable data. Because of Theorem \ref{additionaldecay} from the previous section, we know that $(L^2\cap L^m)$-$L^2$ estimates are determined by the behavior of the characteristic roots for small frequencies only. For large frequencies, the behavior of the characteristic roots together with the regularity of data implies even an exponential decay. For this reason, the diffusion phenomenon is explained by the behavior of Fourier multipliers localized to small frequencies.

To obtain a result on the diffusion phenomenon for our starting linear system \eqref{linearproblem}, we choose the following Cauchy problem for an evolution reference system:
\begin{equation}\label{refsystem}
\widetilde{U}_t+\widetilde{M}_1(-\Delta)^{\sigma_1}\widetilde{U}+\widetilde{M}_2(-\Delta)^{\sigma_2}\widetilde{U}=0,\,\,\,\, \widetilde{U}(0,x)=\ml{F}^{-1}\big(H(|\xi|)W_0^{(0)}(\xi)\big)(x),
\end{equation}
where the nonnegative constants $\sigma_1,\sigma_2$ and matrices $\widetilde{M}_1$, $\widetilde{M}_2$, and $H=H(|\xi|)$ will be given in each subsection. It is clear that the solution $\widetilde{U}=\widetilde{U}(t,x)$ to this evolution system \eqref{refsystem} can be represented as follows:
\begin{equation}\label{represenheat}
\widetilde{U}(t,x)=\ml{F}_{\xi\rightarrow x}^{-1}\Big(\diag\big(e^{-\tilde{\mu}_l(|\xi|)t}\big)_{l=1}^6H(|\xi|)W_0^{(0)}(\xi)\Big),
\end{equation}
where the eigenvalues $\tilde{\mu}_l=\tilde{\mu}_l(|\xi|)$ are the principal part of the corresponding eigenvalues $\mu_l=\mu_l(|\xi|)$ from Theorem \ref{smallmatrix} for small frequencies. We will explain them in detail later. Now, we introduce $\widetilde{W}=\widetilde{W}(t,\xi)=\ml{F}_{x\rightarrow \xi}(\widetilde{U})(t,\xi)$.
\begin{rem} Assume $\theta=1/2$ in the dissipative elastic waves \eqref{linearproblem}. We observe that $e^{-\frac{1}{2}|\xi|t}$ plays an important role in the representation of $W^{(1)}=W^{(1)}(t,\xi)$ from \emph{Case 2.3}. Consequently, from direct calculation there is not any improvement in the decay estimates for the difference between the solutions to the system \eqref{linearproblem} with $\theta=1/2$ and the solutions to its reference evolution system. For this reason, we only study the diffusion type phenomena to the dissipative system \eqref{linearproblem} with $\theta\in\left[0,1/2\right)\cup\left(1/2,1\right]$.
\end{rem}
\subsection{Diffusion phenomenon for the linear model with $\theta=0$}
According to the principal real part of $\mu_l(|\xi|)$ for small frequencies from Theorem \ref{smallmatrix}, we choose $\sigma_1=1$, $\sigma_2=0$ and $\widetilde{M}_1=\diag\big(a^2,a^2,b^2,0,0,0\big)$, $\widetilde{M}_2=\diag(0,0,0,1,1,1)$ and $H(|\xi|)=\big(I+\mathcal{N}_2(|\xi|,0)\big)^{-1}\big(I+\mathcal{N}_1(|\xi|,0)\big)^{-1}T_1^{-1}$ in the evolution system \eqref{refsystem}, that is,
\begin{equation}\label{di1}
\left\{\begin{aligned}
&\widetilde{U}_t-\diag\big(a^2,a^2,b^2,0,0,0\big)\Delta\widetilde{U}+\diag(0,0,0,1,1,1)\widetilde{U}=0,\\
&\widetilde{U}(0,x)=\ml{F}^{-1}\big(\big(I+\mathcal{N}_2(|\xi|,0)\big)^{-1}\big(I+\mathcal{N}_1(|\xi|,0)\big)^{-1}T_1^{-1}W_0^{(0)}(\xi)\big)(x).
\end{aligned}\right.
\end{equation}
Therefore, the eigenvalues in \eqref{represenheat} are $\tilde{\mu}_{1,2}(|\xi|)=a^2|\xi|^2$, $\tilde{\mu}_3(|\xi|)=b^2|\xi|^2$ and $\tilde{\mu}_{4,5,6}(|\xi|)=1$, that is, we take the corresponding $\mu_l(|\xi|)$ from Theorem \ref{smallmatrix} after neglecting the terms $O\big(|\xi|^4\big)$ when $l=1,2,3$ and neglecting the terms $O\big(|\xi|^2\big)$ when $l=4,5,6$. We now state our result.
\begin{thm}\label{diffusion01}
	Let us consider the system \eqref{linearproblem} with $\theta=0$. We assume that data $(u_0^{(k)},u_1^{(k)})\in \dot{H}^1_m\times L^m$ with $m\in[1,2]$ for $k=1,2,3$. Then, we obtain for the solution $W^{(0)}=W^{(0)}(t,\xi)$ to the Cauchy problem \eqref{weshould} the estimate
	\begin{equation*}
	\big\|\chi_{\intt}(D)\ml{F}_{\xi\rightarrow x}^{-1}\big(W^{(0)}-T_1\big(I+\mathcal{N}_1(|\xi|,0)\big)\big(I+\mathcal{N}_2(|\xi|,0)\big)\widetilde{W}\big)(t,\cdot)\big\|_{\dot{H}^s}\lesssim(1+t)^{-\frac{3(2-m)+2ms}{4m}-\frac{1}{2}}\sum\limits_{k=1}^3\big\|\big(u_0^{(k)},u_1^{(k)}\big)\big\|_{\dot{H}_m^1\times L^m}.
	\end{equation*}
\end{thm}
\begin{proof} We know that
	\begin{equation*}
	\begin{aligned}
	T_{0,\intt}(|\xi|)&=H(|\xi|)+P_1(|\xi|),&\text{where}&\,\,\,\, H(|\xi|)=\big(I+\mathcal{N}_2(|\xi|,0)\big)^{-1}\big(I+\mathcal{N}_1(|\xi|,0)\big)^{-1}T_1^{-1},&P_1(|\xi|)=O(|\xi|),\\
	T_{0,\intt}^{-1}(|\xi|)&=L(|\xi|)+P_2(|\xi|),&\text{where}& \,\,\,\, L(|\xi|)=T_1\big(I+\mathcal{N}_1(|\xi|,0)\big)\big(I+\mathcal{N}_2(|\xi|,0)\big), &P_2(|\xi|)=O(|\xi|).
	\end{aligned}
	\end{equation*}
	Therefore, we decompose the function of interest into three parts, that is,
	\begin{equation*}
	\chi_{\intt}(D)\ml{F}_{\xi\rightarrow x}^{-1}\big(W^{(0)}-L(|\xi|)\widetilde{W}\big)(t,x)=I_1(t,x)+I_2(t,x)+I_3(t,x),
	\end{equation*}
	where we define
	\begin{equation*}
	\begin{split}
	I_1(t,x)&:=\ml{F}_{\xi\rightarrow x}^{-1}\big(\chi_{\intt}(\xi)L(|\xi|)\diag\big(e^{-\mu_l(|\xi|)t}-e^{-\tilde{\mu}_l(|\xi|)t}\big)_{l=1}^6H(|\xi|)W_0^{(0)}(\xi)\big),\\
	I_2(t,x)&:=\ml{F}_{\xi\rightarrow x}^{-1}\big(\chi_{\intt}(\xi)L(|\xi|)\diag\big(e^{-\mu_l(|\xi|)t}\big)_{l=1}^6P_1(|\xi|)W_0^{(0)}(\xi)\big),\\
	I_3(t,x)&:=\ml{F}_{\xi\rightarrow x}^{-1}\big(\chi_{\intt}(\xi)P_2(|\xi|)\diag\big(e^{-\mu_l(|\xi|)t}\big)_{l=1}^6T_{0,\intt}(|\xi|)W_0^{(0)}(\xi)\big).
	\end{split}
	\end{equation*}
	Here, we make use of the fact that
	\begin{equation*}
	e^{-\mu_l(|\xi|)t}-e^{-\tilde{\mu}_l(|\xi|)t}=-r_l(|\xi|)te^{-\tilde{\mu}(|\xi|)t}\int\nolimits_0^1e^{-r_l(|\xi|)t\tau}d\tau,
	\end{equation*}
	where $r_l=r_l(|\xi|)$ denotes the $O\big(|\xi|^4\big)$-terms from Theorem \ref{smallmatrix} for the corresponding eigenvalues $\mu_l(|\xi|)$. Applying Theorem \ref{additionaldecay}, we obtain the estimate
	\begin{equation*}
	\|I_1(t,\cdot)\|_{\dot{H}^s}\lesssim(1+t)^{-\frac{3(2-m)+2ms}{4m}-1}\sum\limits_{k=1}^3\big\|\big(u_0^{(k)},u_1^{(k)}\big)\big\|_{\dot{H}_m^1\times L^m}.
	\end{equation*}
	Following the same procedure for the other two parts gives
	\begin{equation*}
	 \|I_2(t,\cdot)\|_{\dot{H}^s}+\|I_3(t,\cdot)\|_{\dot{H}^s}\lesssim(1+t)^{-\frac{3(2-m)+2ms}{4m}-\frac{1}{2}}\sum\limits_{k=1}^3\big\|\big(u_0^{(k)},u_1^{(k)}\big)\big\|_{\dot{H}_m^1\times L^m}.
	\end{equation*}
	Summarizing, we obtain
	\begin{equation*}
	\big\|\chi_{\intt}(D)\ml{F}_{\xi\rightarrow \xi}^{-1}\big(W^{(0)}-L(\xi)\widetilde{W}\big)(t,\cdot)\big\|_{\dot{H}^s}\lesssim(1+t)^{-\frac{3(2-m)+2ms}{4m}-\frac{1}{2}}\sum\limits_{k=1}^3\big\|\big(u_0^{(k)},u_1^{(k)}\big)\big\|_{\dot{H}_m^1\times L^m}.
	\end{equation*}
	The proof is complete.
\end{proof}
\subsection{Double diffusion phenomena for the linear model with $\theta\in\left(0,1/2\right)$}
According to the principal real part of $\mu_l(|\xi|)$ for small frequencies from Theorem \ref{smallmatrix} we assume $\sigma_1=1-\theta$, $\sigma_2=\theta$ and $\widetilde{M}_1=\diag\big(a^2,a^2,b^2,0,0,0\big)$, $\widetilde{M}_2=\diag(0,0,0,1,1,1)$, $H(|\xi|)=\big(I+\mathcal{N}_2(|\xi|,\theta)\big)^{-1}\big(I+\mathcal{N}_1(|\xi|,\theta)\big)^{-1}T_1^{-1}$ in the evolution system \eqref{refsystem}, that is,
\begin{equation}\label{di2}
\left\{
\begin{aligned}
&\widetilde{U}_t+\diag\big(a^2,a^2,b^2,0,0,0\big)(-\Delta)^{1-\theta}\widetilde{U}+\diag(0,0,0,1,1,1)(-\Delta)^{\theta}\widetilde{U}=0,\\
&\widetilde{U}(0,x)=\ml{F}^{-1}\big(\big(I+\mathcal{N}_2(|\xi|,\theta)\big)^{-1}\big(I+\mathcal{N}_1(|\xi|,\theta)\big)^{-1}T_1^{-1}W_0^{(0)}(\xi)\big)(x).
\end{aligned}\right.
\end{equation}
Hence, we have $\tilde{\mu}_{1,2}(|\xi|)=a^2|\xi|^{2-2\theta}$, $\tilde{\mu}_{3}(|\xi|)=b^2|\xi|^{2-2\theta}$ and $\tilde{\mu}_{4,5,6}(|\xi|)=|\xi|^{2\theta}$ in \eqref{represenheat}.
\begin{thm}\label{diffusion02}
	Let us consider the system \eqref{linearproblem} with $\theta\in\left(0,1/2\right)$. We assume that data $(u_0^{(k)},u_1^{(k)})\in \dot{H}^1_m\times L^m$ with $m\in[1,2]$ for $k=1,2,3$. Then, we obtain for the solution $W^{(0)}=W^{(0)}(t,\xi)$ to the Cauchy problem \eqref{weshould} the estimate
	\begin{equation*}
	\big\|\chi_{\intt}(D)\ml{F}_{\xi\rightarrow x}^{-1}\big(W^{(0)}-T_1\big(I+\mathcal{N}_1(|\xi|,\theta)\big)\big(I+\mathcal{N}_2(|\xi|,\theta)\big)\widetilde{W}\big)(t,\cdot)\big\|_{\dot{H}^s}\lesssim(1+t)^{-\frac{3(2-m)+2ms}{4m(1-\theta)}-\frac{1-2\theta}{2(1-\theta)}}\sum\limits_{k=1}^3\big\|\big(u_0^{(k)},u_1^{(k)}\big)\big\|_{\dot{H}_m^1\times L^m}.
	\end{equation*}
\end{thm}
\begin{proof} Following the same steps of the proof of Theorem \ref{diffusion01} we immediately arrive at the statement of the theorem.
\end{proof}
The asymptotic behavior of eigenvalues $\mu_{1,2,3}(|\xi|)=O\big(|\xi|^{2-2\theta}\big)$ and $\mu_{4,5,6}(|\xi|)=O\big(|\xi|^{2\theta}\big)$ in Theorem \ref{smallmatrix} is the motivation for us to study \emph{double diffusion phenomena}. This new effect has been interpreted for the wave equation with structural damping $(-\Delta)^{\theta}u_t$ when $\theta\in\left(0,1/2\right)$ in the paper \cite{D'abbiccoEbert2014}. In fact, if we rewrite the solution to \eqref{di2} by $\widetilde{U}=\widetilde{U}(t,x)=\big(\widetilde{U}^{+}(t,x),\widetilde{U}^{-}(t,x)\big)^{\mathrm{T}}$, the first part $\ml{F}_{\xi\rightarrow x}^{-1}\big(W^{(0)}_{1,2,3}\big)(t,x)$ behaves like the solution to the parabolic-type system with a suitable choice of data $\widetilde{U}^{+}_0=\widetilde{U}^{+}_0(x)$, that is,
\begin{equation*}
\widetilde{U}^{+}_t+\diag\big(a^2,a^2,b^2\big)(-\Delta)^{1-\theta}\widetilde{U}^{+}=0,\,\,\,\, \widetilde{U}^{+}(0,x)=\widetilde{U}^{+}_0(x).
\end{equation*}
The second part $\ml{F}_{\xi\rightarrow x}^{-1}\big(W^{(0)}_{4,5,6}\big)(t,x)$ behaves like the solution to another parabolic-type system with a suitable choice of data $\widetilde{U}^{-}_0=\widetilde{U}^{-}_0(x)$, that is,
\begin{equation*}
\widetilde{U}^{-}_t+(-\Delta)^{\theta}\widetilde{U}^{-}=0,\,\,\,\,\widetilde{U}^{-}(0,x)=\widetilde{U}^{-}_0(x).
\end{equation*}
Nevertheless, due to the mixed influence from the multiplication by matrices $T_{\theta,\intt}$ and $T^{-1}_{\theta,\intt}$, we can observe only the decay rate influenced by the eigenvalues $\mu_{1,2,3}(|\xi|)=O\big(|\xi|^{2-2\theta}\big)$ in Theorem \ref{diffusion02}.
\subsection{Diffusion phenomenon for the linear model with $\theta\in\left(1/2,1\right]$}
By the same reason, the components of $\mu_l(|\xi|)$ in Theorem \ref{smallmatrix} imply $\sigma_1=\theta$, $\sigma_2=1/2$ and $\widetilde{M}_1=\frac{1}{2}\diag(1,1,1,1,1,1)$, $\widetilde{M}_2=i\diag\big(\sqrt{a^2},\sqrt{a^2},\sqrt{b^2},-\sqrt{a^2},-\sqrt{a^2},-\sqrt{b^2}\big)$, $H(|\xi|)=\big(I+\mathcal{N}_3(|\xi|,\theta)\big)^{-1}$ in the evolution system \eqref{refsystem}, that is,
\begin{equation}\label{di3}
\left\{
\begin{aligned}
&\widetilde{U}_t+\frac{1}{2}\diag(1,1,1,1,1,1)(-\Delta)^{\theta}\widetilde{U}+i\diag\big(\sqrt{a^2},\sqrt{a^2},\sqrt{b^2},-\sqrt{a^2},-\sqrt{a^2},-\sqrt{b^2}\big)(-\Delta)^{1/2}\widetilde{U}=0,\\
&\widetilde{U}(0,x)=\ml{F}^{-1}\big(\big(I+\mathcal{N}_3(|\xi|,\theta)\big)^{-1}W_0^{(0)}(\xi)\big)(x).
\end{aligned}\right.
\end{equation}
So, the eigenvalues in \eqref{represenheat} can be written as $\tilde{\mu}_{1,2}(|\xi|)=-i|\xi|\sqrt{a^2}+\frac{1}{2}|\xi|^{2\theta}$, $\tilde{\mu}_{3}(|\xi|)=-i|\xi|\sqrt{b^2}+\frac{1}{2}|\xi|^{2\theta}$, $\tilde{\mu}_{4,5}(|\xi|)=i|\xi|\sqrt{a^2}+\frac{1}{2}|\xi|^{2\theta}$ and $\tilde{\mu}_6(|\xi|)=i|\xi|\sqrt{b^2}+\frac{1}{2}|\xi|^{2\theta}$.
\begin{thm}\label{diffusion03}
	Let us consider the system \eqref{linearproblem} with $\theta\in\left(1/2,1\right]$. We assume that data $(u_0^{(k)},u_1^{(k)})\in \dot{H}^1_m\times L^m$ with $m\in[1,2]$ for $k=1,2,3$. Then, we obtain for the solution $W^{(0)}=W^{(0)}(t,\xi)$ to the Cauchy problem \eqref{weshould} the estimate
	\begin{equation*}
	\big\|\chi_{\intt}(D)\ml{F}_{\xi\rightarrow x}^{-1}\big(W^{(0)}-\big(I+\mathcal{N}_3(|\xi|,\theta)\big)\widetilde{W}\big)(t,\cdot)\big\|_{\dot{H}^s}\lesssim(1+t)^{-\frac{3(2-m)+2ms}{4m\theta}-\frac{2\theta-1}{2\theta}}\sum\limits_{k=1}^3
	\big\|\big(u_0^{(k)},u_1^{(k)}\big)\big\|_{\dot{H}_m^1\times L^m}.
	\end{equation*}
\end{thm}
\begin{proof} We follow the same step of the proof of Theorem \ref{diffusion01} to derive the desired result.
\end{proof}
\begin{rem} From Theorems \ref{diffusion01} to \ref{diffusion03}, we found the \emph{diffusion structure} for linear elastic waves with structural damping $(-\Delta)^{\theta}u_t$ if $\theta\in\left[0,1/2\right)\cup\left(1/2,1\right]$ as $t\rightarrow \infty$. This means, if we compare the estimates from Theorems \ref{diffusion01} to \ref{diffusion03} with the estimates in Theorem \ref{additionaldecay}, then we see that the decay rate can be improved by $-\frac{1-2\theta}{2(1-\theta)}$ when $\theta\in\left[0,1/2\right)$ and $-\frac{2\theta-1}{2\theta}$ when $\theta\in\left(1/2,1\right]$ as $t\rightarrow\infty$.
\end{rem}
\section{Global (in time) existence of small data solutions}\label{ge}
This section is devoted to the study of the global (in time) existence of small data solutions (GESDS) to the Cauchy problem for the weakly coupled system \eqref{fsdew001}. By using the estimates for solutions to linear parameter dependent Cauchy problems and Banach's fixed-point theorem, the global (in time) existence of energy solutions for small data belonging to the space $H^1\times L^2$ with an additional $L^m$ regularity and regularity parameter $m\in[1,2)$ or to the space $H^{s+1}\times H^s$, $s>0$, with an additional $L^m$ regularity and regularity parameter $m\in[1,2)$ are established.

For the sake of clarity, in this section the triplet $(k_1,k_2,k_3)$ can be chosen in the following way:
\begin{itemize}
	\item $k_1=1$, $k_2=2$ and $k_3=3$;
	\item $k_1=2$, $k_2=3$ and $k_3=1$;
	\item $k_1=3$, $k_2=1$ and $k_3=2$.
\end{itemize}
We re-define $p_{k_j+1}=p_{k_{j+1}}$, $p_{k_j+2}=p_{k_{j+2}}$ with $p_{k_{4}}=p_{k_1}$, $p_{k_{5}}=p_{k_2}$ for $j=1,2,3$ and $U^{(k_j-1)}(t,x)=U^{(k_{j-1})}(t,x)$ with $U^{(k_0)}(t,x)=U^{(k_3)}(t,x)$ for $j=1,2,3$.\\
Finally, we introduce for our further approach exponents $p_c(m,\theta)$, $\alpha_k(m,\theta)$ and $\widetilde{\alpha}_k(m,\theta)$ with some parameters $\theta\in\left[1/2,1\right]$, $m\in\left[1,6/5\right)$ and some balanced exponents  for  $k=k_1,k_2,k_3$.
\begin{enumerate}
		\item  According to the paper \cite{D'abbiccoReissig2014} we introduce the following exponent:
		\begin{equation}\label{cri01}
		p_c(m,\theta):=1+\frac{m(2\theta+1)}{3-m}\,\,\,\,\text{if}\,\,\,\,\theta\in\left[1/2,1\right]\,\,\,\,\text{and}\,\,\,\,m\in\left[1,6/5\right).
		\end{equation}
		As mentioned in the Introduction of this paper the paper \cite{D'abbiccoReissig2014} proved the critical exponent $p_c(1,1/2)=2$ to the single semi-linear wave equation with structural damping $(-\Delta)^{1/2}U_t$ in three-dimensions. Moreover, the authors of the paper \cite{D'abbiccoReissig2014} showed that the global (in time) existence of small data solutions to the semi-linear structurally damped wave equation can be proved for $p_c(1,\theta)<p\leq 3$ for $\theta\in\left[1/2,1\right]$ in the Cauchy problem \eqref{semistructuraldampedwave}.\\
		Moreover, let us introduce the balanced exponents $p_{\text{bal}}(3/2,s,\theta)$ and $p_{\text{bal}}(m,0,\theta)$ respectively:
		\begin{equation*}
		\left.\begin{aligned}
		&p_{\text{bal}}(3/2,s,\theta):=2+\frac{2+4s(1-\theta)}{5-6\theta+2s}&\text{if}\,\,\,\,&m=3/2,\,\,\,\, s\in\left[0,1/2\right),\, \theta\in\left[0,1/2\right),\\
		&p_{\text{bal}}(m,0,\theta):=2+\frac{6(m-2+2\theta)}{2m\theta-3m+6}&\text{if}\,\,\,\,&m\in\left[6/5,3/2\right),\,\,\,\, s=0,\,\theta\in\left[1/2,1\right].\\
		\end{aligned}\right.
		\end{equation*}
		\item The following parameter
		\begin{equation}\label{cri1}
		 \alpha_{k}(m,\theta):=m\frac{2\theta+(1+2\theta)p_{k+1}+p_{k}p_{k+1}}{2(p_{k}p_{k+1}-1)}\,\,\,\,\text{if}\,\,\,\,\theta\in\left[1/2,1\right]\,\,\,\,\text{and}\,\,\,\,m\in\left[1,6/5\right)
		\end{equation}
		is motivated by the recent paper \cite{D'abbicco2015}. The author proved the existence of global (in time) Sobolev solutions to the weakly coupled system for structurally damped wave equations \eqref{weaklycoupledstrucuraldamped}. Especially, in three-dimensions, unique global (in time) solutions exist under the condition
		\begin{equation*}
		\alpha_{\max}(1,1/2)=\max\left\{\alpha_{1}(1,1/2);\alpha_{2}(1,1/2)\right\}<3/2,
		\end{equation*}
		where we re-define $p_1=p$, $p_2=q$ and $p_3=p$ in the condition \eqref{cri1}. Under the assumption $\alpha_{\max}(1,1/2)>3/2$ the author also proved blow up of solutions by applying the test function method. Additionally, we should point out the relation between the parameters \eqref{cri01} and \eqref{cri1}. If we consider the condition $\alpha_{k}(m,\theta)<3/2$, it also can be rewritten as
		\begin{equation*}
		p_{k+1}\big(p_{k}+1-p_c(m,\theta)\big)>p_c(m,\theta).
		\end{equation*}
		Next, we introduce the balanced parameters $\alpha_{k,\text{bal}}(3/2,s,\theta)$ and $\alpha_{k,\text{bal}}(m,0,\theta)$. If $s\in\left[0,1/2\right)$ and $\theta\in\left[0,1/2\right)$, we introduce
		\begin{equation*}
		 \alpha_{k,\text{bal}}(3/2,s,\theta):=\frac{9-12\theta+4s(2-\theta)+((7-6\theta)+2s(3-2\theta))p_{k+1}-((2-6\theta)+2s)p_{k}p_{k+1}}{2(p_{k}p_{k+1}-1)}.
		\end{equation*}
		If $m\in\left[6/5,3/2\right)\text{ and }\theta\in\left[1/2,1\right]$, we define
		\begin{equation*}
		\alpha_{k,\text{bal}}(m,0,\theta):=\frac{4m\theta+12\theta-3+(2m\theta+12\theta+3m-6)p_{k+1}-(2m\theta-3m+3)p_{k}p_{k+1}}{2(p_{k}p_{k+1}-1)}.
		\end{equation*}
		\item We introduce a parameter
		\begin{equation}\label{cri4}
		 \widetilde{\alpha}_{k}(m,\theta):=m\frac{2\theta+(1+2\theta)(p_{k+1}+1)p_{k+2}+p_{1}p_{2}p_3}{2(p_{1}p_{2}p_3-1)}\,\,\,\,\text{if}\,\,\,\,\theta\in\left[1/2,1\right]\,\,\,\,\text{and}\,\,\,\,m\in\left[1,6/5\right).
		\end{equation}
		Also, we should indicate the relation between the parameters \eqref{cri01} and \eqref{cri4}. If we consider the condition $\widetilde{\alpha}_{k}(m,\theta)<3/2$, it also can be rewritten as
		\begin{equation*}
		p_{k+2}\big(p_{k+1}\big(p_k+1-p_c(m,\theta)\big)+1-p_c(m,\theta)\big)>p_c(m,\theta).
		\end{equation*}
		Furthermore, the balanced parameters $\widetilde{\alpha}_{k,\text{bal}}(3/2,s,\theta)$ and $\widetilde{\alpha}_{k,\text{bal}}(m,0,\theta)$ should be introduced. If $s\in\left[0,1/2\right)$ and $\theta\in\left[0,1/2\right)$, we take the notation
		\begin{equation*}
		 \widetilde{\alpha}_{k,\text{bal}}(3/2,s,\theta):=\frac{9-12\theta+4s(2-\theta)+((7-6\theta)+2s(3-2\theta))(p_{k+1}+1)p_{k+2}-((2-6\theta)+2s)p_1p_2p_3}{2(p_1p_2p_3-1)}.
		\end{equation*}
		If $m\in\left[6/5,3/2\right)\text{ and }\theta\in\left[1/2,1\right]$, we denote
		\begin{equation*}
		 \widetilde{\alpha}_{k,\text{bal}}(m,0,\theta):=\frac{4m\theta+12\theta-3+(2m\theta+12\theta+3m-6)(p_{k+1}+1)p_{k+2}-(2m\theta-3m+3)p_1p_2p_3}{2(p_1p_2p_3-1)}.
		\end{equation*}
\end{enumerate}
\subsection{Philosophy of our approach}
Now, we explain our strategy to study the global (in time) existence of small data Sobolev solutions for the semi-linear Cauchy problem \eqref{fsdew001}.

Let us consider the family of linear parameter dependent Cauchy problems
\begin{equation}\label{parametherdep}
\left\{
\begin{aligned}
&u_{tt}-a^2\Delta u-\big(b^2-a^2\big)\nabla\divv u+(-\Delta)^{\theta}u_t=0,\quad &(t,x)\in[\tau,\infty)\times\mb{R}^3,\\
&(u,u_t)(\tau,x)=(u_0,u_1)(x),\quad &x\in\mb{R}^3.
\end{aligned}
\right.
\end{equation}
With the aim of studying the system \eqref{parametherdep}, we define $K_0=K_0(t,\tau,x)$, $K_1=K_1(t,\tau,x)$ as the fundamental solutions with data $(u_0,u_1)=(\delta_0,0)$ and $(u_0,u_1)=(0,\delta_0)$, respectively. Here, $\delta_0$ denotes the Dirac distribution in $x=0$ with respect to the spatial variables. Then, the solution $u=u(t,x)$ to the linear Cauchy problem \eqref{parametherdep} is given by
\begin{equation*}
u(t,x)=K_0(t,\tau,x)\ast_{(x)}u_0(x)+K_1(t,\tau,x)\ast_{(x)}u_1(x).
\end{equation*}
Next, by Duhamel's principle we see that
\begin{equation*}
u(t,x)=\int\nolimits_0^tK_1(t,\tau,x)\ast_{(x)}f(\tau,x)d\tau
\end{equation*}
is the solution to the inhomogeneous linear Cauchy problem
\begin{equation*}
\left\{
\begin{aligned}
&u_{tt}-a^2\Delta u-\big(b^2-a^2\big)\nabla\divv u+(-\Delta)^{\theta}u_t=f(t,x),\quad &(t,x)\in(0,\infty)\times\mb{R}^3,\\
&(u,u_t)(0,x)=(0,0),\quad &x\in\mb{R}^3.
\end{aligned}
\right.
\end{equation*}
We define on the family of complete spaces $\left\{X(T)\right\}_{T>0}$ the operator $N$ as follows:
\begin{equation*}
\begin{split}
N:U\in X(T)\longrightarrow NU(t,x):=\big(N_1U(t,x),N_2U(t,x),N_3U(t,x)\big)^{\mathrm{T}},
\end{split}
\end{equation*}
where for $k=1,2,3,$ we introduce
\begin{equation} \label{fixedpointformulation}
N_kU(t,x):= K_0(t,0,x)\ast_{(x)}U_0^{(k)}(x)+K_1(t,0,x)\ast_{(x)}U^{(k)}_1(x)+\int\nolimits_0^tK_1(t,\tau,x)\ast_{(x)}|U^{(k-1)}(\tau,x)|^{p_k}d\tau.
\end{equation}
The next inequalities play an essential role:
\begin{equation}\label{Improtant1}
\|NU\|_{X(T)}\lesssim\sum\limits_{k=1}^3\big\|\big(U^{(k)}_0,U^{(k)}_1\big)\big\|_{\ml{D}^s_{m,1}}+\sum\limits_{k=1}^3\|U\|_{X(T)}^{p_k},
\end{equation}
\begin{equation}\label{Improtant2}
\|NU-NV\|_{X(T)}\lesssim\|U-V\|_{X(T)}\sum\limits_{k=1}^3\big(\|U\|_{X(T)}^{p_k-1}+\|V\|_{X(T)}^{p_k-1}\big),
\end{equation}
uniformly with respect to $T\in[0,\infty)$. They mainly show that the mapping $N:X(T)\rightarrow X(T)$ is a contraction for small data. Then, according to Banach's fixed-point theorem, there exists a uniquely determined solution $U^*=U^*(t,x)$ to the semi-linear Cauchy problem \eqref{fsdew001} satisfying $NU^*=U^*\in X(T)$ for all positive $T$.

The key tools to prove \eqref{Improtant1} and \eqref{Improtant2} are Gagliardo-Nirenberg inequalities, the fractional chain rule, the fractional Leibniz rule and the fractional powers rules, which have been extensively and intensively discussed in Harmonic Analysis (cf. with Appendix \ref{toolfractional} or the book \cite{ReissigEbert2018}).

Additionally, because different power source nonlinearities have different influences on conditions for the global (in time) existence of solutions, we allow the {\it effect of the loss of decay}, in particular, in the case that one of the exponents $p_1,\,p_2,\,p_3$ is below the exponent $p_c(m,\theta)$ or the balanced parameter $p_{\text{bal}}(m,s,\theta)$. For this reason we take the derived energy estimates for the solutions to the linear model \eqref{linearproblem} with vanishing right-hand side and allow in the solution spaces some parameters describing the loss of decay.

We now state the strategy of the loss of decay. To prove the global (in time) existence of small data Sobolev solutions, the main difficulty is to estimate the integral in (\ref{fixedpointformulation}) over the interval $[0,t]$. We divide the interval $[0,t]$ in two sub-intervals $\left[0,t/2\right]$ and $\left[t/2,t\right]$. The difficulty is the estimate of the power nonlinearities in the norm of the solution space in each interval. If we allow to apply the Gagliardo-Nirenberg inequality, then there appear some relations including these parameters describing the loss of decay.

Here we take an example to show how to choose the suitable parameters describing the loss of decay. Let us consider the semi-linear model \eqref{fsdew001} with $\theta\in\left[1/2,1\right]$ and data belonging to $\ml{D}_{m,1}^0$  with $m\in\left[1,6/5\right)$, where exactly one exponent is not above the exponent $p_c(m,\theta)$. Without loss of generality we choose $1<p_{k_1}< p_c(m,\theta)$ and $p_{k_2},p_{k_3}>p_c(m,\theta)$, where the exponent $p_c(m,\theta)$ is defined by \eqref{cri01}. Let us choose the evolution space \eqref{sp} with the norm \eqref{norm1}. Our purpose is to prove the following estimates for $j+l=0,1$ with $j,l\in\mb{N}_0$:
\begin{equation}\label{es12}
	 (1+t)^{\frac{6-5m+2(j+l)m}{4m\theta}-g_{k_1}}\|\partial_t^j\nabla_x^lN_{k_1}U(t,\cdot)\|_{L^2}\lesssim\sum\limits_{k=1}^3\big\|\big(U_0^{(k)},U_1^{(k)}\big)\big\|_{\ml{D}_{m,1}^0}+\|U\|_{X(t)}^{p_{k_1}}.
	\end{equation}
	Applying the classical Gagliardo-Nirenberg inequality (see Proposition \ref{claGNineq}) we obtain
	\begin{equation*}
	\begin{split}
	\big\||U^{(k_3)}(\tau,x)|^{p_{k_1}}\big\|_{L^m}&\lesssim (1+\tau)^{-\frac{(3-m)p_{k_1}-3}{2m\theta}+g_{k_3}p_{k_1}}\|U\|_{X(\tau)}^{p_{k_1}},\\
	\big\||U^{(k_3)}(\tau,x)|^{p_{k_1}}\big\|_{L^2}&\lesssim (1+\tau)^{-\frac{2(3-m)p_{k_1}-3m}{4m\theta}+g_{k_3}p_{k_1}}\|U\|_{X(\tau)}^{p_{k_1}}.
	\end{split}
	\end{equation*}
	After using the derived $(L^2\cap L^m)$-$L^2$ estimates and $L^2$-$L^2$ estimates to the solution and its derivatives, the following estimates can be obtained:
	\begin{equation*}
	\begin{split}
	 &(1+t)^{\frac{6-5m+2(j+l)m}{4m\theta}-g_{k_1}}\|\partial_t^j\nabla_x^lN_{k_1}U(t,\cdot)\|_{L^2}\lesssim(1+t)^{-g_{k_1}}\sum\limits_{k=1}^3\big\|\big(U_0^{(k)},U_1^{(k)}\big)\big\|_{\ml{D}_{m,1}^0}\\
	 &\qquad\qquad\qquad+(1+t)^{-g_{k_1}}\|U\|_{X(t)}^{p_{k_1}}\Big(\int\nolimits_{0}^{t/2}(1+\tau)^{-\frac{(3-m)p_{k_1}-3}{2m\theta}+g_{k_3}p_{k_1}}d\tau+(1+t)^{1-\frac{(3-m)p_{k_1}-3}{2m\theta}+g_{k_3}p_{k_1}}\Big).
	\end{split}
	\end{equation*}
	Because of the assumption $1<p_{k_1}<p_c(m,\theta)$, the first integral over $\left[0,t/2\right]$ is not uniformly bounded for all $t>0$ because of
\[ -\frac{(3-m)p_{k_1}-3}{2m\theta}+g_{k_3}p_{k_1} >-1. \]
For this reason it holds
	\begin{equation*}
	\int\nolimits_{0}^{t/2}(1+\tau)^{-\frac{(3-m)p_{k_1}-3}{2m\theta}+g_{k_3}p_{k_1}}d\tau\lesssim(1+t)^{1-\frac{(3-m)p_{k_1}-3}{2m\theta}+g_{k_3}p_{k_1}}.
	\end{equation*}
	Thus, we can get
	\begin{equation*}
	 (1+t)^{\frac{6-5m+2(j+l)m}{4m\theta}-g_{k_1}}\|\partial_t^j\nabla_x^lN_{k_1}U(t,\cdot)\|_{L^2}\lesssim(1+t)^{-g_{k_1}}\sum\limits_{k=1}^3\big\|\big(U_0^{(k)},U_1^{(k)}\big)\big\|_{\ml{D}_{m,1}^0}+(1+t)^{1-g_{k_1}-\frac{(3-m)p_{k_1}-3}{2m\theta}+g_{k_3}p_{k_1}}\|U\|_{X(t)}^{p_{k_1}}.
	\end{equation*}
	Obviously, the non-negative parameters $g_{k_1}$ and $g_{k_3}$ describing the loss of decay should satisfy the following condition:
	\begin{equation*}
	1-g_{k_1}-\frac{(3-m)p_{k_1}-3}{2m\theta}+g_{k_3}p_{k_1}\leq0.
	\end{equation*}
	Providing that we choose the parameters 
	\begin{equation*}
	g_{k_1}=1-\frac{(3-m)p_{k_1}-3}{2m\theta}\,\,\,\,\text{and}\,\,\,\, g_{k_3}=0,
	\end{equation*}
	we can prove the desired estimate \eqref{es12}.
\begin{rem}
		It is not reasonable to compare the hereinafter proposed results with the results of \cite{Takeda2009}. In \cite{Takeda2009}, the authors proved results for the global (in time) existence of Sobolev solution to weakly coupled systems of damped wave equations with data belonging to the space $(W^{1,1}\cap W^{1,\infty})\times(L^1\cap L^{\infty})$. Moreover, their proof is based on $L^p$-$L^q$ estimates of fundamental solutions for the linear damped wave equation. What we do is to derive the global (in time) existence of solutions to weakly coupled systems for elastic waves with different damping mechanisms with data belonging to the space $(H^{s+1}\cap L^m)\times(H^s\cap L^m)$.
\end{rem}
\subsection{GESDS for models with $\theta\in\left[0,1/2\right)$}\label{0,1/2GESDS}
From Theorem \ref{enee} we know that the time-dependent coefficients in the energy estimates for solutions to the linear Cauchy problem \eqref{linearproblem} depend continuously on the parameters $\theta\in\left[0,1/2\right)$, $m\in[1,2)$ and $s\geq0$. In the following, we will choose the special cases $m=1$ and $m=3/2$ to show clearly and succinctly our strategy to prove results for the global (in time) existence of small data Sobolev solutions.

First, we recall some energy estimates for solutions to the linear Cauchy problem \eqref{linearproblem} (cf. with Theorem \ref{enee}). If the date belong to the space $\ml{D}_{1,1}^s$, that is , $\big(u_0^{(k)},u_1^{(k)}\big)\in (H^{s+1}\cap L^{1})\times (H^{s}\cap L^{1})$ for all $s\geq0$ and $k=1,2,3$, we have the following estimates:
	\begin{equation}\label{11119}
	\begin{split}
	\|u^{(k)}(t,\cdot)\|_{L^2}&\lesssim(1+t)^{-\rho_0(1,\theta)}\sum\limits_{k=1}^3\big\|\big(u^{(k)}_0,u_1^{(k)}\big)\big\|_{(H^1\cap L^{1})\times(L^2\cap L^{1})},\\
	 \||D|^{s+1}u^{(k)}(t,\cdot)\|_{L^2}+\||D|^su_t^{(k)}(t,\cdot)\|_{L^2}&\lesssim(1+t)^{-\rho_{s+1}(1,\theta)}\sum\limits_{k=1}^3\big\|\big(u^{(k)}_0,u_1^{(k)}\big)\big\|_{(H^{s+1}\cap L^{1})\times(H^{s}\cap L^{1})},
	\end{split}
	\end{equation}
	where
	\begin{equation}\label{111110}
	\rho_0(1,\theta)<\frac{3-4\theta}{4(1-\theta)}\,\,\,\,\text{and}\,\,\,\,\rho_{s+1}(1,\theta)<\frac{5-4\theta+2s}{4(1-\theta)}.
	\end{equation}
	If the date belong to the space $\ml{D}_{3/2,1}^s$, that is, $\big(u_0^{(k)},u_1^{(k)}\big)\in (H^{s+1}\cap L^{3/2})\times (H^{s}\cap L^{3/2})$ for all $s\geq0$ and $k=1,2,3$, we have the following estimates:
	\begin{equation*}
	\begin{split}
	\|u^{(k)}(t,\cdot)\|_{L^2}&\lesssim(1+t)^{1-\rho_1(3/2,\theta)}\sum\limits_{k=1}^3\big\|\big(u^{(k)}_0,u^{(k)}_1\big)\big\|_{(H^1\cap L^{3/2})\times (L^2\cap L^{3/2})},\\
	 \||D|^{s+1}u^{(k)}(t,\cdot)\|_{L^2}+\||D|^su_t^{(k)}(t,\cdot)\|_{L^2}&\lesssim(1+t)^{-\rho_{s+1}(3/2,\theta)}\sum\limits_{k=1}^3\big\|\big(u^{(k)}_0,u^{(k)}_1\big)\big\|_{(H^{s+1}\cap L^{3/2})\times (H^s\cap L^{3/2})},
	\end{split}
	\end{equation*}
	where
	\begin{equation*}
	\rho_1(3/2,\theta)<\frac{3-4\theta}{4(1-\theta)}\,\,\,\,\text{and}\,\,\,\,\rho_{s+1}(3/2,\theta)<\frac{3-4\theta+2s}{4(1-\theta)}.
	\end{equation*}
\subsubsection{Data from classical energy space with suitable regularity}
Because data belong to the spaces $\ml{D}^0_{1,1}$ or $\ml{D}^0_{3/2,1}$ in this part, we mainly use the classical Gagliardo-Nirenberg inequality to estimate nonlinearities in the $L^2$ norm and the $L^m$ norm $(m=1\text{ or  }m=3/2)$. The restriction of admissible parameters from the application of the classical Gagliardo-Nirenberg inequality implies the condition $p_k\in\left[2/m,3\right]$ for all $k=1,2,3$. For this reason, we observe that in the following theorem all exponents are above the exponent $p=2$ (see Remark \ref{onlyrem}).
\begin{thm}\label{GESDS01} Let us consider the semi-linear model \eqref{fsdew001} with $\theta\in [0,1/2)$. Let us assume $p_k\in(2,3]$ for $k=1,2,3$. Then, there exists a constant $\varepsilon_0>0$ such that for all $\big(U^{(k)}_0,U^{(k)}_1\big)\in\ml{D}_{1,1}^0$ with $\sum\limits_{k=1}^3\big\|\big(U^{(k)}_0,U^{(k)}_1\big)\big\|_{\ml{D}_{1,1}^0}\leq\varepsilon_0$ there exists a uniquely determined energy solution
	\begin{equation*}
	U\in\big(\mathcal{C}\big([0,\infty),H^1\big(\mb{R}^3\big)\big)\cap \mathcal{C}^1\big([0,\infty),L^2\big(\mb{R}^3\big)\big)\big)^3
	\end{equation*}
	to the Cauchy problem \eqref{fsdew001}. Moreover, the following estimates hold:
	\begin{equation*}
	\begin{split}
	\|U^{(k)}(t,\cdot)\|_{L^2}&\lesssim (1+t)^{-\rho_{0}(1,\theta)}\sum\limits_{k=1}^3\big\|\big(U^{(k)}_0,U^{(k)}_1\big)\big\|_{\ml{D}_{1,1}^0},\\
	\|\nabla_xU^{(k)}(t,\cdot)\|_{L^2}+\|U_t^{(k)}(t,\cdot)\|_{L^2}&\lesssim (1+t)^{-\rho_{1}(1,\theta)}\sum\limits_{k=1}^3\big\|\big(U^{(k)}_0,U^{(k)}_1\big)\big\|_{\ml{D}_{1,1}^0}.
	\end{split}
	\end{equation*}
\end{thm}
\begin{proof} For any $T>0$ let us introduce the evolution space
	\begin{equation}\label{sp}
	X(T):=\big(\mathcal{C}\big([0,T],H^1\big(\mb{R}^3\big)\big)\cap \mathcal{C}^1\big([0,T],L^2\big(\mb{R}^3\big)\big)\big)^3
	\end{equation}
	with the corresponding norm
	\begin{equation*}
	\|U\|_{X(T)}:=\sup\limits_{0\leq t\leq T}\Big(\sum\limits_{k=1}^3(1+t)^{\rho_0(1,\theta)}\|U^{(k)}(t,\cdot)\|_{L^2}+\sum\limits_{k=1}^3(1+t)^{\rho_1(1,\theta)}\big(\|\nabla_x U^{(k)}(t,\cdot)\|_{L^2}+\|U_t^{(k)}(t,\cdot)\|_{L^2}\big)\Big).
	\end{equation*}
	In the definition of the norm the weights $(1+t)^{\rho_0(1,\theta)}$ and $(1+t)^{\rho_1(1,\theta)}$ come from the decay estimates of solutions to the corresponding linear Cauchy problem \eqref{linearproblem} with data belonging to $\ml{D}_{1,1}^0$.\\
	Applying the classical Gagliardo-Nirenberg inequality we have
	\begin{equation*}
	\begin{split}
	 \||U^{(k-1)}(\tau,x)|^{p_k}\|_{L^m}&\lesssim\|U^{(k-1)}(\tau,\cdot)\|_{L^2}^{p_k(1-\beta_{0,1}(mp_k))}\|\nabla_xU^{(k-1)}(\tau,\cdot)\|_{L^2}^{p_k\beta_{0,1}(mp_k)}\\
	&\lesssim (1+\tau)^{-\rho_0(1,\theta)p_k+3(\frac{p_k}{2}-\frac{1}{m})(\rho_0(1,\theta)-\rho_1(1,\theta))}\|U\|_{X(\tau)}^{p_k},
	\end{split}
	\end{equation*}
	where $\beta_{0,1}(mp_k)=3(\frac{1}{2}-\frac{1}{mp_k})$ for $m=1$ and $m=2$. The restrictions from the application of the classical Gagliardo-Nirenberg inequality, i.e., $\beta_{0,1}(p_k),\beta_{0,1}(2p_k)\in[0,1]$, lead to $p_k\in[2,3]$ for all $k=1,2,3$.
	
	Now we apply on $[0,t]$ the derived $(L^2\cap L^1)$-$L^2$ estimates for the solution itself  to get
	\begin{equation*}
	\begin{split}
	(1+t)^{\rho_0(1,\theta)}&\|N_kU(t,\cdot)\|_{L^2}\lesssim \sum\limits_{k=1}^3\big\|\big(U^{(k)}_0,U^{(k)}_1\big)\big\|_{\ml{D}_{1,1}^0}+(1+t)^{\rho_0(1,\theta)}\int\nolimits_0^t(1+t-\tau)^{-\rho_{0}(1,\theta)}
\big\||U^{(k-1)}(\tau,x)|^{p_k}\big\|_{L^2\cap L^1}d\tau\\
	 &\lesssim\sum\limits_{k=1}^3\big\|\big(U^{(k)}_0,U^{(k)}_1\big)\big\|_{\ml{D}_{1,1}^0}+(1+t)^{\rho_0(1,\theta)}
\|U\|_{X(t)}^{p_k}\int\nolimits_0^t(1+t-\tau)^{-\rho_0(1,\theta)}(1+\tau)^{-\rho_0(1,\theta)p_k+3(\frac{p_k}{2}-1)(\rho_0(1,\theta)-\rho_1(1,\theta))}d\tau,
	\end{split}
	\end{equation*}
	where we use $\|U\|_{X(\tau)}\leq\|U\|_{X(t)}$ for any $0\leq\tau\leq t$. According to $(1+t-\tau)\approx(1+t)$ for any $\tau\in\left[0,t/2\right]$ and $(1+\tau)\approx(1+t)$ for any $\tau\in\left[t/2,t\right]$ we divide the interval $\left[0,t\right]$ into sub-intervals $\left[0,t/2\right]$ and $\left[t/2,t\right]$ to get
	\begin{equation*}
	\begin{split}
	 (1+t)^{\rho_0(1,\theta)}\|N_kU(t,\cdot)\|_{L^2}&\lesssim\sum\limits_{k=1}^3\big\|\big(U^{(k)}_0,U^{(k)}_1\big)\big\|_{\ml{D}_{1,1}^0}
+\|U\|_{X(t)}^{p_k}\int\nolimits_0^{t/2}(1+\tau)^{-\rho_0(1,\theta)p_k+3(\frac{p_k}{2}-1)(\rho_0(1,\theta)-\rho_1(1,\theta))}d\tau\\
	&\quad+(1+t)^{1-\rho_0(1,\theta)p_k+3(\frac{p_k}{2}-1)(\rho_0(1,\theta)-\rho_1(1,\theta))}\|U\|_{X(t)}^{p_k}.
	\end{split}
	\end{equation*}
	Here we used $\rho_0(1,\theta)<1$. Due to the assumption $p_k>2$ we may use $p_k>1+2/(3-2\theta)$ for all $k=1,2,3$. But then we have
	\begin{equation*}
	-\rho_0(1,\theta)p_k+3\big(\frac{p_k}{2}-1\big)(\rho_0(1,\theta)-\rho_1(1,\theta))<-1.
	\end{equation*}
	Therefore, it implies
	\begin{equation*}
	(1+t)^{\rho_0(1,\theta)}\|N_kU(t,\cdot)\|_{L^2}\lesssim\sum\limits_{k=1}^3\big\|\big(U^{(k)}_0,U^{(k)}_1\big)\big\|_{\ml{D}_{1,1}^0}+\|U\|_{X(t)}^{p_k}.
	\end{equation*}
	Similarly, we apply the derived $(L^2\cap L^1)$-$L^2$ estimates on $\left[0,t/2\right]$ and $L^2$-$L^2$ estimates on $\left[t/2,t\right]$ to get
	\begin{equation*}
	\begin{split}
	 (1+t)^{\rho_1(1,\theta)}\|\partial_t^j\nabla_x^lN_kU(t,\cdot)\|_{L^2}&\lesssim\sum\limits_{k=1}^3\big\|\big(U^{(k)}_0,U^{(k)}_1\big)\big\|_{\ml{D}_{1,1}^0}
+\|U\|_{X(t)}^{p_k}\int\nolimits_{0}^{t/2}(1+\tau)^{-\rho_0(1,\theta)p_k+3(\frac{p_k}{2}-1)(\rho_0(1,\theta)-\rho_1(1,\theta))}d\tau\\
	&\quad+(1+t)^{\rho_1(1,\theta)+1-\rho_0(1,\theta)p_k+\frac{3}{2}(p_k-1)(\rho_0(1,\theta)-\rho_1(1,\theta))}\|U\|_{X(t)}^{p_k}\\
	\end{split}
	\end{equation*}
	for $j=1,\,l=0$ and $j=0,\,l=1$. Thanks to the condition $\min\left\{p_1;p_2;p_3\right\}>2$ we have 
	\begin{equation*}
	\rho_1(1,\theta)+1-\rho_0(1,\theta)p_k+\frac{3}{2}(p_k-1)(\rho_0(1,\theta)-\rho_1(1,\theta))=\frac{3-2\theta}{2(1-\theta)}(2-p_k)+\epsilon(p_k-1)\leq0,
	\end{equation*}
	where $\epsilon$ is a sufficiently small positive constant. The sufficient small constant $\epsilon>0$ comes from the almost sharp energy estimates \eqref{11119}-\eqref{111110}, which can be written as follows: 
		\begin{equation*}
		\begin{split}
		\|u^{(k)}(t,\cdot)\|_{L^2}&\lesssim(1+t)^{-\frac{3-4\theta}{4(1-\theta)}+\epsilon}\sum\limits_{k=1}^3\big\|\big(u^{(k)}_0,u_1^{(k)}\big)\big\|_{(H^1\cap L^{1})\times(L^2\cap L^{1})},\\
		\||D|^{s+1}u^{(k)}(t,\cdot)\|_{L^2}+\||D|^su_t^{(k)}(t,\cdot)\|_{L^2}&\lesssim(1+t)^{-\frac{5-4\theta+2s}{4(1-\theta)}+\epsilon}\sum\limits_{k=1}^3\big\|\big(u^{(k)}_0,u_1^{(k)}\big)\big\|_{(H^{s+1}\cap L^{1})\times(H^{s}\cap L^{1})}.
		\end{split}
		\end{equation*}
	 Thus, the estimates for derivatives hold for all $k=1,2,3$. In this way we obtain for $j+l=1$ and $j,l\in\mb{N}_0$
	\begin{equation*}
	 (1+t)^{\rho_1(1,\theta)}\|\partial_t^j\nabla_x^lN_kU(t,\cdot)\|_{L^2}\lesssim\sum\limits_{k=1}^3\big\|\big(U^{(k)}_0,U^{(k)}_1\big)\big\|_{\ml{D}_{1,1}^0}+\|U\|_{X(t)}^{p_k}.
	\end{equation*}
	Next, we derive the Lipschitz condition by remarking that
	\begin{equation*}
	\|\partial_t^j\nabla_x^l(N_kU-N_kV)(t,\cdot)\|_{L^2}=\Big\|\partial_t^j\nabla_x^l\int\nolimits_0^tK_1(t-\tau,0,x)
	\ast_{(x)}\left(|U^{(k-1)}(\tau,x)|^{p_k}-|V^{(k-1)}(\tau,x)|^{p_k}\right)d\tau\Big\|_{L^2}.
	\end{equation*}
	Thanks to H\"older's inequality we get for $m=1,2$ the estimates
	\begin{equation*}
	\Big\||U^{(k-1)}(\tau,x)|^{p_k}-|V^{(k-1)}(\tau,x)|^{p_k}\Big\|_{L^m}
	\lesssim\|U^{(k-1)}(\tau,\cdot)-V^{(k-1)}(\tau,\cdot)\|_{L^{mp_k}}\Big(\|U^{(k-1)}(\tau,\cdot)\|_{L^{mp_k}}^{p_k-1}
	+\|V^{(k-1)}(\tau,\cdot)\|_{L^{mp_k}}^{p_k-1}\Big).
	\end{equation*}
	As above, we can use the classical Gagliardo-Nirenberg inequality again to estimate
	\begin{equation*}
	\|U^{(k-1)}(\tau,\cdot)-V^{(k-1)}(\tau,\cdot)\|_{L^{mp_k}},\,\,\,\,\|U^{(k-1)}(\tau,\cdot)\|_{L^{mp_k}},\,\,\,\, \|V^{(k-1)}(\tau,\cdot)\|_{L^{mp_k}},
	\end{equation*}
	with $m=1,2$ and we can conclude \eqref{Improtant2}. The proof is complete.
\end{proof}
\begin{rem}\label{onlyrem} Again, in Theorem \ref{GESDS01}, we only allow exponents $p_1,p_2,p_3$ are larger than the exponent $p=2$. If we assume that there exists a number $k_1=1,2,3$ such that $1<p_{k_1}< 2$, the condition $p_{k_1}\in[2,3]$ from the application of the classical Gagliardo-Nirenberg inequality leads to the empty set for the exponent $p_{k_1}$.
\end{rem}
For data belonging to the classical energy space with an additional regularity $L^{3/2}$, we can obtain a larger admissible range of exponents $p_1,p_2,p_3$ because of the condition $p_k\in \left[4/3,3\right]$ for all $k=1,2,3$ coming from the application of the classical Gagliardo-Nirenberg inequality. \\ We observe the following three different cases:
\begin{enumerate}
	\item the orders of power nonlinearities are above the balanced exponent $p_{\text{bal}}(3/2,0,\theta)$;
	\item only one exponent is below or equal to the balanced exponent $p_{\text{bal}}(3/2,0,\theta)$;
	\item two exponents are below or equal to the balanced exponent $p_{\text{bal}}(3/2,0,\theta)$.
\end{enumerate} 
\begin{rem}\label{varepsilon1}
If in the Cases (ii) or (iii) of the following theorem some of the exponents $p_{k_j}=p_{\text{bal}}(3/2,0,\theta)$, then we can choose the parameters $g_{k_j}$ in the loss of decay as $g_{k_j}=\varepsilon_1$ with a sufficiently small constant $\varepsilon_1>0$ to avoid a logarithmic term $\log(e+t)$ in the estimate of the integral over $\left[0,t/2\right]$. Then, we can follow the proof of Theorem \ref{GESDS03} without any new difficulties.
\end{rem}
\begin{thm}\label{GESDS03} Let us consider the semi-linear model \eqref{fsdew001} with $\theta\in\left[0,1/2\right)$. Let us assume $p_k\in\left[4/3,3\right]$ for $k=1,2,3,$ such that
	\begin{flalign}\label{000exp01}
	&\text{(i)}\quad\text{there are no other restrictions when}\,\,\,\,\min\left\{p_1;p_2;p_3\right\}>p_{\text{bal}}(3/2,0,\theta);&
	\end{flalign}
	\begin{flalign}\label{000exp02}
	 &\text{(ii)}\quad\alpha_{k_1,\text{bal}}(3/2,0,\theta)<3/2\,\,\,\,\text{when}\,\,\,\,1<p_{k_1}<p_{\text{bal}}(3/2,0,\theta)\,\,\,\,\text{and}\,\,\,\,p_{k_2},p_{k_3}>p_{\text{bal}}(3/2,0,\theta);&
	\end{flalign}
	\begin{flalign}\label{000exp03}
	&\text{(iii)}\quad\widetilde{\alpha}_{k_1,\text{bal}}(3/2,0,\theta)<3/2\,\,\,\,\text{when}\,\,\,\,1<p_{k_1},p_{k_2}< p_{\text{bal}}(3/2,0,\theta)\,\,\,\,\text{and}\,\,\,\,p_{k_3}>p_{\text{bal}}(3/2,0,\theta).&
	\end{flalign}
	Then, there exists a constant $\varepsilon_0>0$ such that for all $\big(U^{(k)}_0,U^{(k)}_1\big)\in\ml{D}_{3/2,1}^0$ with $\sum\limits_{k=1}^3\big\|\big(U^{(k)}_0,U^{(k)}_1\big)\big\|_{\ml{D}_{3/2,1}^0}\leq\varepsilon_0$ there is a uniquely determined energy solution
	\begin{equation*}
	U\in\big(\mathcal{C}\big([0,\infty),H^1\big(\mb{R}^3\big)\big)\cap \mathcal{C}^1\big([0,\infty),L^2\big(\mb{R}^3\big)\big)\big)^3
	\end{equation*}
	to the Cauchy problem \eqref{fsdew001}. Moreover, the following estimates hold:
	\begin{equation*}
	\begin{split}
	\|U^{(k)}(t,\cdot)\|_{L^2}&\lesssim (1+t)^{1-\rho_1(3/2,\theta)+g_k}\sum\limits_{k=1}^3\big\|\big(U^{(k)}_0,U^{(k)}_1\big)\big\|_{\ml{D}_{3/2,1}^0},\\
	\|\nabla_xU^{(k)}(t,\cdot)\|_{L^2}+\|U_t^{(k)}(t,\cdot)\|_{L^2}&\lesssim (1+t)^{-\rho_1(3/2,\theta)+g_k}\sum\limits_{k=1}^3\big\|\big(U^{(k)}_0,U^{(k)}_1\big)\big\|_{\ml{D}_{3/2,1}^0},
	\end{split}
	\end{equation*}
	where in the decay functions the numbers $g_k$ are chosen in the following way:\begin{enumerate}
	\item $g_{k}=0$ for $k=1,2,3,$ when $p_1,p_2,p_3$ satisfy the condition \eqref{000exp01};
	\item $g_{k_1}=3+\big(\frac{1}{4(1-\theta)}-\frac{3}{2}\big)p_{k_1}$ and $g_{k_2}=g_{k_3}=0$, when $p_{k_1},p_{k_2},p_{k_3}$ satisfy the condition \eqref{000exp02};
	\item $g_{k_1}=3+\big(\frac{1}{4(1-\theta)}-\frac{3}{2}\big)p_{k_1}$, $g_{k_2}=3+\big(\frac{3}{2}+\frac{1}{4(1-\theta)}\big)p_{k_2}+\big(\frac{1}{4(1-\theta)}-\frac{3}{2}\big)p_{k_1}p_{k_2}$ and $g_{k_3}=0$, when $p_{k_1},p_{k_2},p_{k_3}$ satisfy the condition \eqref{000exp03}.
\end{enumerate}
\end{thm}
\begin{proof}
	For any $T>0$, let us introduce the evolution space \eqref{sp} with the following norm:
	\begin{equation}\label{norm1}
	\begin{split}
	\|U\|_{X(T)}:=\sup\limits_{0\leq t\leq T}&\Big(\sum\limits_{k=1}^3(1+t)^{-1+\rho_1(3/2,\theta)-g_k}\|U^{(k)}(t,\cdot)\|_{L^2}+\sum\limits_{k=1}^3(1+t)^{\rho_1(3/2,\theta)-g_k}\big(\|\nabla_x U^{(k)}(t,\cdot)\|_{L^2}+\|U_t^{(k)}(t,\cdot)\|_{L^2}\big)\Big).
	\end{split}
	\end{equation}
	The classical Gagliardo-Nirenberg inequality implies
	\begin{equation*}
	\begin{split}
	\big\||U^{(k-1)}(\tau,x)|^{p_k}\big\|_{L^{3/2}}&\lesssim (1+\tau)^{(1-\rho_1(3/2,\theta))p_k-3(\frac{p_k}{2}-\frac{2}{3})+g_{k-1}p_k}\|U\|_{X(\tau)}^{p_k},\\
	\big\||U^{(k-1)}(\tau,x)|^{p_k}\big\|_{L^2}&\lesssim (1+\tau)^{(1-\rho_1(3/2,\theta))p_k-3(\frac{p_k}{2}-\frac{1}{2})+g_{k-1}p_k}\|U\|_{X(\tau)}^{p_k}.
	\end{split}
	\end{equation*}
	The restriction of the parameters from applying the Gagliardo-Nirenberg inequality leads to $p_k\in[4/3,3]$ for all $k=1,2,3$.
	
	Firstly, the application of the derived $(L^2\cap L^{3/2})$-$L^2$ estimate leads on the interval $[0,t]$ to
	\begin{equation*}
	\begin{split}
	 &(1+t)^{-1+\rho_1(3/2,\theta)-g_k}\|N_kU(t,\cdot)\|_{L^2}\lesssim(1+t)^{-g_k}\sum\limits_{k=1}^3\big\|\big(U_0^{(k)},U_1^{(k)}\big)\big\|_{\ml{D}_{3/2,1}^0}\\
	 &\qquad\qquad\qquad\qquad+(1+t)^{-1+\rho_1(3/2,\theta)-g_k}\|U\|_{X(t)}^{p_k}\int\nolimits_0^{t}(1+t-\tau)^{1-\rho_1(3/2,\theta)}(1+\tau)^{(1-\rho_1(3/2,\theta))p_k
-3(\frac{p_k}{2}-\frac{2}{3})+g_{k-1}p_k}d\tau.
	\end{split}
	\end{equation*}
	After dividing the interval $[0,t]$ into sub-intervals $\left[0,t/2\right]$ and $\left[t/2,t\right]$ it follows
	\begin{equation*}
	\begin{split}
	&(1+t)^{-1+\rho_1(3/2,\theta)-g_k}\|N_kU(t,\cdot)\|_{L^2}\\
	 &\qquad\qquad\lesssim(1+t)^{-g_k}\sum\limits_{k=1}^3\big\|\big(U_0^{(k)},U_1^{(k)}\big)\big\|_{\ml{D}_{3/2,1}^0}+(1+t)^{-g_k}\|U\|_{X(t)}^{p_k}\int\nolimits_0^{t/2}
(1+\tau)^{(1-\rho_1(3/2,\theta))p_k-3(\frac{p_k}{2}-\frac{2}{3})+g_{k-1}p_k}d\tau\\
	&\qquad\qquad\quad+(1+t)^{1-g_k+(1-\rho_1(3/2,\theta))p_k-3(\frac{p_k}{2}-\frac{2}{3})+g_{k-1}p_k}\|U\|_{X(t)}^{p_k},
	\end{split}
	\end{equation*}
	where we use the following estimate:
	\begin{equation*}
	\int\nolimits_{t/2}^t(1+t-\tau)^{1-\rho_1(3/2,\theta)}d\tau\lesssim(1+t)^{2-\rho_1(3/2,\theta)} \,\,\,\,\mbox{and}\,\,\,\,\rho_1(3/2,\theta) <1.
	\end{equation*}
	In the same way, we may obtain the following estimates for the derivatives ($j+l=1$):
	\begin{equation*}
	\begin{split}
	 (1+t)^{\rho_1(3/2,\theta)-g_k}\|\partial_t^j\nabla_x^lN_kU(t,\cdot)\|_{L^2}\lesssim&(1+t)^{-g_k}\sum\limits_{k=1}^3\big\|\big(U_0^{(k)},U_1^{(k)}\big)\big\|_{\ml{D}_{3/2,1}^0}\\
	&+(1+t)^{-g_k}\|U\|_{X(t)}^{p_k}\int\nolimits_0^{t/2}(1+\tau)^{(1-\rho_1(3/2,\theta))p_k-3(\frac{p_k}{2}-\frac{2}{3})+g_{k-1}p_k}d\tau\\
	&+(1+t)^{1-g_k+(1-\rho_1(3/2,\theta))p_k-3(\frac{p_k}{2}-\frac{2}{3})+g_{k-1}p_k}\|U\|_{X(t)}^{p_k}.
	\end{split}
	\end{equation*}
	Summarizing the above estimates we may conclude
	\begin{equation}\label{01/2estimate}
	\begin{split}
	 &(1+t)^{(l+j)-1+\rho_1(3/2,\theta)-g_k}\|\partial_t^j\nabla_x^lN_kU(t,\cdot)\|_{L^2}\lesssim(1+t)^{-g_k}\sum\limits_{k=1}^3\big\|\big(U_0^{(k)},U_1^{(k)}\big)\big\|_{\ml{D}_{3/2,1}^0}\\
	 &\qquad\qquad+(1+t)^{-g_k}\|U\|_{X(t)}^{p_k}\Big(\int\nolimits_0^{t/2}(1+\tau)^{2-(1/2+\rho_1(3/2,\theta))p_k+g_{k-1}p_k}d\tau+(1+t)^{3-(1/2+\rho_1(3/2,\theta))p_k+g_{k-1}p_k}\Big).
	\end{split}
	\end{equation}
	for all $j+l=0,1$ with $j,l\in\mb{N}_0$. In order to prove
	\begin{equation}\label{01/2aim}
	 (1+t)^{(l+j)-1+\rho_1(3/2,\theta)-g_k}\|\partial_t^j\nabla_x^lN_kU(t,\cdot)\|_{L^2}\lesssim\sum\limits_{k=1}^3\big\|\big(U_0^{(k)},U_1^{(k)}\big)\big\|_{\ml{D}_{3/2,1}^0}+\|U\|_{X(t)}^{p_k}
	\end{equation}
	we have to distinguish between three cases.\\
	\\
	\emph{Case 1} $\quad$ We assume the condition (\ref{000exp01}), that is, $\min\left\{p_1;p_2;p_3\right\}>p_{\text{bal}}(3/2,0,\theta)$.\medskip
	
	\noindent Here, the orders of power nonlinearities are above $p_{\text{bal}}(3/2,0,\theta)$ and it allows to assume no loss of decay. Thus, we choose the parameters $g_1=g_2=g_3=0$ and we get from the estimate \eqref{01/2estimate}
	\begin{equation*}
	\begin{split}
	 (1+t)^{(l+j)-1+\rho_1(3/2,\theta)}\|\partial_t^j\nabla_x^lN_kU(t,\cdot)\|_{L^2}\lesssim&\sum\limits_{k=1}^3\big\|\big(U_0^{(k)},U_1^{(k)}\big)\big\|_{\ml{D}_{3/2,1}^0}+\|U\|_{X(t)}^{p_k}\int\nolimits_0^{t/2}(1+\tau)^{2-(1/2+\rho_1(3/2,\theta))p_k}d\tau\\
	&+(1+t)^{3-(1/2+\rho_1(3/2,\theta))p_k}\|U\|_{X(t)}^{p_k},
	\end{split}
	\end{equation*}
	where $k=1,2,3$. If we guarantee
	\begin{equation*}
	\min\left\{p_1;p_2;p_3\right\}>p_{\text{bal}}(3/2,0,\theta)\,\,\,\,\text{for}\,\,\,\,\theta\in\left[0,1/2\right),
	\end{equation*}
	then we can prove
	\begin{equation*}
	(1+\tau)^{2-(1/2+\rho_1(3/2,\theta))p_k}\in L^1[0,\infty)\,\,\,\,\text{and}\,\,\,\, (1+t)^{3-(1/2+\rho_1(3/2,\theta))p_k}\lesssim 1.
	\end{equation*}
	Thus, the desired estimate \eqref{01/2aim} holds for all $k=1,2,3$.\\
	\\
	\emph{Case 2} $\quad$ We assume the condition (\ref{000exp02}), that is, $1<p_{k_1}< p_{\text{bal}}(3/2,0,\theta)$ and $p_{k_2},p_{k_3}>p_{\text{bal}}(3/2,0,\theta)$.\medskip
	
	\noindent In this case, where only two exponents $p_{k_2},p_{k_3}$ are above $p_{\text{bal}}(3/2,0,\theta)$ we shall prove a global (in time) existence result with a loss of decay in one component of the solution under the additional condition
	\begin{equation}\label{condddddd}
	\alpha_{k_1,\text{bal}}(3/2,0,\theta)=\frac{9-12\theta+(7-6\theta)p_{k_2}-(2-6\theta)p_{k_1}p_{k_2}}{2(p_{k_1}p_{k_2}-1)}<\frac{3}{2}.
	\end{equation}
	We choose the parameters describing the loss of decay as $g_{k_1}=3+\big(\frac{1}{4(1-\theta)}-\frac{3}{2}\big)p_{k_1}$ and $g_{k_2}=g_{k_3}=0$. The assumption $1<p_{k_1}< p_{\text{bal}}(3/2,0,\theta)$ leads to $g_{k_1}>0$. Moreover, the condition \eqref{condddddd} is equivalent to
	\begin{equation}\label{9}
	12(1-\theta)+(7-6\theta)p_{k_2}-(5-6\theta)p_{k_1}p_{k_2}<0.
	\end{equation}
	If we assume $p_{k_2}>p_{\text{bal}}(3/2,0,\theta)$, the condition \eqref{9} is valid.
	
	Taking account of \eqref{01/2estimate} when $k=k_1$ and using the estimate
	\begin{equation*}
	\int\nolimits_0^{t/2}(1+\tau)^{2-(1/2+\rho_1(3/2,\theta))p_{k_1}}d\tau\lesssim(1+t)^{3-(1/2+\rho_1(3/2,\theta))p_{k_1}}
	\end{equation*}
	because $2-(1/2+\rho_1(3/2,\theta))p_{k_1}>-1$, we may conclude
	\begin{equation}\label{8}
	\begin{split}
	 (1+t)^{(l+j)-1+\rho_1(3/2,\theta)-g_{k_1}}\|\partial_t^j\nabla_x^lN_{k_1}U(t,\cdot)\|_{L^2}\lesssim\sum\limits_{k=1}^3\big\|\big(U_0^{(k)},U_1^{(k)}\big)\big\|_{\ml{D}_{3/2,1}^0}+(1+t)^{-g_{k_1}+3-(1/2+\rho_1(3/2,\theta))p_{k_1}}\|U\|_{X(t)}^{p_{k_1}}.
	\end{split}
	\end{equation}
	So, our desired estimate \eqref{01/2aim} has been proved for $k=k_1$.\\
	Considering the case $k=k_2$, we obtain the following estimates:
	\begin{equation*}
	\begin{split}
	 (1+t)^{(l+j)-1+\rho_1(3/2,\theta)}\|\partial_t^j\nabla_x^lN_{k_2}U(t,\cdot)\|_{L^2}\lesssim&\sum\limits_{k=1}^3\big\|\big(U_0^{(k)},U_1^{(k)}\big)\big\|_{\ml{D}_{3/2,1}^0}+\|U\|_{X(t)}^{p_{k_2}}\int\nolimits_0^{t/2}(1+\tau)^{2-(1/2+\rho_1(3/2,\theta))p_{k_2}+g_{k_1}p_{k_2}}d\tau\\
	&+(1+t)^{3-(1/2+\rho_1(3/2,\theta))p_{k_2}+g_{k_1}p_{k_2}}\|U\|_{X(t)}^{p_{k_2}}.
	\end{split}
	\end{equation*}
Taking account of \eqref{9} the following inequality holds:
	\begin{equation*}
	2-(1/2+\rho_1(3/2,\theta))p_{k_2}+g_{k_1}p_{k_2}<-1.
	\end{equation*}
	Thus, it completes the estimate \eqref{01/2aim} for $k=k_2$.\\
	Finally, we consider the case $k=k_3$. By the same procedure we treated \emph{Case 1}, we immediately obtain
	\begin{equation*}
	\begin{split}
	&(1+t)^{(l+j)-1+\rho_1(3/2,\theta)}\|\partial_t^j\nabla_x^lN_{k_3}U(t,\cdot)\|_{L^2}\\
	 &\qquad\quad\lesssim\sum\limits_{k=1}^3\big\|\big(U_0^{(k)},U_1^{(k)}\big)\big\|_{\ml{D}_{3/2,1}^0}+\|U\|_{X(t)}^{p_{k_3}}\Big(\int\nolimits_0^{t/2}(1+\tau)^{2-(1/2+\rho_1(3/2,\theta))p_{k_3}}d\tau+(1+t)^{3-(1/2+\rho_1(3/2,\theta))p_{k_3}}\Big)\\
	&\qquad\quad\lesssim\sum\limits_{k=1}^3\big\|\big(U_0^{(k)},U_1^{(k)}\big)\big\|_{\ml{D}_{3/2,1}^0}+\|U\|_{X(t)}^{p_{k_3}},
	\end{split}
	\end{equation*}
	where we use our assumption $p_{k_3}>p_{\text{bal}}(3/2,0,\theta)$.\\
	\\
	\emph{Case 3} $\quad$ We assume the condition (\ref{000exp03}), that is, $1<p_{k_1},p_{k_2}< p_{\text{bal}}(3/2,0,\theta)$ and $p_{k_3}>p_{\text{bal}}(3/2,0,\theta)$.\medskip
	
	\noindent Here, there exists only one exponent $p_{k_3}$ larger than $p_{\text{bal}}(3/2,0,\theta)$. Hence, we shall prove a global (in time) existence of small data solutions result with a loss of decay for two components of the solution under the intersectional condition
	\begin{equation}\label{condiddd}
	\widetilde{\alpha}_{k_1,\text{bal}}(3/2,0,\theta)= \frac{9-12\theta+(7-6\theta)(p_{k_2}+1)p_{k_3}-(2-6\theta)p_1p_2p_3}{2(p_1p_2p_3-1)}<\frac{3}{2}.
	\end{equation}
	We choose the parameters as follows: \[ g_{k_1}=3+\Big(\frac{1}{4(1-\theta)}-\frac{3}{2}\Big)p_{k_1}, \,\,\,\, g_{k_2}=3+\Big(\frac{3}{2}+\frac{1}{4(1-\theta)}\Big)p_{k_2}+\Big(\frac{1}{4(1-\theta)}-\frac{3}{2}\Big)p_{k_1}p_{k_2},\,\,\,\,\mbox{and}\,\, \,\,g_{k_3}=0.\] With the help of the assumption $1<p_{k_1},p_{k_2}< p_{\text{bal}}(3/2,0,\theta)$, we have $g_{k_1}>0$ and $g_{k_2}>0$. The condition \eqref{condiddd} can be rewritten as
	\begin{equation}\label{11}
	12(1-\theta)+(7-6\theta)(p_{k_2}+1)p_{k_3}-(5-6\theta)p_{k_1}p_{k_2}p_{k_3}<0.
	\end{equation}
	If we assume $p_{k_3}>p_{\text{bal}}(3/2,0,\theta)$, the above condition is valid.\\
	When $k=k_1$ in the estimate \eqref{01/2estimate}, we can get the same estimates as \eqref{8} in \emph{Case 2}. Choosing $k=k_2$, we apply the same method as we did in \emph{Case 2} to get
	\begin{equation*}
	\begin{split}
	&(1+t)^{(l+j)-1+\rho_1(3/2,\theta)-g_{k_2}}\|\partial_t^j\nabla_x^lN_{k_2}U(t,\cdot)\|_{L^2}\\
	 &\qquad\qquad\lesssim\sum\limits_{k=1}^3\big\|\big(U_0^{(k)},U_1^{(k)}\big)\big\|_{\ml{D}_{3/2,1}^0}+(1+t)^{-g_{k_2}}\|U\|_{X(t)}^{p_{k_2}}\int\nolimits_0^{t/2}(1+\tau)^{2-(1/2+\rho_1(3/2,\theta))p_{k_2}+g_{k_1}p_{k_2}}d\tau\\
	&\qquad\qquad\quad+(1+t)^{-g_{k_2}+3-(1/2+\rho_1(3/2,\theta))p_{k_2}+g_{k_1}p_{k_2}}\|U\|_{X(t)}^{p_{k_2}}\\
	&\qquad\qquad\lesssim\sum\limits_{k=1}^3\big\|\big(U_0^{(k)},U_1^{(k)}\big)\big\|_{\ml{D}_{3/2,1}^0}+\|U\|_{X(t)}^{p_{k_2}},
	\end{split}
	\end{equation*}
	where the choice of $g_{k_2}$ implies the inequality \[ g_{k_2} >3-(1/2+\rho_1(3/2,\theta))p_{k_2}+g_{k_1}p_{k_2}.\]  But this gives us for all $t \geq 0$ a uniformly bounded estimate from the above inequality. Finally,
	let us take $k=k_3$ in the estimate \eqref{01/2estimate}. In this way we get
	\begin{equation*}
	\begin{split}
	 (1+t)^{(l+j)-1+\rho_1(3/2,\theta)}\|\partial_t^j\nabla_x^lN_{k_3}U(t,\cdot)\|_{L^2}\lesssim&\sum\limits_{k=1}^3\big\|\big(U_0^{(k)},U_1^{(k)}\big)\big\|_{\ml{D}_{3/2,1}^0}+\|U\|_{X(t)}^{p_{k_3}}\int\nolimits_0^{t/2}(1+\tau)^{2-(1/2+\rho_1(3/2,\theta))p_{k_3}+g_{k_2}p_{k_3}}d\tau\\
	&+(1+t)^{3-(1/2+\rho_1(3/2,\theta))p_{k_3}+g_{k_2}p_{k_3}}\|U\|_{X(t)}^{p_{k_3}}.
	\end{split}
	\end{equation*}
From the condition \eqref{11} it follows
	\begin{equation*}
	2-(1/2+\rho_1(3/2,\theta))p_{k_3}+g_{k_2}p_{k_3}<-1.
	\end{equation*}
	So, we immediately obtain our desired estimate \eqref{01/2aim} for $k=k_3$.\\
	All in all, the estimate \eqref{01/2aim} has been completed for all the cases.
	
	Lastly, similar as in the proof of Theorem \ref{GESDS01}, we may apply H\"older's inequality and the Gagliardo-Nirenberg inequality to prove
	\begin{equation*}
	 (1+t)^{(l+j)-1+\rho_1(3/2,\theta)}\|\partial_t^j\nabla_x^lN_{k}(U(t,\cdot)-V(t,\cdot))\|_{L^2}\lesssim\|U-V\|_{X(t)}\big(\|U\|_{X(t)}^{p_k-1}+\|V\|_{X(t)}^{p_k-1}\big)
	\end{equation*}
	for all $j+l=0,1$ with $j,l\in\mb{N}_0$ and $k=1,2,3$ in all cases. So, the proof is complete.
\end{proof}

\subsubsection{Data from energy space with suitable higher regularity}
Now, we are interested in studying the global (in time) existence of small data energy solutions possessing energies of higher-order. As we know, the parameter $\min \{p_1;p_2;p_3\}$ is bounded to below by the regularity parameter $s+1$.
\begin{thm}\label{HlL1nonlinearythm}
	Let us consider the semi-linear model \eqref{fsdew001} with $\theta\in\left[0,1/2\right)$. Let us choose
	\begin{equation*}
	\begin{aligned}
	&1+\lceil s\rceil<\min\left\{p_1;p_2;p_3\right\}\leq\max\left\{p_1;p_2;p_3\right\}\leq 1+2/(1-2s)&\,\,\,\,\text{if}\,\,\,\,&s\in\left(0,1/2\right),\\
	&1+\lceil s\rceil<\min\left\{p_1;p_2;p_3\right\}\leq\max\left\{p_1;p_2;p_3\right\}<\infty&\,\,\,\,\text{if}\,\,\,\,&s\in\left[1/2,\infty\right).\\
	\end{aligned}
	\end{equation*}
	Then, there exists a constant $\varepsilon_0>0$ such that for all $\big(U^{(k)}_0,U^{(k)}_1\big)\in\ml{D}^s_{1,1}$ with $\sum\limits_{k=1}^3\big\|\big(U^{(k)}_0,U^{(k)}_1\big)\big\|_{\ml{D}^s_{1,1}}\leq\varepsilon_0$ there exists a uniquely determined energy solution
	\begin{equation*}
	U\in \big(\mathcal{C}\big([0,\infty),H^{s+1}\big(\mb{R}^3\big)\big)\cap \mathcal{C}^1\big([0,\infty),H^s\big(\mb{R}^3\big)\big)\big)^3
	\end{equation*}
	to the Cauchy problem \eqref{fsdew001}. Moreover, the following estimates hold:
	\begin{equation*}
	\begin{split}
	\|U^{(k)}(t,\cdot)\|_{L^2}&\lesssim (1+t)^{-\rho_{0}(1,\theta)} \sum\limits_{k=1}^3\big\|\big(U^{(k)}_0,U^{(k)}_1\big)\big\|_{\ml{D}^s_{1,1}},\\
	\|U_t^{(k)}(t,\cdot)\|_{L^2}&\lesssim (1+t)^{-\rho_{1}(1,\theta)} \sum\limits_{k=1}^3\big\|\big(U^{(k)}_0,U^{(k)}_1\big)\big\|_{\ml{D}^s_{1,1}},\\
	\||D|U^{(k)}(t,\cdot)\|_{\dot{H}^{s}}+\|U_t^{(k)}(t,\cdot)\|_{\dot{H}^{s}}&\lesssim (1+t)^{-\rho_{s+1}(1,\theta)}\sum\limits_{k=1}^3\big\|\big(U^{(k)}_0,U^{(k)}_1\big)\big\|_{\ml{D}^s_{1,1}}.
	\end{split}
	\end{equation*}
	
\end{thm}
\begin{rem} In Theorem \ref{HlL1nonlinearythm}, our purpose is to weaken the upper bound for the exponents $p_{1},p_{2},p_{3}$ in comparison to the condition $\max\left\{p_1;p_2;p_3\right\}\leq3$ in Theorem \ref{GESDS01}. After posing Bessel potential spaces with higher regularity for data, that is, $\big(U^{(k)}_0,U^{(k)}_1\big)\in\ml{D}^{s}_{1,1}$ with $0<s<1/2$, the largest admissible range for the exponents can be obtained when $s-1/2\rightarrow-0$. To increase the upper bound for $\max\left\{p_1;p_2;p_3\right\}$, we suppose more regularity for data.
\end{rem}
\begin{proof}  For any $T>0$ we define the complete evolution space
	\begin{equation}\label{spp}
	X(T):=\big(\mathcal{C}\big([0,T],H^{s+1}\big(\mb{R}^3\big)\big)\cap \mathcal{C}^1\big([0,T],H^s\big(\mb{R}^3\big)\big)\big)^3
	\end{equation}
	with the corresponding norm
	\begin{equation}\label{norm}
	\begin{split}
	\|U\|_{X(T)}:=\sup\limits_{0\leq t\leq T}&\Big(\sum\limits_{k=1}^3(1+t)^{\rho_{0}(1,\theta)}\|U^{(k)}(t,\cdot)\|_{L^2}+\sum\limits_{k=1}^3(1+t)^{\rho_1(1,\theta)}\|U^{(k)}_{t}(t,\cdot)\|_{L^2}\\
	&\quad\quad\quad\quad\quad +\sum\limits_{k=1}^3(1+t)^{\rho_{s+1}(1,\theta)}\big(\||D|U^{(k)}(t,\cdot)\|_{\dot{H}^{s}}+\|U^{(k)}_{t}(t,\cdot)\|_{\dot{H}^s}\big)\Big).
	\end{split}
	\end{equation}
	We shall estimate the norms $\|\partial_t^jN_kU(t,\cdot)\|_{L^2}$, $\|\partial_t^jN_kU(t,\cdot)\|_{\dot{H}^{s+1-j}}$ for $j=0,1$.
	Firstly, using the derived $(L^2\cap L^1)$-$L^2$ estimates n the interval $[0,t]$ we have
	\begin{equation*}
	\begin{split}
	 (1+t)^{\rho_{0}(1,\theta)}\|N_kU(t,\cdot)\|_{L^2}\lesssim\sum\limits_{k=1}^3\big\|\big(U^{(k)}_0,U^{(k)}_1\big)\big\|_{\ml{D}^0_{1,1}}+(1+t)^{\rho_{0}(1,\theta)}\int\nolimits_0^{t}(1+t-\tau)^{-\rho_{0}(1,\theta)}\big\||U^{(k-1)}(\tau,x)|^{p_k}\big\|_{L^2\cap L^1}d\tau.
	\end{split}
	\end{equation*}
	Next, the application of the $(L^2\cap L^1)$-$L^2$ estimates on $\left[0,t/2\right]$ and $L^2$-$L^2$ estimates on $\left[t/2,t\right]$ yields
	\begin{equation*}
	\begin{split}
	 (1+t)^{\rho_{1}(1,\theta)}\|\partial_tN_kU(t,\cdot)\|_{L^2}\lesssim&\sum\limits_{k=1}^3\big\|\big(U^{(k)}_0,U^{(k)}_1\big)\big\|_{\ml{D}^0_{1,1}}+\int\nolimits_0^{t/2}\big\||U^{(k-1)}(\tau,x)|^{p_k}\big\|_{L^2\cap L^1}d\tau\\
	 &+(1+t)^{\rho_{1}(1,\theta)}\int\nolimits_{t/2}^t\big\||U^{(k-1)}(\tau,x)|^{p_k}\big\|_{L^2}d\tau.
	\end{split}
	\end{equation*}
	The fractional Gagliardo-Nirenberg inequality implies for $m=1,2$ the estimates
	\begin{equation*}
	\begin{split}
	\big\||U^{(k-1)}(\tau,x)|^{p_k}\big\|_{L^{m}}&\lesssim\|U^{(k-1)}(\tau,\cdot)\|_{L^2}^{(1-\beta_{0,s+1}(mp_k))p_k}
	\|U^{(k-1)}(\tau,\cdot)\|_{\dot{H}^{s+1}}^{\beta_{0,s+1}(mp_k)p_k}
	\\
	&\lesssim(1+\tau)^{-\rho_0(1,\theta)p_k+\frac{3}{s+1}(\frac{p_k}{2}-\frac{1}{m})(\rho_0(1,\theta)-\rho_{s+1}(1,\theta))}\|U\|_{X(\tau)}^{p_k},
	\end{split}
	\end{equation*}
	where \[ \beta_{0,s+1}(mp_k)=\frac{3}{s+1}\Big(\frac{1}{2}-\frac{1}{mp_k}\Big)\in[0,1].\] This implies $2\leq p_k\leq\frac{3}{1-2s}$ for $0<s<1/2$ or $2\leq p_k<\infty$ for $s\geq1/2$.\\ Hence, repeating the same procedure as in the proof of Theorem \ref{GESDS01} and applying $\min\left\{p_1;p_2;p_3\right\}>2$ we may conclude
	\begin{equation*}
	(1+t)^{\rho_j(1,\theta)}\|\partial_t^j N_kU(t,\cdot)\|_{L^2}\lesssim\sum\limits_{k=1}^3\big\|\big(U^{(k)}_0,U^{(k)}_1\big)\big\|_{\ml{D}^0_{1,1}}+\|U\|_{X(t)}^{p_k}
	\end{equation*}
	for $j=0,1$ and $k=1,2,3$.
	
	Now, we estimate $N_kU(t,\cdot)$ in the $\dot{H}^{s+1}$ norm and $\partial_tN_kU(t,\cdot)$ in the $\dot{H}^{s}$ norm. The application of the derived $(\dot{H}^s\cap L^1)$-$\dot{H}^s$ estimates on $\left[0,t/2\right]$ and $\dot{H}^s$-$\dot{H}^s$ estimates on $\left[t/2,t\right]$ gives immediately
	\begin{equation*}
	\begin{split}
	(1+t)^{\rho_{s+1}(1,\theta)}\|\partial_t^jN_kU(t,\cdot)\|_{\dot{H}^{s+1-j}}\lesssim&  \sum\limits_{k=1}^3\big\|\big(U^{(k)}_0,U^{(k)}_1\big)\big\|_{\ml{D}_{1,1}^s}+\int\nolimits_0^{t/2}\big\||U^{(k-1)}(\tau,x)|^{p_k}\big\|_{\dot{H}^s\cap L^1}d\tau\\
	&+(1+t)^{\rho_{s+1}(1,\theta)}\int\nolimits_{t/2}^t\big\||U^{(k-1)}(\tau,x)|^{p_k}\big\|_{\dot{H}^s}d\tau.
	\end{split}
	\end{equation*}
	We should estimate the nonlinear term in the $L^1$ norm and the $\dot{H}^s$ norm, respectively. The estimate of the $L^1$ norm we can easily get from the fractional Gagliardo-Nirenberg inequality. In this way we obtain
	\begin{equation*}
	 \big\||U^{(k-1)}(\tau,x)|^{p_k}\big\|_{L^1}\lesssim(1+\tau)^{-\rho_0(1,\theta)p_k+\frac{3}{s+1}(\frac{p_k}{2}-1)(\rho_0(1,\theta)-\rho_{s+1}(1,\theta))}\|U\|_{X(\tau)}^{p_k}.
	\end{equation*}
	We apply more tools from Harmonic Analysis to estimate the $\dot{H}^s$ norm of $|U^{(k-1)}(\tau,x)|^{p_k}$. Applying the fractional chain rule from Proposition \ref{fractionalchainrule}
we have
	\begin{equation}\label{cont}
	\big\||U^{(k-1)}(\tau,x)|^{p_k}\big\|_{\dot{H}^s}\lesssim\|U^{(k-1)}(\tau,\cdot)\|_{L^{q_1}}^{p_k-1}\|U^{(k-1)}(\tau,\cdot)\|_{\dot{H}^{s}_{q_2}},
	\end{equation}
	where $\frac{p_k-1}{q_1}+\frac{1}{q_2}=\frac{1}{2}$ and $p_k>\lceil s\rceil$ for $k=1,2,3$. Moreover, the fractional Gagliardo-Nirenberg-type inequality comes into play again
	to conclude \begin{equation*}
	\begin{split}
	 \|U^{(k-1)}(\tau,\cdot)\|_{L^{q_1}}&\lesssim\|U^{(k-1)}(\tau,\cdot)\|_{L^2}^{1-\beta_{0,s+1}(q_1)}\|U^{(k-1)}(\tau,\cdot)\|_{\dot{H}^{s+1}}^{\beta_{0,s+1}(q_1)},\\
	 \|U^{(k-1)}(\tau,\cdot)\|_{\dot{H}^{s}_{q_2}}&\lesssim\|U^{(k-1)}(\tau,\cdot)\|_{L^2}^{1-\beta_{s,s+1}(q_2)}\|U^{(k-1)}(\tau,\cdot)\|_{\dot{H}^{s+1}}^{\beta_{s,s+1}(q_2)},
	\end{split}
	\end{equation*}
	where \[ \beta_{0,s+1}(q_1)=\frac{3}{s+1}\Big(\frac{1}{2}-\frac{1}{q_1}\Big)\in[0,1]\quad\mbox{and}\quad \beta_{s,s+1}(q_2)=\frac{3}{s+1}\Big(\frac{1}{2}-\frac{1}{q_2}+\frac{s}{3}\Big)\in\Big[\frac{s}{s+1},1\Big].\] The existence of parameters $q_1$ and $q_2$ will be discussed in Appendix \ref{appendix2}. Combining with what we discussed above we arrive at the estimate
	\begin{equation}\label{conttt}
	\begin{split}
	\big\||U^{(k-1)}(\tau,x)|^{p_k}\big\|_{\dot{H}^s}&\lesssim\|U^{(k-1)}(\tau,\cdot)\|_{L^2}^{p_k-\frac{3}{s+1}(\frac{p_k-1}{2}+\frac{s}{3})}
	\|U^{(k-1)}(\tau,\cdot)\|_{\dot{H}^{s+1}}^{\frac{3}{s+1}(\frac{p_k-1}{2}+\frac{s}{3})}\\
	&\lesssim(1+\tau)^{-\big(p_k-\frac{3}{s+1}\big(\frac{p_k-1}{2}+\frac{s}{3}\big)\big)\rho_0(1,\theta)-\frac{3}{s+1}\big(\frac{p_k-1}{2}+\frac{s}{3}\big)
		\rho_{s+1}(1,\theta)}\|U\|_{X(\tau)}^{p_k}.
	\end{split}
	\end{equation}
	Thus, by the condition $\min\left\{p_{1};p_{2};p_{3}\right\}>1+\lceil s\rceil\geq2$ and the estimate 
		\begin{equation*}
	\big\||U^{(k-1)}(\tau,x)|^{p_k}\big\|_{\dot{H}^s\cap L^1}\lesssim(1+\tau)^{-\rho_0(1,\theta)p_k+\frac{3}{s+1}(\frac{p_k}{2}-1)(\rho_0(1,\theta)-\rho_{s+1}(1,\theta))}\|U\|_{X(\tau)}^{p_k}
	\end{equation*}
	we get
	\begin{equation*}
	\begin{split}
	 (1+t)^{\rho_{s+1}(1,\theta)}\|\partial_t^jN_kU(t,\cdot)\|_{\dot{H}^{s+1-j}}&\lesssim\big\|\big(U^{(k)}_0,U^{(k)}_1\big)\big\|_{\ml{D}_{1,1}^s}+\|U\|_{X(t)}^{p_k}\int\nolimits_0^{t/2}(1+\tau)^{-\rho_0(1,\theta)p_k+\frac{3}{s+1}(\frac{p_k}{2}-1)(\rho_0(1,\theta)-\rho_{s+1}(1,\theta))}d\tau\\
	 &\quad+(1+t)^{1+\rho_{s+1}(1,\theta)-\big(p_k-\frac{3}{s+1}\big(\frac{p_k-1}{2}+\frac{s}{3}\big)\big)\rho_0(1,\theta)-\frac{3}{s+1}\big(\frac{p_k-1}{2}+\frac{s}{3}\big)
		\rho_{s+1}(1,\theta)}\|U\|_{X(t)}^{p_k}\\
	&\lesssim\big\|\big(U^{(k)}_0,U^{(k)}_1\big)\big\|_{\ml{D}_{1,1}^s}+\|U\|_{X(t)}^{p_k}.
	\end{split}
	\end{equation*}
	Summarizing all derived inequalities we get \eqref{Improtant1}.
	
	The last step is to derive the Lipschitz condition. The application of H\"older's inequality and the fractional Gagliardo-Nirenberg inequality yields for $j=0,1$
	\begin{equation*}
	(1+t)^{\rho_{j}(1,\theta)}\|\partial_t^j(N_kU-N_kV)(t,\cdot)\|_{L^2}\lesssim \|U-V\|_{X(t)}\big(\|U\|_{X(t)}^{p_{k}-1}+\|V\|_{X(t)}^{p_k-1}\big).
	\end{equation*}
	In the following we will show how to estimate \[ \|\partial_t^j(N_kU-N_kV)(t,\cdot)\|_{\dot{H}^{s+1-j}}\,\,\,\,\mbox{for} \,\,\,\,j=0,1.\] We apply the derived $(\dot{H}^s\cap L^1)$-$\dot{H}^s$ estimates and $\dot{H}^s$-$\dot{H}^s$ estimates again to conclude
	\begin{equation*}
	\begin{split}
	(1+t)^{\rho_{s+1}(1,\theta)}\|\partial_t^j(N_kU-N_kV)(t,\cdot)\|_{\dot{H}^{s+1-j}}\lesssim& \int\nolimits_0^{t/2}\big\||U^{(k-1)}(\tau,x)|^{p_k}-|V^{(k-1)}(\tau,x)|^{p_k}\big\|_{\dot{H}^s\cap L^1}d\tau\\
	&+(1+t)^{\rho_{s+1}(1,\theta)}\int\nolimits_{t/2}^t\big\||U^{(k-1)}(\tau,x)|^{p_k}-|V^{(k-1)}(\tau,x)|^{p_k}\big\|_{\dot{H}^s}d\tau.
	\end{split}
	\end{equation*}
	To estimate the $L^1$ norm the fractional Gagliardo-Nirenberg inequality implies
	\begin{equation*}
	\begin{split}
	 \big\||U^{(k-1)}(\tau,x)|^{p_k}-|V^{(k-1)}(\tau,x)|^{p_k}\big\|_{L^1}\lesssim&(1+\tau)^{-\rho_0(1,\theta)p_k+\frac{3}{s+1}(\frac{p_k}{2}-1)(\rho_0(1,\theta)-\rho_{s+1}(1,\theta))}
\|U-V\|_{X(\tau)}\big(\|U\|_{X(\tau)}^{p_k-1}+\|V\|_{X(\tau)}^{p_k-1}\big).
	\end{split}
	\end{equation*}
	Using the relation $\frac{d}{dx_i}|x|^{p_k}=p_k|x|^{p_k-2}x_i$ and setting $G(U^{(k-1)})=U^{(k-1)}|U^{(k-1)}|^{p_k-2}$ we have
	\begin{equation*}
	\begin{split}
	|U^{(k-1)}(\tau,x)|^{p_k}-|V^{(k-1)}(\tau,x)|^{p_k}=p_k\int\nolimits_0^1\big(U^{(k-1)}(\tau,x)-V^{(k-1)}(\tau,x)\big) G\big(\nu U^{(k-1)}(\tau,x)+(1-\nu) V^{(k-1)}(\tau,x)\big)d\nu.
	\end{split}
	\end{equation*}
	Therefore, Minkowski's inequality and the fractional Leibniz rule from Proposition \ref{fractionleibnizrule} show that
	\begin{equation*}
	\begin{split}
	 &\big\||U^{(k-1)}(\tau,x)|^{p_k}-|V^{(k-1)}(\tau,x)|^{p_k}\big\|_{\dot{H}^s}\lesssim\int\nolimits_0^1\big\|\big(U^{(k-1)}(\tau,\cdot)-V^{(k-1)}(\tau,\cdot)\big)G\big(\nu U^{(k-1)}(\tau,\cdot)+(1-\nu) V^{(k-1)}(\tau,\cdot)\big)\big\|_{\dot{H}^s}d\nu\\
	& \qquad \lesssim\int\nolimits_0^1\|U^{(k-1)}(\tau,\cdot)-V^{(k-1)}(\tau,\cdot)\|_{\dot{H}^{s}_{r_1}}\big\|G\big(\nu U^{(k-1)}(\tau,\cdot)+(1-\nu) V^{(k-1)}(\tau,\cdot)\big)\big\|_{L^{r_2}}d\nu\\
	&\qquad \quad+\int\nolimits_0^1\|U^{(k-1)}(\tau,\cdot)-V^{(k-1)}(\tau,\cdot)\|_{L^{r_3}}\big\|G\big(\nu U^{(k-1)}(\tau,\cdot)+(1-\nu) V^{(k-1)}(\tau,\cdot)\big)\big\|_{\dot{H}^{s}_{r_4}}d\nu,
	\end{split}
	\end{equation*}
	where $\frac{1}{r_1}+\frac{1}{r_2}=\frac{1}{r_3}+\frac{1}{r_4}=\frac{1}{2}$. Taking account of the first term on right-hand side we notice that
	\begin{equation*}
	\int\nolimits_0^1\big\|G\big(\nu U^{(k-1)}(\tau,\cdot)+(1-\nu) V^{(k-1)}(\tau,\cdot)\big)\big\|_{L^{r_2}}d\nu\lesssim \|U^{(k-1)}(\tau,\cdot)\|_{L^{r_2(p_k-1)}}^{p_k-1}+\|V^{(k-1)}(\tau,\cdot)\|_{L^{r_2(p_k-1)}}^{p_k-1}.
	\end{equation*}
	Actually, we use the fractional Gagliardo-Nirenberg inequality and the fractional chain rule to get the following inequalities:
	\begin{equation*}
	\begin{split}
	 \|U^{(k-1)}(\tau,\cdot)-V^{(k-1)}(\tau,\cdot)\|_{\dot{H}^{s}_{r_1}}&\lesssim\|U^{(k-1)}(\tau,\cdot)-V^{(k-1)}(\tau,\cdot)\|_{L^2}^{1-\beta_{s,s+1}(r_1)}\|U^{(k-1)}(\tau,\cdot)-V^{(k-1)}(\tau,\cdot)\|_{\dot{H}^{s+1}}^{\beta_{s,s+1}(r_1)},\\
	 \|U^{(k-1)}(\tau,\cdot)-V^{(k-1)}(\tau,\cdot)\|_{L^{r_3}}&\lesssim\|U^{(k-1)}(\tau,\cdot)-V^{(k-1)}(\tau,\cdot)\|_{L^2}^{1-\beta_{0,s+1}(r_3)}\|U^{(k-1)}(\tau,\cdot)-V^{(k-1)}(\tau,\cdot)\|_{\dot{H}^{s+1}}^{\beta_{0,s+1}(r_3)},\\
	 \|U^{(k-1)}(\tau,\cdot)\|_{L^{r_2(p_k-1)}}&\lesssim\|U^{(k-1)}(\tau,\cdot)\|_{L^2}^{1-\beta_{0,s+1}(r_2(p_k-1))}\|U^{(k-1)}(\tau,\cdot)\|_{\dot{H}^{s+1}}^{\beta_{0,s+1}(r_2(p_k-1))},\\
	 \|V^{(k-1)}(\tau,\cdot)\|_{L^{r_2(p_k-1)}}&\lesssim\|V^{(k-1)}(\tau,\cdot)\|_{L^2}^{1-\beta_{0,s+1}(r_2(p_k-1))}\|V^{(k-1)}(\tau,\cdot)\|_{\dot{H}^{s+1}}^{\beta_{0,s+1}(r_2(p_k-1))},\\
	\end{split}
	\end{equation*}
	and
	\begin{equation*}
	\begin{split}
	&\big\|G\big(\nu U^{(k-1)}(\tau,\cdot)+(1-\nu)V^{(k-1)}(\tau,\cdot)\big)\big\|_{\dot{H}^{s}_{r_4}}\\
	&\quad\quad\quad\quad\lesssim\|\nu U^{(k-1)}(\tau,\cdot)+(1-\nu)V^{(k-1)}(\tau,\cdot)\|_{L^{r_5}}^{p_k-2}\|\nu U^{(k-1)}(\tau,\cdot)+(1-\nu)V^{(k-1)}(\tau,\cdot)\|_{\dot{H}^{s}_{r_6}}\\
	&\quad\quad\quad\quad\lesssim\big(\|U^{(k-1)}(\tau,\cdot)\|_{L^2}+\|V^{(k-1)}(\tau,\cdot)\|_{L^2}\big)^{(1-\beta_{0,s+1}(r_5))(p_k-2)+1-\beta_{s,s+1}(r_6)}\\
	 &\quad\quad\quad\quad\quad\quad\times\big(\|U^{(k-1)}(\tau,\cdot)\|_{\dot{H}^{s+1}}+\|V^{(k-1)}(\tau,\cdot)\|_{\dot{H}^{s+1}}\big)^{\beta_{0,s+1}(r_5)(p_k-2)+\beta_{s,s+1}(r_6)},
	\end{split}
	\end{equation*}
	where the conditions for the parameters are
	\begin{equation*}
	\begin{aligned}
	\beta_{0,s+1}(r)&=\frac{3}{s+1}\Big(\frac{1}{2}-\frac{1}{r}\Big)\in[0,1]&\,\,\,\,&\text{for}\,\,\,\,r=r_2(p_k-1),r_3,r_5,\\
	\beta_{s,s+1}(r)&=\frac{3}{s+1}\Big(\frac{1}{2}-\frac{1}{r}+\frac{s}{3}\Big)\in\Big[\frac{s}{s+1},1\Big]&\quad&\text{for}\,r=r_1,r_6,\\
	\frac{1}{r_4}&=\frac{p_k-2}{r_5}+\frac{1}{r_6},
	\end{aligned}
	\end{equation*}
	for $\min\left\{p_1;p_2;p_3\right\}>\lceil s\rceil+1$. The existences of parameters $r_1,\dots,r_6$ is discussed in Appendix \ref{appendix2}. Straight-forward computations lead to
	\begin{equation*}
	\begin{split}
	\big\||U^{(k-1)}(\tau,x)|^{p_k}-|V^{(k-1)}(\tau,x)|^{p_k}\big\|_{\dot{H}^s}
	\lesssim&(1+\tau)^{-\big(p_k-\frac{3}{s+1}\big(\frac{p_k-1}{2}+\frac{s}{3}\big)\big)\rho_0(1,\theta)-\frac{3}{s+1}\big(\frac{p_k-1}{2}+\frac{s}{3}\big)
		\rho_{s+1}(1,\theta)}\\
	&\times\|U-V\|_{X(\tau)}\big(\|U\|_{X(\tau)}^{p_k-1}+\|V\|_{X(\tau)}^{p_k-1}\big).
	\end{split}
	\end{equation*}
	Summarizing all estimates allows to conclude \eqref{Improtant2}. This completes the proof.
\end{proof}
\begin{rem}
	Again, in Theorem \ref{HlL1nonlinearythm}, we only expect the exponents $p_1,p_2,p_3$ above the exponent $p=2$. If we would assume $1<p_{k_1}\leq2$, the admissible set for the exponent $p_{k_1}$ will be empty. The reason is that to derive the Lipschitz condition, we apply the fractional chain rule and the fractional Leibniz rule. Therefore, we propose the condition $\min\left\{p_1;p_2;p_3\right\}>1+\lceil s\rceil$.
\end{rem}
Finally, let us say a few things about the case $s>3/2$, where we suppose that data belong to $\ml{D}_{1,1}^s$. Applying the fractional chain rule from Proposition \ref{fractionalchainrule} would imply the admissible range for the exponents $p_1,p_2,p_3$ by the condition
\begin{equation*}
\min\left\{p_1;p_2;p_3\right\}>1+\lceil s\rceil.
\end{equation*}
Applying instead fractional powers rules from Proposition \ref{fractionalpowersrule} the last condition can be relaxed to $\min\left\{p_1;p_2;p_3\right\}>1+s$.
This is explained in the next result.
\begin{thm}\label{embthm}Let us consider the semi-linear model \eqref{fsdew001} with $\theta\in\left[0,1/2\right)$ and $s>3/2$. Let us choose
	\begin{equation*}
	1+s<\min\left\{p_1;p_2;p_3\right\}.
	\end{equation*}
	Then, there exists a constant $\varepsilon_0>0$ such that for all $\big(U_0^{(k)},U_1^{(k)}\big)\in\ml{D}_{1,1}^s$ with $\sum\limits_{k=1}^3\big\|\big(U_0^{(k)},U_1^{(k)}\big)\big\|_{\ml{D}_{1,1}^s}\leq\varepsilon_0$ there is a uniquely determined energy solution
	\begin{equation*}
	U\in \big(\mathcal{C}\big([0,\infty),H^{s+1}\big(\mb{R}^3\big)\big)\cap \mathcal{C}^1\big([0,\infty),H^s\big(\mb{R}^3\big)\big)\big)^3
	\end{equation*}
	to the Cauchy problem \eqref{fsdew001}. Moreover, the estimates for the solutions are the same as in Theorem \ref{HlL1nonlinearythm}.
\end{thm}
\begin{proof}
	Firstly, we define the evolution space and its norm as in \eqref{spp} and \eqref{norm}, respectively. Discussing the global (in time) existence of solutions with large regular data, we may use the fractional powers rules \cite{RunstSickel1996} instead of the fractional chain rule and the fractional Leibniz rule. More precisely, the estimates of $|U^{(k-1)}(\tau,x)|^{p_k}$ in the $\dot{H}^s$ norm should be changed. If we take $\min\left\{p_1;p_2;p_3\right\}>s$, then we have the following estimate:
	\begin{equation*}
	\begin{split}
	\big\||U^{(k-1)}(\tau,x)|^{p_k}\big\|_{\dot{H}^s}&\lesssim\|U^{(k-1)}(\tau,\cdot)\|_{\dot{H}^s}\|U^{(k-1)}(\tau,\cdot)\|_{L^{\infty}}^{p_k-1}\\
	&\lesssim\|U^{(k-1)}(\tau,\cdot)\|_{\dot{H}^s}^{p_k}+\|U^{(k-1)}(\tau,\cdot)\|_{\dot{H}^{s^*}}^{p_k-1}\|U^{(k-1)}(\tau,\cdot)\|_{\dot{H}^s},
	\end{split}
	\end{equation*}
	where we apply Proposition \ref{supplement} with $0<2s^*<3<2s$.
	
	Using the fractional Gagliardo-Nirenberg inequality again shows
	\begin{equation*}
	 \big\||U^{(k-1)}(\tau,x)|^{p_k}\big\|_{\dot{H}^s}\lesssim\|U^{(k-1)}(\tau,\cdot)\|_{L^2}^{\frac{p_k}{s+1}}\|U^{(k-1)}(\tau,\cdot)\|_{\dot{H}^{s+1}}^{\frac{p_ks}{s+1}}+\|U^{(k-1)}(\tau,\cdot)\|_{L^2}^{p_k-\frac{s^*}{s+1}p_k+\frac{s^*-s}{s+1}}\|U^{(k-1)}(\tau,\cdot)\|_{\dot{H}^{s+1}}^{\frac{s^*}{s+1}p_k-\frac{s^*-s}{s+1}}.
	\end{equation*}
	So, we get
	\begin{equation*}
	 \big\||U^{(k-1)}(\tau,x)|^{p_k}\big\|_{\dot{H}^s}\lesssim(1+\tau)^{-\rho_0(1,\theta)p_k-\big(\frac{s^*}{s+1}p_k-\frac{s^*-s}{s+1}\big)(\rho_{s+1}(1,\theta)-\rho_0(1,\theta))}\|U\|_{X(\tau)}^{p_k}.
	\end{equation*}
	Now we choose $s^*=3/2-\epsilon$ with an arbitrary small constant $\epsilon>0$. Therefore, in order to obtain \eqref{Improtant1}, the exponents $p_1,p_2,p_3$ should satisfy
	\begin{equation}\label{stronger}
	\max\{2;s\}<\min\left\{p_1;p_2;p_3\right\}.
	\end{equation}
	To verify the Lipschitz condition \eqref{Improtant2}, taking $\min\left\{p_1;p_2;p_3\right\}>1+s$, we have
	\begin{equation*}
	\begin{split}
	&\big\||U^{(k-1)}(\tau,\cdot)|^{p_k}-|V^{(k-1)}(\tau,\cdot)|^{p_k}\big\|_{\dot{H}^s}\\
	&\quad\lesssim\|U^{(k-1)}(\tau,\cdot)-V^{(k-1)}(\tau,\cdot)\|_{\dot{H}^s}\int\nolimits_0^1\big\|G\big(\nu U^{(k-1)}(\tau,\cdot)+(1-\nu)V^{(k-1)}(\tau,\cdot)\big)\big\|_{L^{\infty}}d\nu\\
	&\quad\quad+\|U^{(k-1)}(\tau,\cdot)-V^{(k-1)}(\tau,\cdot)\|_{L^{\infty}}\int\nolimits_0^1\big\|G\big(\nu U^{(k-1)}(\tau,\cdot)+(1-\nu)V^{(k-1)}(\tau,\cdot)\big)\big\|_{\dot{H}^s}d\nu\\
	&\quad\lesssim\|U^{(k-1)}(\tau,\cdot)-V^{(k-1)}(\tau,\cdot)\|_{\dot{H}^s}\int\nolimits_0^1\|\nu U^{(k-1)}(\tau,\cdot)+(1-\nu)V^{(k-1)}(\tau,\cdot)\|_{L^{\infty}}^{p_k-1}d\nu\\
	&\quad\quad+\|U^{(k-1)}(\tau,\cdot)-V^{(k-1)}(\tau,\cdot)\|_{L^{\infty}}\int\nolimits_0^1\|\nu U^{(k-1)}(\tau,\cdot)+(1-\nu)V^{(k-1)}(\tau,\cdot)\|_{\dot{H}^s}\|\nu U^{(k-1)}(\tau,\cdot)+(1-\nu)V^{(k-1)}(\tau,\cdot)\|_{L^{\infty}}^{p_k-2}d\nu.
	\end{split}
	\end{equation*}
	After applying Proposition \ref{supplement} again, we can conclude \eqref{Improtant2} with the condition
	\begin{equation}\label{stronger1}
	\max\left\{2;1+s\right\}<\min\left\{p_1;p_2;p_3\right\}.
	\end{equation}
	The proof is completed.
\end{proof}

\subsection{GESDS for models with $\theta\in\left[1/2,1\right]$}\label{1/2,1GESDS}
Before stating our main theorems, we recall the following energy estimates for solutions to the linear Cauchy problem \eqref{linearproblem} with data belonging to $\ml{D}_{m,1}^s$, that is, $(u_0^{(k)},u_1^{(k)})\in(H^{s+1}\cap L^m)\times(H^s\cap L^m)$ for all $k=1,2,3$ (see Theorem \ref{enee}):
\begin{equation*}
\|u^{(k)}(t,\cdot)\|_{L^2}\lesssim\sum\limits_{k=1}^3\big\|\big(u_0^{(k)},u_1^{(k)}\big)\big\|_{(H^1\cap L^m)\times (L^2\cap L^m)}\times\left\{\begin{aligned}
&(1+t)^{-\frac{6-5m}{4m\theta}}&\text{if}\,\,\,\,&m\in\left[1,6/5\right),\\
&(1+t)^{1-\frac{6-3m}{4m\theta}}&\text{if}\,\,\,\,&m\in\left[6/5,2\right),
\end{aligned}
\right.
\end{equation*}
and some estimates for the derivatives with $m\in[1,2)$ and $s\geq0$
\begin{equation*}
\||D|^{s+1}u^{(k)}(t,\cdot)\|_{L^2}+\||D|^su_t^{(k)}(t,\cdot)\|_{L^2}\lesssim(1+t)^{-\frac{6-3m+2sm}{4m\theta}}\sum\limits_{k=1}^3\big\|\big(u_0^{(k)},u_1^{(k)}\big)\big\|_{(H^{s+1}\cap L^m)\times (H^s\cap L^m)}.
\end{equation*}
\subsubsection{Data from classical energy space with suitable regularity}
In this section, we mainly study the global (in time) existence of energy solutions with small data having an additional regularity $L^m$ for $m\in\left[1,3/2\right)$. The main reason of the suitable choice of regularity $m\in\left[1,3/2\right)$ will be explained in Remark \ref{remm}.

Firstly, we state our result for $m\in\left[1,6/5\right)$. As explained in Remark \ref{varepsilon1}, if the exponents $p_{k_j}=p_{c}(m,\theta)$ in one of the Cases (ii) or (iii) for some $j=1,2,3$, then we can choose the parameters $g_{k_j}$ describing the loss of decay as $g_{k_j}=\varepsilon_1$ with a sufficiently small constant $\varepsilon_1>0$.
\begin{thm}\label{GESDS02} Let us consider the semi-linear model \eqref{fsdew001} with $\theta\in\left[1/2,1\right]$ and choose $m\in\left[1,6/5\right)$. Let us assume $p_k\in\left[2/m,3\right]$ for $k=1,2,3,$ such that
	\begin{flalign}\label{0exp01}
	&\text{(i)}\quad\text{there are no other restrictions when}\,\,\,\,\min\left\{p_1;p_2;p_3\right\}>p_c(m,\theta);&
	\end{flalign}
	\begin{flalign}\label{0exp02}
	&\text{(ii)}\quad\alpha_{k_1}(m,\theta)<3/2\,\,\,\,\text{when}\,\,\,\,1<p_{k_1}< p_c(m,\theta)\,\,\,\,\text{and}\,\,\,\,p_{k_2},p_{k_3}>p_c(m,\theta);&
	\end{flalign}
	\begin{flalign}\label{0exp03}
	&\text{(iii)}\quad\widetilde{\alpha}_{k_1}(m,\theta)<3/2\,\,\,\,\text{when}\,\,\,\,1<p_{k_1},p_{k_2}< p_c(m,\theta)\,\,\,\,\text{and}\,\,\,\,p_{k_3}>p_c(m,\theta).&
	\end{flalign}
	Then, there exists a constant $\varepsilon_0>0$ such that for all $\big(U^{(k)}_0,U^{(k)}_1\big)\in\ml{D}_{m,1}^0$ with $\sum\limits_{k=1}^3\big\|\big(U^{(k)}_0,U^{(k)}_1\big)\big\|_{\ml{D}_{m,1}^0}\leq\varepsilon_0$ there is a uniquely determined energy solution
	\begin{equation*}
	U\in\big(\mathcal{C}\big([0,\infty),H^1\big(\mb{R}^3\big)\big)\cap \mathcal{C}^1\big([0,\infty),L^2\big(\mb{R}^3\big)\big)\big)^3
	\end{equation*}
	to the Cauchy problem \eqref{fsdew001}. Moreover, the following estimates hold:
	\begin{equation*}
	\begin{split}
	\|U^{(k)}(t,\cdot)\|_{L^2}&\lesssim (1+t)^{-\frac{6-5m}{4m\theta}+g_k}\sum\limits_{k=1}^3\big\|\big(U^{(k)}_0,U^{(k)}_1\big)\big\|_{\ml{D}_{m,1}^0},\\
	\|\nabla_xU^{(k)}(t,\cdot)\|_{L^2}+\|U_t^{(k)}(t,\cdot)\|_{L^2}&\lesssim (1+t)^{-\frac{6-3m}{4m\theta}+g_k}\sum\limits_{k=1}^3\big\|\big(U^{(k)}_0,U^{(k)}_1\big)\big\|_{\ml{D}_{m,1}^0},
	\end{split}
	\end{equation*}
	where the parameters $g_k$ are chosen in the following way:
\begin{enumerate}
\item  $g_{k}=0$ for $k=1,2,3,$ when $p_1,p_2,p_3$ satisfy the condition \eqref{0exp01};
 \item \[ g_{k_1}=\frac{3+2m\theta}{2m\theta}-\frac{3-m}{2m\theta}p_{k_1}\,\,\,\,\mbox{and}\,\,\,\, g_{k_2}=g_{k_3}=0,\] when $p_{k_1},p_{k_2},p_{k_3}$ satisfy the condition \eqref{0exp02};
\item \begin{eqnarray*} && g_{k_1}=\frac{3+2m\theta}{2m\theta}-\frac{3-m}{2m\theta}p_{k_1}, \,\,\,\, g_{k_2}=\frac{3+2m\theta}{2m\theta}+\frac{1+2\theta}{2\theta}p_{k_2}-\frac{3-m}{2m\theta}p_{k_1}p_{k_2}\,\,\,\,\mbox{and}\,\,\,\, g_{k_3}=0,
     \end{eqnarray*} when $p_{k_1},p_{k_2},p_{k_3}$ satisfy the condition \eqref{0exp03}.
     \end{enumerate}
\end{thm}
\begin{proof} We define for $T>0$ the spaces of solutions $X(T)$ by
	\begin{equation}\label{evolutionspace01}
	X(T):=\big(\ml{C}\big([0,T],H^1\big(\mb{R}^3\big)\big)\cap\ml{C}^1\big([0,T],L^2\big(\mb{R}^3\big)\big)\big)^3
	\end{equation}
	with the corresponding norm
	\begin{equation*}
	\|U\|_{X(T)}:=\sup\limits_{0\leq t\leq T}\Big(\sum\limits_{k=1}^3(1+t)^{\frac{6-5m}{4m\theta}-g_k}\|U^{(k)}(t,\cdot)\|_{L^2}+\sum\limits_{k=1}^3(1+t)^{\frac{6-3m}{4m\theta}-g_k}\big(\|\nabla_x U^{(k)}(t,\cdot)\|_{L^2}+\|U_t^{(k)}(t,\cdot)\|_{L^2}\big)\Big).
	\end{equation*}
	The classical Gagliardo-Nirenberg inequality implies
	\begin{equation*}
	\begin{split}
	\big\||U^{(k-1)}(\tau,x)|^{p_k}\big\|_{L^m}&\lesssim (1+\tau)^{-\frac{(3-m)p_k-3}{2m\theta}+g_{k-1}p_k}\|U\|_{X(\tau)}^{p_k},\\
	\big\||U^{(k-1)}(\tau,x)|^{p_k}\big\|_{L^2}&\lesssim (1+\tau)^{-\frac{2(3-m)p_k-3m}{4m\theta}+g_{k-1}p_k}\|U\|_{X(\tau)}^{p_k},
	\end{split}
	\end{equation*}
	where $m\in\left[1,6/5\right)$ and the parameters $\beta_{0,1}(mp_k)\in[0,1]$ and $\beta_{0,1}(2p_k)\in[0,1]$ in the classical Gagliardo-Nirenberg inequality lead to the condition $p_k\in\left[2/m,3\right]$ for all $k=1,2,3$.
	To begin with, we apply the derived $(L^2\cap L^m)$-$L^2$ estimate on $[0,t]$ to estimate the solution as follows:
	\begin{equation*}
	\begin{split}
	 (1+t)^{\frac{6-5m}{4m\theta}-g_k}\|N_kU(t,\cdot)\|_{L^2}&\lesssim(1+t)^{-g_k}\sum\limits_{k=1}^3\big\|\big(U_0^{(k)},U_1^{(k)}\big)\big\|_{\ml{D}_{m,1}^0}+(1+t)^{-g_k}\|U\|_{X(t)}^{p_k}\int\nolimits_{0}^{t/2}(1+\tau)^{-\frac{(3-m)p_k-3}{2m\theta}+g_{k-1}p_k}d\tau\\
	&\quad\,\,+(1+t)^{1-g_k-\frac{(3-m)p_k-3}{2m\theta}+g_{k-1}p_k}\|U\|_{X(t)}^{p_k},
	\end{split}
	\end{equation*}
	where we divide the interval $[0,t]$ into sub-intervals $\left[0,t/2\right]$ and $\left[t/2,t\right]$ and use
	\begin{equation*}
	\int\nolimits_{t/2}^t(1+t-\tau)^{-\frac{6-5m}{4m\theta}}d\tau\lesssim(1+t)^{1-\frac{6-5m}{4m\theta}}
	\end{equation*}
	due to the fact that $6-5m<4m\theta$ for all $m\in\left[1,6/5\right)$ and $\theta\in\left[1/2,1\right]$.\\
	Next, using the derived $(L^2\cap L^m)$-$L^2$ estimate on $\left[0,t/2\right]$ and $L^2$-$L^2$ estimate on $\left[t/2,t\right]$, we obtain the estimate for the first order derivatives ($j+l=1$) as follows:
	\begin{equation*}
	\begin{split}
	 (1+t)^{\frac{6-3m}{4m\theta}-g_k}\|\partial_t^j\nabla_x^lN_kU(t,\cdot)\|_{L^2}\lesssim&(1+t)^{-g_k}\sum\limits_{k=1}^3\big\|\big(U_0^{(k)},U_1^{(k)}\big)\big\|_{\ml{D}_{m,1}^0}+(1+t)^{-g_k}\|U\|_{X(t)}^{p_k}\int\nolimits_{0}^{t/2}(1+\tau)^{-\frac{(3-m)p_k-3}{2m\theta}+g_{k-1}p_k}d\tau\\
	&+(1+t)^{1+\frac{6-3m}{4m\theta}-g_k-\frac{2(3-m)p_k-3m}{4m\theta}+g_{k-1}p_k}\|U\|_{X(t)}^{p_k}.
	\end{split}
	\end{equation*}
	Summarizing the above estimates gives
	\begin{equation}\label{1/2,1estimate01}
	\begin{split}
	 &(1+t)^{\frac{6-5m+2(j+l)m}{4m\theta}-g_k}\|\partial_t^j\nabla_x^lN_kU(t,\cdot)\|_{L^2}\lesssim(1+t)^{-g_k}\sum\limits_{k=1}^3\big\|\big(U_0^{(k)},U_1^{(k)}\big)\big\|_{\ml{D}_{m,1}^0}\\
	 &\qquad\qquad\quad+(1+t)^{-g_k}\|U\|_{X(t)}^{p_k}\Big(\int\nolimits_{0}^{t/2}(1+\tau)^{-\frac{(3-m)p_k-3}{2m\theta}+g_{k-1}p_k}d\tau+(1+t)^{1-\frac{(3-m)p_k-3}{2m\theta}+g_{k-1}p_k}\Big)
	\end{split}
	\end{equation}
	for all $j+l=0,1$ with $j,l\in\mb{N}_0$. Now, we distinguish between three cases to prove
	\begin{equation}\label{1/2,1aim01}
	 (1+t)^{\frac{6-5m+2(j+l)m}{4m\theta}-g_k}\|\partial_t^j\nabla_x^lN_kU(t,\cdot)\|_{L^2}\lesssim\sum\limits_{k=1}^3\big\|\big(U_0^{(k)},U_1^{(k)}\big)\big\|_{\ml{D}_{m,1}^0}+\|U\|_{X(t)}^{p_k}.
	\end{equation}
	\emph{Case 1} $\quad$ We assume the condition (\ref{0exp01}), that is, $\min\left\{p_1;p_2;p_3\right\}>p_c(m,\theta)$.\medskip
	
	\noindent In this case, we have no loss of decay. So, we choose the parameters $g_1=g_2=g_3=0$. Therefore, \eqref{1/2,1estimate01} implies the following estimates for $j+l=0,1$:
	\begin{equation*}
	\begin{split}
	 (1+t)^{\frac{6-5m+2(j+l)m}{4m\theta}}\|\partial_t^j\nabla_x^lN_kU(t,\cdot)\|_{L^2}\lesssim\sum\limits_{k=1}^3\big\|\big(U_0^{(k)},U_1^{(k)}\big)\big\|_{\ml{D}_{m,1}^0}+\|U\|_{X(t)}^{p_k}\Big(\int\nolimits_{0}^{t/2}(1+\tau)^{-\frac{(3-m)p_k-3}{2m\theta}}d\tau+(1+t)^{1-\frac{(3-m)p_k-3}{2m\theta}}\Big),
	\end{split}
	\end{equation*}
	where $k=1,2,3$. To prove \eqref{1/2,1aim01}, we need that the right-hand side in the last inequality is uniformly bounded in $t>0$. But, this follows from the condition
\begin{equation*}
	\min\left\{p_1;p_2;p_3\right\}>p_c(m,\theta)=\frac{2m\theta+3}{3-m}\,\,\,\,\text{for}\,\,\,\,\theta\in\left[1/2,1\right].
	\end{equation*}
	\emph{Case 2} $\quad$ We assume the condition (\ref{0exp02}), that is, $1<p_{k_1}< p_c(m,\theta)$ and $p_{k_2},p_{k_3}>p_c(m,\theta)$.\medskip
	
	\noindent Now, we allow a loss of decay in one component of the solution. We choose $g_{k_1}=\frac{3+2m\theta}{2m\theta}-\frac{3-m}{2m\theta}p_{k_1}$ and $g_{k_2}=g_{k_3}=0$. Obviously, by the assumption $1<p_{k_1}< p_c(m,\theta)$ we may get $g_{k_1} >0$. The condition $\alpha_{k_1}(m,\theta)<3/2$ is equivalent to the following inequality
	\begin{equation}\label{condicase2}
	3+2m\theta+m(1+2\theta)p_{k_2}+(m-3)p_{k_1}p_{k_2}<0.
	\end{equation}
	With the assumption $1<p_{k_1}< p_c(m,\theta)$, the condition \eqref{condicase2} is valid only when $p_{k_2}>p_c(m,\theta)$. Moreover, we have
	\begin{equation}\label{estimate11}
	\begin{split}
	 (1+t)^{\frac{6-5m+2(j+l)m}{4m\theta}-g_{k_1}}\|\partial_t^j\nabla_x^lN_{k_1}U(t,\cdot)\|_{L^2}\lesssim\sum\limits_{k=1}^3\big\|\big(U_0^{(k)},U_1^{(k)}\big)\big\|_{\ml{D}_{m,1}^0}+(1+t)^{1-g_{k_1}-\frac{(3-m)p_{k_1}-3}{2m\theta}}\|U\|_{X(t)}^{p_{k_1}},
	\end{split}
	\end{equation}
	where we use
	\begin{equation*}
	\int\nolimits_{0}^{t/2}(1+\tau)^{-\frac{(3-m)p_{k_1}-3}{2m\theta}}d\tau\lesssim(1+t)^{1-\frac{(3-m)p_{k_1}-3}{2m\theta}}.
	\end{equation*}
	Hence, the above estimates lead to the desired estimate \eqref{1/2,1aim01} when $k=k_1$.\\
	When $k=k_2$, we obtain the following estimate:
	\begin{equation*}
	\begin{split}
	 (1+t)^{\frac{6-5m+2(j+l)m}{4m\theta}}\|\partial_t^j\nabla_x^lN_{k_2}U(t,\cdot)\|_{L^2}\lesssim&\sum\limits_{k=1}^3\big\|\big(U_0^{(k)},U_1^{(k)}\big)\big\|_{\ml{D}_{m,1}^0}\\
	 &+\|U\|_{X(t)}^{p_{k_2}}\Big(\int\nolimits_{0}^{t/2}(1+\tau)^{-\frac{(3-m)p_{k_2}-3}{2m\theta}+g_{k_1}p_{k_2}}d\tau+(1+t)^{1-\frac{(3-m)p_{k_2}-3}{2m\theta}+g_{k_1}p_{k_2}}\Big).
	\end{split}
	\end{equation*}
	Applying the condition \eqref{condicase2} it follows
	\begin{equation*}
	-\frac{(3-m)p_{k_2}-3}{2m\theta}+g_{k_1}p_{k_2}=\frac{3+m(1+2\theta)p_{k_2}+(m-3)p_{k_1}p_{k_2}}{2m\theta}<-1.
	\end{equation*}
	So it immediately leads to the estimate \eqref{1/2,1aim01} when $k=k_2$.\\
	The case $k=k_3$ can be treated by using the same arguments we did in studying \emph{Case 1}. Precisely, we have
	\begin{equation*}
	\begin{split}
	 (1+t)^{\frac{6-5m+2(j+l)m}{4m\theta}}\|\partial_t^j\nabla_x^lN_{k_3}U(t,\cdot)\|_{L^2}\lesssim\sum\limits_{k=1}^3\big\|\big(U_0^{(k)},U_1^{(k)}\big)\big\|_{\ml{D}_{m,1}^0}+\|U\|_{X(t)}^{p_{k_3}}\Big(\int\nolimits_{0}^{t/2}(1+\tau)^{-\frac{(3-m)p_{k_3}-3}{2m\theta}}d\tau+(1+t)^{1-\frac{(3-m)p_{k_3}-3}{2m\theta}}\Big).
	\end{split}
	\end{equation*}
	Taking account of $p_{k_3}>p_c(m,\theta)$, the estimate \eqref{1/2,1aim01} is valid for $k=k_3$.\\
	\\
	\emph{Case 3} $\quad$ We assume the condition (\ref{0exp03}), that is, $1<p_{k_1},p_{k_2}< p_c(m,\theta)$ and $p_{k_3}>p_c(m,\theta)$.\medskip
	
	\noindent Here, we take the parameters describing the loss of decay as follows: \[ g_{k_1}=\frac{3+2m\theta}{2m\theta}-\frac{3-m}{2m\theta}p_{k_1},\,\,\,\, g_{k_2}=\frac{3+2m\theta}{2m\theta}+\frac{1+2\theta}{2\theta}p_{k_2}-\frac{3-m}{2m\theta}p_{k_1}p_{k_2}\,\,\,\,\mbox{and}\,\,\,\, g_{k_3}=0.\] With the help of $1<p_{k_1},p_{k_2}< p_c(m,\theta)$, we conclude that $g_{k_1}>0$ as well as $g_{k_2}>0$. Then, the condition $\widetilde{\alpha}_{k_1}(m,\theta)<3/2$ can be rewritten as
	\begin{equation}\label{condicase3}
	3+2m\theta+m(1+2\theta)p_{k_3}(1+p_{k_2})+(m-3)p_{k_1}p_{k_2}p_{k_3}<0.
	\end{equation}
	We know that the inequality \eqref{condicase3} is valid only if $p_{k_3}>p_c(m,\theta)$. Following the same approach for treating \emph{Case 2} we immediately obtain the desired estimate \eqref{1/2,1aim01} in this case.
	
Lastly, no matter in which case, we may derive the Lipschitz condition by using H\"older's inequality and Gagliardo-Nirenberg inequality. In other words, we may prove
	\begin{equation*}
	 (1+t)^{\frac{6-5m+2(j+l)m}{4m\theta}-g_k}\|\partial_t^j\nabla_x^l(N_kU-N_kV)(t,\cdot)\|_{L^2}\lesssim\|U-V\|_{X(t)}\big(\|U\|_{X(t)}^{p_k-1}+\|V\|_{X(t)}^{p_k-1}\big)
	\end{equation*}
	for $j+l=0,1$ with $j,l\in\mb{N}_0$ and $k=1,2,3$ for \emph{Cases 1-3}. Therefore, the proof is complete.
\end{proof}
Next, when $m\in\left[6/5,3/2\right)$, the estimates for the solutions to the linear Cauchy problem \eqref{linearproblem} are different to those in the case $m\in\left[1,6/5\right)$.
For this reason we also feel differences in estimating the norms $\||U^{(k-1)}(\tau,x)|^{p_k}\|_{L^m}$ and $\||U^{(k-1)}(\tau,x)|^{p_k}\|_{L^2}$. Now, we state our result for $m\in\left[6/5,3/2\right)$. If some of the exponents $p_{k_j}=p_{\text{bal}}(m,0,\theta)$ for some $j=1,2,3$, we can choose the parameters $g_{k_j}$ describing the loss of decay as $g_{k_j}=\varepsilon_1$ with a sufficiently small constant $\varepsilon_1>0$.
\begin{thm}\label{GESDS09} Let us consider the semi-linear model \eqref{fsdew001} with $\theta\in\left[1/2,1\right]$ and choose $m\in\left[6/5,3/2\right)$. Let us assume $p_k\in\left[2/m,3\right]$ for $k=1,2,3,$ such that
	\begin{flalign}\label{0000exp01}
	&\text{(i)}\quad\text{there are no other restrictions when}\,\,\,\,\min\left\{p_1;p_2;p_3\right\}>p_{\text{bal}}(m,0,\theta);&
	\end{flalign}
	\begin{flalign}\label{0000exp02}
	&\text{(ii)}\quad\alpha_{k_1,\text{bal}}(m,0,\theta)<3/2\,\,\,\,\text{when}\,\,\,\,1<p_{k_1}< p_{\text{bal}}(m,0,\theta)\,\,\,\,\text{and}\,\,\,\,p_{k_2},p_{k_3}>p_{\text{bal}}(m,0,\theta);&
	\end{flalign}
	\begin{flalign}\label{0000exp03}
	&\text{(iii)}\quad \widetilde{\alpha}_{k_1,\text{bal}}(m,0,\theta)<3/2\,\,\,\,\text{when}\,\,\,\,1<p_{k_1},p_{k_2}< p_{\text{bal}}(m,0,\theta)\,\,\,\,\text{and}\,\,\,\,p_{k_3}>p_{\text{bal}}(m,0,\theta).&
	\end{flalign}
	Then, there exists a constant $\varepsilon_0>0$ such that for all $\big(U^{(k)}_0,U^{(k)}_1\big)\in\ml{D}_{m,1}^0$ with $\sum\limits_{k=1}^3\big\|\big(U^{(k)}_0,U^{(k)}_1\big)\big\|_{\ml{D}_{m,1}^0}\leq\varepsilon_0$ there is a uniquely determined energy solution
	\begin{equation*}
	U\in\big(\mathcal{C}\big([0,\infty),H^1\big(\mb{R}^3\big)\big)\cap \mathcal{C}^1\big([0,\infty),L^2\big(\mb{R}^3\big)\big)\big)^3
	\end{equation*}
	to the Cauchy problem \eqref{fsdew001}. Moreover, the following estimates hold:
	\begin{equation*}
	\begin{split}
	\|U^{(k)}(t,\cdot)\|_{L^2}&\lesssim (1+t)^{1-\frac{6-3m}{4m\theta}+g_k}\sum\limits_{k=1}^3\big\|\big(U^{(k)}_0,U^{(k)}_1\big)\big\|_{\ml{D}_{m,1}^0},\\
	\|\nabla_xU^{(k)}(t,\cdot)\|_{L^2}+\|U_t^{(k)}(t,\cdot)\|_{L^2}&\lesssim (1+t)^{-\frac{6-3m}{4m\theta}+g_k}\sum\limits_{k=1}^3\big\|\big(U^{(k)}_0,U^{(k)}_1\big)\big\|_{\ml{D}_{m,1}^0},
	\end{split}
	\end{equation*}
	where the parameters $g_k$ are chosen in the following way:
\begin{enumerate}
\item $g_{k}=0$ for $k=1,2,3,$ when $p_1,p_2,p_3$ satisfy the condition \eqref{0000exp01};
\item \[ g_{k_1}=\frac{m+3}{m}-\Big(\frac{1}{2}+\frac{6-3m}{4m\theta}\Big)p_{k_1}\,\,\,\,\mbox{and} \,\,\,\,g_{k_2}=g_{k_3}=0,\] when $p_{k_1},p_{k_2},p_{k_3}$ satisfy the condition \eqref{0000exp02};
    \item
	\[ g_{k_1}=\frac{m+3}{m}-\Big(\frac{1}{2}+\frac{6-3m}{4m\theta}\Big)p_{k_1}, \,\,\,\, g_{k_2}=\frac{m+3}{m}-\Big(\frac{6-3m}{4m\theta}-\frac{1}{2}-\frac{3}{m}\Big)p_{k_2}-\Big(\frac{1}{2}+\frac{6-3m}{4m\theta}\Big)p_{k_1}p_{k_2}\,\,\,\,\mbox{and} \,\,\,\,g_{k_3}=0,\] when $p_{k_1},p_{k_2},p_{k_3}$ satisfy the condition \eqref{0000exp03}.
\end{enumerate}
\end{thm}
\begin{proof}
	Here, we only redefine the norm of the evolution space \eqref{evolutionspace01} as follows:
	\begin{equation*}
	\|U\|_{X(T)}:=\sup\limits_{0\leq t\leq T}\Big(\sum\limits_{k=1}^3(1+t)^{-1+\frac{6-3m}{4m\theta}-g_k}\|U^{(k)}(t,\cdot)\|_{L^2}+\sum\limits_{k=1}^3(1+t)^{\frac{6-3m}{4m\theta}-g_k}\big(\|\nabla_x U^{(k)}(t,\cdot)\|_{L^2}+\|U_t^{(k)}(t,\cdot)\|_{L^2}\big)\Big).
	\end{equation*}
	After following the proof of Theorem \ref{GESDS02} we may conclude the desired statements.
\end{proof}
\begin{rem}\label{remm} If we would consider the semi-linear model \eqref{fsdew001} with $\theta\in\left[1/2,1\right]$ and  $m\in\left[3/2,2\right)$, we should guarantee $\max\left\{p_1;p_2;p_3\right\}\leq3$ from parameter restrictions appearing by the application of the classical Gagliardo-Nirenberg inequality. However, we should give at least for one exponent $k=1,2,3$ another restriction
	\begin{equation*}
	 p_k>p_{\text{bal}}(m,0,\theta)=2+\frac{6(m-2+2\theta)}{2m\theta-3m+6}\geq3\,\,\,\,\text{if}\,\,\,\,m\in\left[3/2,2\right)\,\,\,\,\text{and}\,\,\,\,\theta\in\left[1/2,1\right].
	\end{equation*}
	In conclusion, the set of admissible triplets of exponents $(p_1,p_2,p_3)$ is empty.
\end{rem}
\subsubsection{Data from energy spaces with suitable higher regularity}
\begin{thm}\label{HlL1nonlinearythm01}
	Let us consider the semi-linear model \eqref{fsdew001} with $\theta\in\left[1/2,1\right]$ and choose $m\in\left[1,6/5\right)$. Let us assume
	\begin{equation*}
	\begin{aligned}
	&1+\lceil s\rceil<\min\left\{p_1;p_2;p_3\right\}\leq\max\left\{p_1;p_2;p_3\right\}\leq1+2/(1-2s)\,\,\,\,&\text{if}\,\,\,\,&s\in\left(0,1/2\right),\\
	&1+\lceil s\rceil<\min\left\{p_1;p_2;p_3\right\}\leq\max\left\{p_1;p_2;p_3\right\}<\infty\,\,\,\,&\text{if}\,\,\,\,&s\in\left[1/2,\infty\right).\\
	\end{aligned}
	\end{equation*}
	The exponents $p_1,p_2,p_3$ satisfy one of the conditions \eqref{0exp01} to \eqref{0exp03}. Then, there exists a constant $\varepsilon_0>0$ such that for all $\big(U^{(k)}_0,U^{(k)}_1\big)\in\ml{D}^s_{m,1}$ with $\sum\limits_{k=1}^3\big\|\big(U^{(k)}_0,U^{(k)}_1\big)\big\|_{\ml{D}^s_{m,1}}\leq\varepsilon_0$ there is a uniquely determined energy solution
	\begin{equation*}
	U\in \big(\mathcal{C}\big([0,\infty),H^{s+1}\big(\mb{R}^3\big)\big)\cap \mathcal{C}^1\big([0,\infty),H^s\big(\mb{R}^3\big)\big)\big)^3
	\end{equation*}
	to the Cauchy problem \eqref{fsdew001}. Moreover, the following estimates hold:
	\begin{equation*}
	\begin{split}
	\|U^{(k)}(t,\cdot)\|_{L^2}&\lesssim (1+t)^{-\frac{6-5m}{4m\theta}+g_k} \sum\limits_{k=1}^3\big\|\big(U^{(k)}_0,U^{(k)}_1\big)\big\|_{\ml{D}^s_{m,1}},\\
	\|U_t^{(k)}(t,\cdot)\|_{L^2}&\lesssim (1+t)^{-\frac{6-3m}{4m\theta}+g_k} \sum\limits_{k=1}^3\big\|\big(U^{(k)}_0,U^{(k)}_1\big)\big\|_{\ml{D}^s_{m,1}},\\
	\||D|U^{(k)}(t,\cdot)\|_{\dot{H}^{s}}+\|U^{(k)}_t(t,\cdot)\|_{\dot{H}^{s}}&\lesssim (1+t)^{-\frac{6-3m+2sm}{4m\theta}+g_k}\sum\limits_{k=1}^3\big\|\big(U^{(k)}_0,U^{(k)}_1\big)\big\|_{\ml{D}^s_{m,1}},
	\end{split}
	\end{equation*}
	where the parameters $g_k$ are the same as  in Theorem \ref{GESDS02}.
\end{thm}
\begin{proof} We define the norm for the evolution space \eqref{spp} as follows:
	\begin{equation*}
	\begin{split}
	\|U\|_{X(T)}:&=\sup\limits_{0\leq t\leq T}\Big(\sum\limits_{k=1}^3(1+t)^{\frac{6-5m}{4m\theta}-g_k}\|U^{(k)}(t,\cdot)\|_{L^2}+\sum\limits_{k=1}^3(1+t)^{\frac{6-3m}{4m\theta}-g_k}\|U_t^{(k)}(t,\cdot)\|_{L^2}\\
	 &\quad\quad\quad\quad\quad\quad\quad\quad+\sum\limits_{k=1}^3(1+t)^{\frac{6-3m+2sm}{4m\theta}-g_k}\big(\||D|U^{(k)}(t,\cdot)\|_{\dot{H}^s}+\|U_t^{(k)}(t,\cdot)\|_{\dot{H}^s}\big)\Big).
	\end{split}
	\end{equation*}
	Then, we immediately follow the proof of Theorem \ref{HlL1nonlinearythm} to complete this proof.
\end{proof}

The choice of data from higher-order energy spaces allows us to weaken the upper bounds for the exponents $p_1,p_2,p_3$ for which we can prove the global (in time) existence of small data energy solutions by choosing the parameter of additional regularity in the interval $m\in\left[3/2,2\right)$. To be more precise, the following statements hold:
\begin{equation*}
\big[2/m,1+2/(1-2s)\big]\cap \big(p_{\text{bal}}(m,0,\theta),\infty\big)\neq\emptyset\,\,\,\,\text{for}\,\,\,\,s\in\left(0,1/2\right)
\end{equation*}
and
\begin{equation*}
\big[2/m,\infty\big)\cap \big(p_{\text{bal}}(m,0,\theta),\infty\big)\neq\emptyset\,\,\,\,\text{for}\,\,\,\,s\in\left[1/2,\infty\right)
\end{equation*}
when $\theta\in\left[1/2,1\right]$ and $m\in\left[3/2,2\right)$. Thus, we can get a more flexible (with respect to $s$) admissible range of exponents $p_1,p_2,p_3$.
\begin{thm}\label{HlL1nonlinearythm09}
	Let us consider the semi-linear model \eqref{fsdew001} with $\theta\in\left[1/2,1\right]$ and choose $m\in\left[6/5,2\right)$. Let us assume
	\begin{equation*}
	\begin{aligned}
	&1+\lceil s\rceil<\min\left\{p_1;p_2;p_3\right\}\leq\max\left\{p_1;p_2;p_3\right\}\leq1+2/(1-2s)\,\,\,\,&\text{if}\,\,\,\,&s\in\left(0,1/2\right),\\
	&1+\lceil s\rceil<\min\left\{p_1;p_2;p_3\right\}\leq\max\left\{p_1;p_2;p_3\right\}<\infty\,\,\,\,&\text{if}\,\,\,\,&s\in\left[1/2,\infty\right).\\
	\end{aligned}
	\end{equation*}
	The exponents $p_1,p_2,p_3$ satisfy one of the conditions \eqref{0000exp01} to \eqref{0000exp03}. Then, there exists a constant $\varepsilon_0>0$ such that for all $\big(U^{(k)}_0,U^{(k)}_1\big)\in\ml{D}^s_{m,1}$ with $\sum\limits_{k=1}^3\big\|\big(U^{(k)}_0,U^{(k)}_1\big)\big\|_{\ml{D}^s_{m,1}}\leq\varepsilon_0$ there is a uniquely determined energy solution
	\begin{equation*}
	U\in \big(\mathcal{C}\big([0,\infty),H^{s+1}\big(\mb{R}^3\big)\big)\cap \mathcal{C}^1\big([0,\infty),H^s\big(\mb{R}^3\big)\big)\big)^3
	\end{equation*}
	to the Cauchy problem \eqref{fsdew001}. Moreover, the following estimates hold:
	\begin{equation*}
	\begin{split}
	\|U^{(k)}(t,\cdot)\|_{L^2}&\lesssim (1+t)^{1-\frac{6-3m}{4m\theta}+g_k} \sum\limits_{k=1}^3\big\|\big(U^{(k)}_0,U^{(k)}_1\big)\big\|_{\ml{D}^s_{m,1}},\\
	\|U_t^{(k)}(t,\cdot)\|_{L^2}&\lesssim (1+t)^{-\frac{6-3m}{4m\theta}+g_k} \sum\limits_{k=1}^3\big\|\big(U^{(k)}_0,U^{(k)}_1\big)\big\|_{\ml{D}^s_{m,1}},\\
	\||D|U^{(k)}(t,\cdot)\|_{\dot{H}^{s}}+\|U^{(k)}_t(t,\cdot)\|_{\dot{H}^{s}}&\lesssim (1+t)^{-\frac{6-3m+2sm}{4m\theta}+g_k}\sum\limits_{k=1}^3\big\|\big(U^{(k)}_0,U^{(k)}_1\big)\big\|_{\ml{D}^s_{m,1}},
	\end{split}
	\end{equation*}
	where the parameters $g_k$ are the same as in Theorem \ref{GESDS02}.
\end{thm}
\begin{proof} Following the proof of Theorem \ref{HlL1nonlinearythm}, this theorem can be proved immediately by redefining the norm for the evolution space \eqref{spp} as follows:
	\begin{equation*}
	\begin{split}
	\|U\|_{X(T)}:&=\sup\limits_{0\leq t\leq T}\Big(\sum\limits_{k=1}^3(1+t)^{-1+\frac{6-3m}{4m\theta}-g_k}\|U^{(k)}(t,\cdot)\|_{L^2}+\sum\limits_{k=1}^3(1+t)^{\frac{6-3m}{4m\theta}-g_k}\|U_t^{(k)}(t,\cdot)\|_{L^2}\\
	 &\quad\quad\quad\quad\quad\quad\qquad\qquad+\sum\limits_{k=1}^3(1+t)^{\frac{6-3m+2sm}{4m\theta}-g_k}\big(\||D|U^{(k)}(t,\cdot)\|_{\dot{H}^s}+\|U_t^{(k)}(t,\cdot)\|_{\dot{H}^s}\big)\Big),
	\end{split}
	\end{equation*}
	where $g_k$ are the same as in Theorem \ref{GESDS02}.
\end{proof}
Finally, we are interested in the case of large regular data belonging to $L^{\infty}\big(\mb{R}^3\big)$, too. For this reason, we choose the regularity parameter $s$ from the interval $\left(3/2,\infty\right)$. Let us restrict ourselves to the case $m\in\left[1,6/5\right)$.
\begin{thm}Let us consider the semi-linear model \eqref{fsdew001} with $\theta\in\left[1/2,1\right]$ and $m\in\left[1,6/5\right)$, $s>3/2$. Let us assume
	\begin{equation*}
	\max\left\{1+s;p_c(m,\theta)\right\}<\min\left\{p_1;p_2;p_3\right\},
	\end{equation*}
	and one of the conditions \eqref{0exp01} to \eqref{0exp03}. Then, there exists a constant $\varepsilon_0>0$ such that for all $\big(U_0^{(k)},U_1^{(k)}\big)\in\ml{D}_{m,1}^s$ with $\sum\limits_{k=1}^3\big\|\big(U_0^{(k)},U_1^{(k)}\big)\big\|_{\ml{D}_{m,1}^s}\leq\varepsilon_0$ there is a uniquely determined energy solution
	\begin{equation*}
	U\in \big(\mathcal{C}\big([0,\infty),H^{s+1}\big(\mb{R}^3\big)\big)\cap \mathcal{C}^1\big([0,\infty),H^s\big(\mb{R}^3\big)\big)\big)^3
	\end{equation*}
	to the Cauchy problem \eqref{fsdew001}. Moreover, the estimates for the solutions are the same as in Theorem \ref{HlL1nonlinearythm01} choosing $g_k=0$ for all $k=1,2,3$.
\end{thm}
\begin{proof} One can complete the proof by following the same steps of the proof of Theorem \ref{embthm}.
\end{proof}
\section{Concluding remarks and open problems}\label{Concludingremark}
\begin{rem}
Sharp energy estimates for the solutions to the linear damped elastic waves \eqref{linearproblem} in three space dimensions with $\theta\in\left[0,1/2\right)\cup\left(1/2,1\right]$ and data $\big(u_0^{(k)},u_1^{(k)}\big)\in(H^{s+1}\cap L^m)\times (H^s\cap L^m)$ for all $k=1,2,3$, are still open.	
\end{rem}

\begin{rem}
By the same motivation for using data from the space $\ml{D}_{3/2,1}^0$ we can obtain another admissible range for the exponents $p_1,p_2,p_3$ for proving the global (in time) existence of energy solutions with small data having an additional regularity $L^{3/2}$. Following the same approach as in the proofs to Theorems \ref{GESDS03} and \ref{HlL1nonlinearythm} we can prove the following result.
\begin{coro}\label{HlL1nonlinearythm000000}
	Let us consider the semi-linear model \eqref{fsdew001} with $\theta\in\left[0,1/2\right)$ and $s\in\left(0,1/2\right)$. Let us choose
	\begin{equation*}
	\begin{aligned}
	&1+\lceil s\rceil<\min\left\{p_1;p_2;p_3\right\}\leq\max\left\{p_1;p_2;p_3\right\}\leq 1+2/(1-2s)
	\end{aligned}
	\end{equation*}
	such that
	\begin{flalign}\label{000exp011111}
	&\text{(i)}\quad\text{there are no other restrictions when}\,\,\,\,\min\left\{p_1;p_2;p_3\right\}>p_{\text{bal}}(3/2,s,\theta);&
	\end{flalign}
	\begin{flalign}\label{000exp011112}
	&\text{(ii)}\quad \alpha_{k_1,\text{bal}}(3/2,s,\theta)<3/2\,\,\,\,\text{when}\,\,\,\,1<p_{k_1}<p_{\text{bal}}(3/2,s,\theta)\,\,\,\,\text{and}\,\,\,\,p_{k_2},p_{k_3}>p_{\text{bal}}(3/2,s,\theta);&
	\end{flalign}
	\begin{flalign}\label{000exp011113}
	 &\text{(iii)}\quad\widetilde{\alpha}_{k_1,\text{bal}}(3/2,s,\theta)<3/2\,\,\,\,\text{when}\,\,\,\,1<p_{k_1},p_{k_2}<p_{\text{bal}}(3/2,s,\theta)\,\,\,\,\text{and}\,\,\,\,p_{k_3}>p_{\text{bal}}(3/2,s,\theta).&
	\end{flalign}
	Then, there exists a constant $\varepsilon_0>0$ such that for all $\big(U^{(k)}_0,U^{(k)}_1\big)\in\ml{D}^s_{3/2,1}$ with $\sum\limits_{k=1}^3\big\|\big(U^{(k)}_0,U^{(k)}_1\big)\big\|_{\ml{D}^s_{3/2,1}}\leq\varepsilon_0$ there exists a uniquely determined energy solution
	\begin{equation*}
	U\in \big(\mathcal{C}\big([0,\infty),H^{s+1}\big(\mb{R}^3\big)\big)\cap \mathcal{C}^1\big([0,\infty),H^s\big(\mb{R}^3\big)\big)\big)^3
	\end{equation*}
	to the Cauchy problem \eqref{fsdew001}. Moreover, the following estimates hold:
	\begin{equation*}
	\begin{split}
	\|U^{(k)}(t,\cdot)\|_{L^2}&\lesssim (1+t)^{1-\rho_{1}(3/2,\theta)+g_k} \sum\limits_{k=1}^3\big\|\big(U^{(k)}_0,U^{(k)}_1\big)\big\|_{\ml{D}^s_{3/2,1}},\\
	\|U_t^{(k)}(t,\cdot)\|_{L^2}&\lesssim (1+t)^{-\rho_{1}(3/2,\theta)+g_k} \sum\limits_{k=1}^3\big\|\big(U^{(k)}_0,U^{(k)}_1\big)\big\|_{\ml{D}^s_{3/2,1}},\\
	\||D|U^{(k)}(t,\cdot)\|_{\dot{H}^{s}}+\|U_t^{(k)}(t,\cdot)\|_{\dot{H}^{s}}&\lesssim (1+t)^{-\rho_{s+1}(3/2,\theta)+g_k}\sum\limits_{k=1}^3\big\|\big(U^{(k)}_0,U^{(k)}_1\big)\big\|_{\ml{D}^s_{3/2,1}},
	\end{split}
	\end{equation*}
	where the parameters $g_k$ are chosen in the following way:
	\begin{enumerate}
		\item  $g_{k}=0$ for $k=1,2,3,$ when $p_1,p_2,p_3$ satisfy the condition \eqref{000exp011111};
		\item \[ g_{k_1}=1+\frac{2-2\theta+s}{(1-\theta)(s+1)}+\frac{6\theta-5-2s}{4(1-\theta)(s+1)}p_{k_1}\,\,\,\,\mbox{and}\,\,\,\, g_{k_2}=g_{k_3}=0,\] when $p_{k_1},p_{k_2},p_{k_3}$ satisfy the condition \eqref{000exp011112};
		\item \begin{eqnarray*}
			&& g_{k_1}=1+\frac{2-2\theta+s}{(1-\theta)(s+1)}+\frac{6\theta-5-2s}{4(1-\theta)(s+1)}p_{k_1}, \\ && g_{k_2}=1+\frac{2-2\theta+s}{(1-\theta)(s+1)}+\Big(1+\frac{3+2s-2\theta}{4(1-\theta)(s+1)}\Big)p_{k_2}+\frac{6\theta-5-2s}{4(1-\theta)(s+1)}p_{k_1}p_{k_2}\,\,\,\,\mbox{and}\,\,\,\, g_{k_3}=0,\end{eqnarray*}
		when $p_{k_1},p_{k_2},p_{k_3}$ satisfy the condition \eqref{000exp011113}.
	\end{enumerate}
\end{coro}
As stated in Remark \ref{varepsilon1}, if some of the exponents $p_{k_j}=p_{c}(m,\theta)$ for some $j=1,2,3$, then we can choose the parameters $g_{k_j}$ describing the loss of decay as $g_{k_j}=\varepsilon_1$ with a sufficiently small constant $\varepsilon_1>0$.
\end{rem}

\begin{rem} If $\theta\in\left[0,1/2\right)$ we proved in Section \ref{0,1/2GESDS} the global (in time) existence of energy solutions with small data belonging to $\ml{D}_{1,1}^s$ or $\ml{D}_{3/2,1}^s$. One may also consider the global (in time) existence of solutions with $\big(U_0^{(k)},U_1^{(k)}\big)\in\ml{D}_{m,1}^s$ for $m\in[1,2)$, $s\geq0$ or $\big(U_0^{(k)},U_1^{(k)}\big)\in\ml{D}_{m,2}^s$ for $m\in\left[1,6/5\right)$, $s\geq0$ by using the energy estimates to the linear model \eqref{linearproblem} (cf. with Theorems \ref{enee} and \ref{additionaldecay}, respectively).
\end{rem}
\begin{rem} In Section \ref{ge}, we proved some results for the global (in time) existence of small data solutions to some semi-linear models with exponents $p_1,p_2,p_3$ satisfying some conditions. Up to now, we did not prove any optimality of the exponent for the global (in time) existence of small data solutions. But we except that the following exponents and parameters are critical to the semi-linear model \eqref{fsdew001} with structural damping $(-\Delta)^{1/2}U_t$:
	\begin{equation*}
	p_c(1,1/2)=2,
	\end{equation*}
	\begin{equation*}
	\alpha_{\max}(1,1/2)=\max\left\{\alpha_1(1,1/2);\alpha_2(1,1/2)\right\}=3/2,
	\end{equation*}
	\begin{equation*}
	\widetilde{\alpha}_{\max}(1,1/2)=\max\left\{\widetilde{\alpha}_1(1,1/2);\widetilde{\alpha}_2(1,1/2);\widetilde{\alpha}_3(1,1/2)\right\}=3/2,
	\end{equation*}
	 The main reason is that the exponent $p_c(1,1/2)$ corresponds to the critical exponent to the semi-linear structurally damped wave equation \eqref{semistructuraldampedwave}. The recent paper \cite{D'abbicco2015} proved a global (in time) existence result when $\alpha_{\max}(1,1/2)>3/2$ and a blow up result when $\alpha_{\max}(1,1/2)<3/2$. Thus, we conjecture that the parameter $\alpha_{\max}(1,1/2)=3/2$ is critical if only one exponent is below or equal to the exponent $p_c(1,1/2)$. In addition, as mentioned at the beginning of Section \ref{ge}, the parameter $\widetilde{\alpha}_{\max}(1,1/2)$ is generalized from the parameter $\alpha_{\max}(1,1/2)$. Therefore, we also conjecture that the parameter $\widetilde{\alpha}_{\max}(1,1/2)=3/2$ is critical when two exponents are below or equal to the exponent $p_c(1,1/2)$.
\end{rem}
\begin{rem}\label{remka}
	Let us turn to weakly coupled systems of semi-linear elastic waves with \emph{structural damping of Kelvin-Voigt type} \cite{WuChaiLi2017} $\left(\theta\in(0,1]\right)$ in three-dimensions, that is, to the model
		\begin{equation}\label{fsdew0}
	\left\{
	\begin{aligned}
	&U_{tt}-a^2\Delta U-\big(b^2-a^2\big)\nabla\divv U+\big(-a^2\Delta-\big(b^2-a^2\big)\nabla\divv\big)^{\theta}U_t=F(U),\quad &(t,x)\in(0,\infty)\times\mb{R}^3,\\
	&(U,U_t)(0,x)=(U_0,U_1)(x),\quad &x\in\mb{R}^3,
	\end{aligned}
	\right.
	\end{equation}
	where $b^2>a^2>0$, $\theta\in(0,1]$. In particular, if $\theta=1$, then we have the weakly coupled system of semi-linear elastic waves \eqref{fsdew004} with \emph{viscoelastic damping of Kelvin-Voigt type} \cite{MunozRiveraRacke2017}. Moreover, the non-linear term is defined by $F(U)=\big(|U^{(3)}|^{p_1},|U^{(1)}|^{p_2},|U^{(2)}|^{p_3}\big)^{\mathrm{T}}$. Our starting point is to study the corresponding Cauchy problem for the linear elastic waves with structural damping of Kelvin-Voigt type
	\begin{equation}\label{fsdew004}
	\left\{
	\begin{aligned}
	&u_{tt}-a^2\Delta u-\big(b^2-a^2\big)\nabla\divv u+\big(-a^2\Delta-\big(b^2-a^2\big)\nabla\divv\big)^{\theta}u_t=0,\quad &(t,x)\in(0,\infty)\times\mb{R}^3,\\
	&(u,u_t)(0,x)=(u_0,u_1)(x),\quad &x\in\mb{R}^3.
	\end{aligned}
	\right.
	\end{equation}
	To study some qualitative properties of solutions we may apply the diagonalization procedure from Section \ref{EstimateforthelinearCauchyproblem}. After applying the partial Fourier transformation with respect to the $x$-variable we obtain the following linear system of ordinary differential equations depending on the parameter $\xi$:
	\begin{equation*}
	\left\{\begin{aligned}
	 &\hat{u}_{tt}+|\xi|^{2\theta}\big(a^{2\theta}I+\big(b^{2\theta}-a^{2\theta}\big)|\xi|^{-2}\xi\cdot\xi^{\mathrm{T}}\big)\hat{u}_t+|\xi|^2\big(a^{2}I+\big(b^{2}-a^{2}\big)|\xi|^{-2}\xi\cdot\xi^{\mathrm{T}}\big)\hat{u}=0,\quad &(t,\xi)\in(0,\infty)\times\mb{R}^3,\\
	&(\hat{u},\hat{u}_t)(0,\xi)=(\hat{u}_0,\hat{u}_1)(\xi),\quad&\xi\in\mb{R}^3.
	\end{aligned}\right.
	\end{equation*}
	Defining the same mirco-energy as before, we obtain the first-order system
	\begin{equation*}
	D_tW^{(0)}-\frac{i}{2}|\xi|^{2\theta}B_2W^{(0)}-|\xi|B_1W^{(0)}=0,\,\,\,\, W^{(0)}(0,\xi)=W_0^{(0)}(\xi),
	\end{equation*}
	where the matrix $B_1$ is the same as before and
	\[
	\begin{split}
	B_2=\left(
	{\begin{array}{*{20}c}
		a^{2\theta} & 0 & 0 & a^{2\theta} & 0 & 0\\
		0 & a^{2\theta} & 0 & 0 & a^{2\theta} & 0\\
		0 & 0 & b^{2\theta} & 0 & 0 & b^{2\theta}\\
		a^{2\theta} & 0 & 0 & a^{2\theta} & 0 & 0\\
		0 & a^{2\theta} & 0 & 0 & a^{2\theta} & 0\\
		0 & 0 & b^{2\theta} & 0 & 0 & b^{2\theta}\\
		\end{array}}
	\right).
	\end{split}
	\]
	We also distinguish between \emph{Case 2.0} to \emph{Case 2.3} by the influence of the parameter $|\xi|$. We observe that $B_2$ has the same structure as $B_0$ and the elements of $B_2$ satisfy $a^{2\theta}\neq b^{2\theta}$. This brings some benefits for applying the diagonalization process. We expect that the approach of this paper can be transferred without any difficulties to deal with the model \eqref{fsdew004}.
	
	From the point of view of estimates obtained by energy methods in phase space, the recent paper \cite{WuChaiLi2017} obtained almost sharp estimates for the total energy with data belonging to the space $(H^{1}\cap L^1)\times(L^2\cap L^1)$ by using multipliers in the Fourier space and the Haraux-Komornik inequality. We also expect that following our approaches in Lemma \ref{Hl+1Hllem} and Theorem \ref{Hl+1Hl} we may obtain estimates for the solutions with data being from $(H^{s+1}\cap L^m)\times(H^s\cap L^m)$ all for $m\in[1,2]$ and $s\geq0$.
	
From the above discussions, because the energy estimates for solutions to the linear model \eqref{fsdew004} are the same as the energy estimates for solutions to the linear model \eqref{linearproblem}, we also expect that the results for the global (in time) existence of small data solutions to the system \eqref{fsdew0} are the same as the derived results to the system \eqref{fsdew001} in Section 5.
\end{rem}
\begin{rem}\label{rem4} The main objectives of this paper are to show the asymptotic behavior of solutions, some estimates for linear dissipative elastic waves basing on the $L^2$ norm and the global (in time) existence for semi-linear weakly coupled systems of elastic waves with different damping mechanisms in three dimensions with power source nonlinearity $F(U)=\big(|U^{(3)}|^{p_1},|U^{(1)}|^{p_2},|U^{(2)}|^{p_3}\big)^{\mathrm{T}}$. Here, we restrict ourselves to energy estimates basing on the $L^2$ norm. In a forthcoming paper, we will develop some different $L^m$-$L^{q}$ estimates not necessarily on the conjugate line to the linear model \eqref{linearproblem} . Then, it allows us to study the global (in time) existence of small data Sobolev solutions basing on the $L^q$ norm for the following weakly coupled systems:
	\begin{equation*}
	\left\{\begin{aligned}
	&U_{tt}-a^2\Delta U-\big(b^2-a^2\big)\nabla\divv U+(-\Delta)^{\theta}U_t=F(|D|^{\sigma}U),&(t,x)\in(0,\infty)\times\mb{R}^3,\\
	&(U,U_t)(0,x)=(U_0,U_1)(x),&x\in\mb{R}^3,
	\end{aligned}\right.
	\end{equation*}
	and
	\begin{equation*}
	\left\{\begin{aligned}
	&U_{tt}-a^2\Delta U-\big(b^2-a^2\big)\nabla\divv U+(-\Delta)^{\theta}U_t=F(\partial_tU),&(t,x)\in(0,\infty)\times\mb{R}^3,\\
	&(U,U_t)(0,x)=(U_0,U_1)(x),&x\in\mb{R}^3,
	\end{aligned}\right.
	\end{equation*}
	where the nonlinearities are described by
	\begin{equation*}
	\begin{split}
	F(|D|^{\sigma}U)&=\big(||D|^{\sigma}U^{(3)}|^{p_1},||D|^{\sigma}U^{(1)}|^{p_2},||D|^{\sigma}U^{(2)}|^{p_3}\big)^{\mathrm{T}},\\
	F(\partial_tU)&=\big(|\partial_tU^{(3)}|^{p_1},|\partial_tU^{(1)}|^{p_2},|\partial_tU^{(2)}|^{p_3}\big)^{\mathrm{T}},
	\end{split}
	\end{equation*}
	where the constant parameters $\theta\in[0,1]$ and $\sigma\in[0,1]$. We will study the influence of different damping mechanisms to several source nonlinearities of power type.
\end{rem}

\section*{Acknowledgments}
The Ph.D study of Mr. Wenhui Chen is supported by S\"achsisches Landesgraduiertenstipendium.
\appendix
\section{Some elements in the matrices in the diagonalization procedure}\label{appendix1}
When we discuss the representation of solutions in Section \ref{representationsolution}, we introduce the following parameters:
\begin{equation*}
\begin{split}
z_1(y)=\frac{iy^4|\xi|^{4-6\theta}}{1-y^2|\xi|^{2-4\theta}}=O\big(|\xi|^{4-6\theta}\big),\quad z_6(y)=\frac{y^3|\xi|^{3-4\theta}}{1-y^2|\xi|^{2-4\theta}}=O\big(|\xi|^{3-4\theta}\big),
\end{split}
\end{equation*}
where $y=a,b$. Moreover, we introduce
\begin{equation*}
\begin{split}
&z_2=\frac{1}{|\xi|^{4\theta}-z_6^2(a)}\left(i\big(a^2+b^2\big)z_6^2(a)|\xi|^{2-2\theta}+2z_1(a)z_6^2(a)+i|\xi|^{2\theta}z_6^2(a)\right),\\
&z_3=\frac{1}{|\xi|^{4\theta}-z_6^2(a)}\left(2ia^2|\xi|^{2-2\theta}z_6^2(a)+i|\xi|^{2\theta}z_6^2(a)+2z_1(a)z_6^2(a)\right),\\
&z_4=\frac{1}{|\xi|^{4\theta}-z_6^2(b)}\left(i|\xi|^{2\theta}z_6^2(b)+z_6^3(b)+z_1(b)z_6^2(b)-i\big(a^2+b^2\big)|\xi|^{2-2\theta}z_6^2(b)\right),\\
&z_5=\frac{1}{|\xi|^{4\theta}-z_6^2(b)}\left(2z_1(b)z_6^2(b)+i|\xi|^{2\theta}z_6^2(b)+i\big(a^2+b^2\big)|\xi|^{2-2\theta}z_6^2(b)\right).
\end{split}
\end{equation*}
\section{Tools from Harmonic Analysis}\label{toolfractional}
In this section we present some tools from Harmonic Analysis that are used in Section \ref{ge}.
\begin{prop}\label{claGNineq} (Classical Gagliardo-Nirenberg inequality) Let $j,m\in\mb{N}$ with $j<m$, and let $f\in \ml{C}_0^{m}\big(\mb{R}^n\big)$. Let $\beta=\beta_{j,m}\in\big[\frac{j}{m},1\big]$ with $p,q,r\in\left[1,\infty\right]$ satisfy
	\begin{equation*}
	j-\frac{n}{q}=\Big(m-\frac{n}{r}\Big)\beta-\frac{n}{p}(1-\beta).
	\end{equation*}
	Then, we have the following inequality:
	\begin{equation*}
	\|D^jf\|_{L^q}\lesssim\|f\|_{L^p}^{1-\beta}\|D^mf\|_{L^r}^{\beta},
	\end{equation*}
	provided that $\big(m-\frac{n}{r}\big)-j\notin\mb{N}$.
	If $\big(m-\frac{n}{r}\big)-j\in\mb{N}$, then the classical Gagliardo-Nirenberg inequality holds provided that $\beta\in\big[\frac{j}{m},1\big)$.
\end{prop}
The proof of the classical Gagliardo-Nirenberg inequality can be found in \cite{Friedman1976}, Part I, Theorem 9.3.
\begin{prop}\label{fractionalgagliardonirenbergineq} (Fractional Gagliardo-Nirenberg inequality)
	Let $p,p_0,p_1\in(1,\infty)$ and $\kappa\in(0,s)$ with $s>0$. Then, for all $f\in L^{p_0}\big(\mb{R}^n\big)\cap \dot{H}^{s}_{p_1}\big(\mb{R}^n\big)$ the following inequality holds:
	\begin{equation*}
	\|f\|_{\dot{H}^{\kappa}_{p}}\lesssim\|f\|_{L^{p_0}}^{1-\beta}\|f\|^{\beta}_{\dot{H}^{s}_{p_1}},
	\end{equation*}
	where $\beta=\beta_{\kappa,s}=\big(\frac{1}{p_0}-\frac{1}{p}+\frac{\kappa}{n}\big)\big/\big(\frac{1}{p_0}-\frac{1}{p_1}+\frac{s}{n}\big)$ and $\beta\in\big[\frac{\kappa}{s},1\big]$.
\end{prop}
The proof of this result can be found in \cite{Hajaiej}.
\begin{prop}\label{fractionleibnizrule} (Fractional Leibniz rule)
	Let $s>0$ and $1\leq r\leq\infty$, $1<p_1,p_2,q_1,q_2\leq\infty$ satisfy the relation
	\begin{equation*}
	\frac{1}{r}=\frac{1}{p_1}+\frac{1}{p_2}=\frac{1}{q_1}+\frac{1}{q_2}.
	\end{equation*}
	Then, for all $f\in\dot{H}^{s}_{p_1}\big(\mb{R}^n\big)\cap L^{q_1}\big(\mb{R}^n\big)$ and $g\in\dot{H}^{s}_{q_2}\big(\mb{R}^n\big)\cap L^{q_2}\big(\mb{R}^n\big)$
the following inequality holds:
	\begin{equation*}
	\|fg\|_{\dot{H}^{s}_{r}}\lesssim \|f\|_{\dot{H}^{s}_{p_1}}\|g\|_{L^{p_2}}+\|f\|_{L^{q_1}}\|g\|_{\dot{H}^{s}_{q_2}}.
	\end{equation*}
\end{prop}
The proof of this inequality can be found in \cite{Grafakos2014}.
\begin{prop}\label{fractionalchainrule} (Fractional chain rule)
	Let $s>0$, $p>\lceil s\rceil$ and  $1<r,r_1,r_2<\infty$ satisfy the relation
	\begin{equation*}
	\frac{1}{r}=\frac{p-1}{r_1}+\frac{1}{r_2}.
	\end{equation*}
	Then, for all $f\in\dot{H}^{s}_{r_2}\big(\mb{R}^n\big)\cap L^{r_1}\big(\mb{R}^n\big)$ the following inequality holds:
	\begin{equation*}
	\|\pm f|f|^{p-1}\|_{\dot{H}^{s}_{r}}+\||f|^p\|_{\dot{H}^{s}_{r}}\lesssim\|f\|_{L^{r_1}}^{p-1}\|f\|_{\dot{H}^{s}_{r_2}}.
	\end{equation*}
\end{prop}
One can find the proof in \cite{PalmieriReissig2018}.
\begin{prop}\label{fractionalpowersrule} (Fractional powers rules)
	Let $r\in(1,\infty)$, $p>1$ and $s\in(0,p)$. Then, for all $f\in{\dot{H}^{s}_{r}}\big(\mb{R}^n\big)\cap L^{\infty}\big(\mb{R}^n\big)$ the following inequality holds:
	\begin{equation*}
	\|\pm f|f|^{p-1}\|_{\dot{H}^{s}_{r}}+\||f|^p\|_{\dot{H}^{s}_{r}}\lesssim\|f\|_{\dot{H}^{s}_{r}}\|f\|_{L^{\infty}}^{p-1}.
	\end{equation*}
\end{prop}
\begin{prop}\label{coroleibunizpower} Let $r\in(1,\infty)$ and $s>0$. Then, for all $f,g\in\dot{H}^{s}_{r}\big(\mb{R}^n\big)\cap L^{\infty}\big(\mb{R}^n\big)$ the following inequality holds:
	\begin{equation*}
	\|fg\|_{\dot{H}^{s}_{r}}\lesssim\|f\|_{\dot{H}^{s}_{r}}\|g\|_{L^{\infty}}+\|f\|_{L^{\infty}}\|g\|_{\dot{H}^{s}_{r}}.
	\end{equation*}
\end{prop}
The above two propositions with their proofs can be found in \cite{RunstSickel1996}.
\begin{prop}\label{supplement} Let $0<2s^*<n<2s$. Then, for any function $f\in\dot{H}^{s^*}\big(\mb{R}^n\big)\cap\dot{H}^s\big(\mb{R}^n\big)$ one has
	\begin{equation*}
	\|f\|_{L^{\infty}}\leq\|f\|_{\dot{H}^{s^*}}+\|f\|_{\dot{H}^s}.
	\end{equation*}
\end{prop}
The proof of this statement was given in \cite{DabbiccoEbertLucente}.

\section{Existence and restriction of parameters}\label{appendix2}
The purpose of this section is to clarify the possibility to choose the parameters $q_1,q_2$, $r_1,\dots,r_6$ in the proof of Theorem \ref{HlL1nonlinearythm}. In other words, we show that the condition \eqref{rest01} below is not only a sufficient condition but also a necessary condition for a suitable choice of these parameters.

First of all, in Theorem \ref{HlL1nonlinearythm} the restrictions for the parameters $q_1$ and $q_2$ are as follows: \begin{eqnarray*} && \frac{3}{s+1}\Big(\frac{1}{2}-\frac{1}{q_1}\Big)\in[0,1],\,\,\,\,  \frac{3}{s+1}\Big(\frac{1}{2}-\frac{1}{q_2}+\frac{s}{3}\Big)\in\Big[\frac{s}{s+1},1\Big]\,\,\,\,\mbox{and}\,\,\,\, \frac{p_k-1}{q_1}+\frac{1}{q_2}=\frac{1}{2}\,\,\,\,\mbox{ with}\,\,\,\, q_1,q_2\neq \infty.\end{eqnarray*}
In other words, we have
\begin{equation*}
\begin{aligned}
\frac{(1-2s)(p_k-1)+1}{6}\leq\frac{p_k-1}{q_1}+\frac{1}{q_2}&\leq\frac{p_k}{2}&\text{if }\,&0<s<1/2,\\
\frac{1}{6}\leq\frac{p_k-1}{q_1}+\frac{1}{q_2}&\leq\frac{p_k}{2}&\text{if }\,&1/2\leq s.
\end{aligned}
\end{equation*}
Hence, the assumption
\begin{equation}\label{rest01}
p_k\leq1+\frac{2}{1-2s}\quad\text{if }\,0<s<1/2
\end{equation}
for $p_k$ for all $k=1,2,3,$ leads to \[ \frac{1}{2}\in\Big[\frac{(1-2s)(p_k-1)+1}{6},\frac{p_k}{2}\Big] \,\,\,\,\mbox{and}\,\,\,\, \frac{1}{2}\in\Big[\frac{1}{6},\frac{p_k}{2}\Big].\] All in all, \eqref{rest01} is a necessary condition to choose suitable parameters $q_1$ and $q_2$.

Next, we want to show that the condition \eqref{rest01} is a sufficient condition for choosing $q_1,q_2$. From the relationship \[ \frac{1}{q_2}=\frac{1}{2}-\frac{p_k-1}{q_1}\,\,\,\,\mbox{and}\,\,\,\, \frac{3}{s+1}\Big(\frac{1}{2}-\frac{1}{q_1}\Big)\in[0,1],\] we may conclude
\begin{equation}\label{rest02}
\begin{aligned}
\frac{1}{q_2}&\in\Big[1-\frac{p_k}{2},\frac{1}{2}-\frac{(1-2s)(p_k-1)}{6}\Big]&\text{if }\,&0<s<1/2,\\
\frac{1}{q_2}&\in\Big[1-\frac{p_k}{2},\frac{1}{2}\Big)&\text{if }\,&1/2\leq s.
\end{aligned}
\end{equation}
We should point out that the interval for $\frac{1}{q_2}$ is not empty because of $p_k>1$. Taking account of $\frac{3}{s+1}\big(\frac{1}{2}-\frac{1}{q_2}+\frac{s}{3}\big)\in\big[\frac{s}{s+1},1\big]$ and \eqref{rest02} together with assumption \eqref{rest01}, we observe
\begin{equation*}
\begin{aligned}
\Big[\frac{1}{6},\frac{1}{2}\Big]&\cap\Big[1-\frac{p_k}{2},\frac{1}{2}-\frac{(1-2s)(p_k-1)}{6}\Big]\neq\emptyset&\text{if }\,&0<s<1/2,\\
\Big(0,\frac{1}{2}\Big]&\cap\Big[1-\frac{p_k}{2},\frac{1}{2}\Big)\neq\emptyset&\text{if }\,&1/2\leq s.
\end{aligned}
\end{equation*}
In conclusion, there exist suitable parameters $q_1,q_2$ in the proof of Theorem \ref{HlL1nonlinearythm}.

Moreover, the restrictions on $r_1,r_2$ are \[ \frac{3}{s+1}\Big(\frac{1}{2}-\frac{1}{r_1}+\frac{s}{3}\Big)\in\Big[\frac{s}{s+1},1\Big],\,\,\,\, \frac{3}{s+1}\Big(\frac{1}{2}-\frac{1}{r_2(p_k-1)}\Big)\in[0,1],\,\,\,\,\mbox{and}\,\,\,\, \frac{1}{r_1}+\frac{1}{r_2}=\frac{1}{2}.\] By the same arguments as above, the condition \eqref{rest01} is a sufficient and necessary condition for the choice of suitable parameters $r_1$ and $r_2$.

Lastly, the restrictions on $r_3,\dots,r_6$ are
\begin{eqnarray*} && \frac{3}{s+1}\Big(\frac{1}{2}-\frac{1}{r_3}\Big)\in[0,1],\,\,\,\, \frac{3}{s+1}\Big(\frac{1}{2}-\frac{1}{r_5}\Big)\in[0,1],\,\,\,\,  \frac{3}{s+1}\Big(\frac{1}{2}-\frac{1}{r_6}+\frac{s}{3}\Big)\in\Big[\frac{s}{s+1},1\Big],\\
&& \qquad \frac{1}{r_3}+\frac{1}{r_4}=\frac{1}{2}\,\,\,\,\mbox{and}\,\,\,\, \frac{1}{r_4}=\frac{p_k-2}{r_5}+\frac{1}{r_6}.\end{eqnarray*} As in the paper \cite{PalmieriReissig2018} we also prove the optimality of the condition \eqref{rest01} for $p_k$.\\
One choice for the parameters $q_1,q_2$ and $r_1,\dots,r_6$ is the following:
\begin{equation*}
\begin{split}
&q_1=3(p_k-1),\quad q_2=6,\quad r_1=6,\quad r_2=3,\quad r_3=3(p_k-1),\quad r_4=\frac{6(p_k-1)}{3(p_k-1)-2},\quad r_5=3(p_k-1),\quad r_6=6.
\end{split}
\end{equation*}
This choice implies the condition for $k=1,2,3,$
\begin{equation*}
1+\frac{2}{3}\leq p_k\leq 1+\frac{2}{1-2s}\,\,\,\,\text{if}\,\,\,\,0<s<1/2.
\end{equation*}

\nocite{*}

\clearpage

\end{document}